\documentclass[final, reqno, 11pt]{amsart}
\usepackage{amsmath}
\usepackage{amsfonts}
\usepackage{amssymb}
\usepackage{amsthm}
\usepackage{bbold}
\usepackage{mathabx}
\usepackage[margin=1in]{geometry}
\usepackage[colorlinks=true, allcolors=blue]{hyperref}

\newtheorem{theorem}{Theorem}[section]
\newtheorem{lemma}[theorem]{Lemma}
\newtheorem{proposition}[theorem]{Proposition}
\newtheorem{corollary}[theorem]{Corollary}
\newtheorem{assumption}{Assumption}
\theoremstyle{remark}
\newtheorem{observation}[theorem]{Remark}

\newtheorem{definition}[theorem]{Definition}
\newcommand{\set}{\mathbb}
\newcommand{\R}{\mathbb R}
\newcommand{\C}{\mathbb C}
\newcommand{\E}{\set E}

\newcommand{\3}{d}

\newcommand{\B}{\mc B}

\newcommand{\Z}{\mathbb Z}

\renewcommand{\P}{\mathbb P}

\newcommand{\dl}{\nabla}
\renewcommand{\frak}{\mathfrak}

\newcommand{\les}{\lesssim}
\newcommand{\mc}{\mathcal}
\newcommand{\be}{\begin{equation}}
\newcommand{\ee}{\end{equation}}
\newcommand{\bee}{\begin{align}}
\newcommand{\eee}{\end{align}}

\newcommand{\ba}{\begin{array}}

\newcommand{\ea}{\end{array}}
\newcommand{\bpm}{\begin{pmatrix}}
\newcommand{\epm}{\end{pmatrix}}
\newcommand{\lb}{\label}
\DeclareMathOperator{\sgn}{sgn}
\DeclareMathOperator{\tr}{tr}

\DeclareMathOperator{\supp}{supp}

\DeclareMathOperator{\Imim}{Im}

\newcommand{\ov}{\overline}
\newcommand{\dd}{{\,}{d}}
\newcommand{\W}{\mc W}
\newcommand{\V}{\mc V}
\renewcommand{\v}{\mathbb v}

\renewcommand{\Im}{\Imim}
\newcommand{\one}{\mathbb 1}
\subjclass[2000]{35Q41, 60J65, 35J10, 35Q40}

\title[Schr\"{o}dinger Equations with Random Potentials]{A Semilinear Schr\"{o}dinger Equation with Random Potential}
\author{Marius Beceanu}
\address{110 Earth Science, University at Albany SUNY Mathematics and Statistics Department, Albany, NY 12222, USA}
\email{mbeceanu@albany.edu}
\author{Avy Soffer}
\address{110 Frelinghuysen Rd., Rutgers Math.\ Dept., Piscataway, NJ 08854, USA}
\email{soffer@math.rutgers.edu}

\begin{document}
\maketitle
\numberwithin{equation}{section}
\begin{abstract}
We study a non-linear Schr\"{o}dinger equation with a Hartree-type nonlinearity and a localized random time-dependent external potential. Sharp dispersive estimates for the linear Schr\"{o}dinger equation with a random time-dependent potential enable us to also treat the case of small semi-linear perturbations.\\
In both the linear and the nonlinear instances, we prove that, on average, energy remains bounded and solutions scatter. 
\end{abstract}

\setcounter{tocdepth}{3}
\setcounter{secnumdepth}{3}
\tableofcontents

\section{Introduction}
\subsection{Statement of the problem}


%
%
%

The purpose of this paper is to understand the interaction between random time-dependent external potentials and non-linear terms in the behavior of solutions of the following nonlinear Schr\"{o}din\-ger equation on $\R^\3$, where $\3 \geq 3$:
\be\lb{random_nonlin}
i (\psi_\omega)_t -\Delta \psi_\omega +V_\omega(x, t) \psi_\omega + \epsilon (\chi\ast |\psi_\omega|^2) \psi_\omega = 0,\ \psi_\omega(0)= \psi_0(x, \omega(0)) \in L^2(\R^\3 \times Y).
\ee
Here $Y$ is some \textbf{compact} Hausdorff topological space, $\omega(0)$ is a random or specific point in $Y$, and 
$\Vert \psi_{\omega} \Vert_{2} =1$. The function
$V_{\omega}(x,t)$ is a random time-dependent potential specified more precisely below, $\epsilon\geq 0$ is a coupling constant, and the ``two-body potential'' $\chi(x)$ is a function on 
$\mathbb{R}^{\3}$ whose properties will be specified below.

\textbf{Assumptions:} The \textbf{random time-dependent potential} is a real-valued function of $x$ that also depends on a random time-dependent parameter, as follows. Let $V:\R^\3 \times Y \to \R$, $V=V(x, y)$, be a family of real-valued scalar potentials on $\R^\3$, indexed by a parameter $y \in Y$. The parameter space $Y$ is taken to be a compact topological space and a measure space. Then we set
\be\lb{def_v}
V_\omega(x, t) := V(x, \omega(t)),
\ee
where $\omega: [0, \infty) \to Y$ is a continuous random path through the parameter space $Y$, parametrized by time $t$.

The \textbf{random path} $\omega$ is chosen as follows. Let $\Omega=C([0, \infty); Y)$ be the set of continuous $Y$-valued paths $\omega: [0, \infty) \to Y$ and let $(\Omega, \mc F, \set P)$ be a measure space on $\Omega$, with $\P:\mc F \to [0, 1]$ a probability measure defined on the $\sigma$-algebra $\mc F$.

The probability measure $\P$ on $\Omega$ determines a continuous $Y$-valued process $(X_t)_{t \geq 0}: \Omega \to Y$ such that each path $\omega \in \Omega$ is an instance of this process, $X_t(\omega) = \omega(t)$.

\noindent\textbf{Further assumptions:}
1. $(X_t)_{t \geq 1}$ is a \textbf{Markov process}, meaning that
$$
\E\{X_t \mid \mc F_s\} = X_s \text{ for any } 0 \leq s \leq t,
$$
where the definition of $\mc F_s$ is given below. Here $\E$ is the expectation with respect to the probability measure $\P$.

For each $t \geq 0$, the distribution of $X_t$ is a probability measure on $Y$ determined by $\P$.

2. The Markov process $(X_t)_{t \geq 0}$ is stationary. Then $(X_t)_{t \geq 0}$ generates a contraction semigroup $(e^{-tA})_{t \geq 0}$ on $L^p(Y)$, $1 \leq p \leq \infty$, with infinitesimal generator $-A$, defined by
$$
[e^{-tA} f](y) = \E \big\{f(X_0) \mid X_t=y\big\},
$$
where $\E$ denotes expectation with respect to $\omega$. It has a non-negative integral kernel given by
\be\lb{eta}
e^{-tA}(y_1, y_2) = \P(X_t=y_1 \mid X_0=y_2).
\ee

We assume that $A \geq 0$ is self-adjoint, hence $-A$ is \textbf{dissipative}, and require the following conditions related to the contraction semigroup $e^{-tA}$:\\

\noindent C1. For each $t>0$, $e^{-tA}$ is $L^p$-bounded for $1 \leq p \leq \infty$, of uniformly bounded operator norm, and its integral kernel is in $L^\infty_{y_2} L^1_{y_1} \cap L^\infty_{y_1} L^1_{y_2}$.

\noindent C2. The integral kernel $e^{-tA}(y_1, y_2)$ is norm-continuous in $L^\infty_{y_2} L^1_{y_1} \cap L^\infty_{y_1} L^1_{y_2}$ as a function of $t$, for $t>0$, and converges in norm to some limit as $t \to \infty$.

\noindent C3. For all $t>0$, $e^{-tA} \in \B(L^1, L^\infty)$.

\noindent C4. For each $t>0$, $e^{-tA}(y_1, y_2)$ can be approximated by the tensor product of bounded, compactly supported functions in $L^\infty_{y_2} L^1_{y_1} \cap L^\infty_{y_1} L^1_{y_2}$.

\noindent C5. The zero-energy eigenspace of $A$ (eigenspace corresponding to the eigenvalue $0$) has dimension one and the \textbf{ground state} (zero-energy eigenfunction) $h(y) > 0$ is positive and integrable, $h \in L^1_y$.\\




An important example is $(X_t)_{t \geq 0}$ being Brownian motion on a Riemannian manifold $Y$, with or without boundary, such as a Lie group of symmetries or some subset of it.


3. The family of potentials $V$ has the property that $V(x, y) \in C_y (L^1_x \cap L^\infty_x)$: each sample $V(\cdot, y)$ is in $L^1_x \cap L^\infty_x$, they are uniformly bounded in norm, and the dependence of $V(\cdot, y)$ on the parameter $y$ is continuous in the $L^1_x \cap L^\infty_x$ norm.



Examples of interest include:\\
$\bullet$ $V(x, y) = V(x-y)$, i.e.\;$V(x,y)$ is the translate of a fixed function on $\mathbb{R}^{\3}$ by a random vector $y \in \mathbb{R}^{\3}$, or is subject to some other randomly chosen symmetry transformation of $\mathbb{R}^\3$; \\
$\bullet$ $V(x, y) = V_1(x) + yV_2(x)$: random changes in amplitude on top of some fixed potential; \\
and combinations of these examples. But our setup is yet more general than that.

4. If the family of potentials $V(\cdot, y)$ depends on the random parameter $y$ in a trivial manner, then the randomness makes no difference and one cannot expect results to differ from those in the deterministic case:\\
$\bullet$ If $V(x, y) = V(x)$ then the equation is deterministic, not random.\\
$\bullet$ If $V(x, y) = V(x) + f(y)$ then one can get rid of the randomness by conjugating the equation by
$$
e^{\textstyle i\int_0^t f(\omega(\tau)) \dd \tau}.
$$
Then the solutions' size will be deterministic, only the phase of $\psi_{\omega}$ will be random.

The above are the only two cases in which the effect of the randomness is trivial. In order to exclude them, we impose the following non-triviality Assumption \ref{nontrivial}:

\begin{assumption}[Nontrivial randomness]\lb{nontrivial} Let $h(y)$ be the zero energy eigenstate of $A$.

There exists an open set $\mc O \subset \R^\3$ such that for almost every $x \in \mc O$ there exist $y_1$, $y_2 \in \supp h$ such that $V(x, y_1) \ne V(x, y_2)$.

There exists an open set $\mc O' \subset \R^\3$ such that, for almost every $(x_1, x_2) \in \mc O' \times \R^\3$, there exist $y_1$, $y_2 \in \supp h$ such that $V(x_1, y_1) - V(x_2, y_1) \ne V(x_1, y_2) - V(x_2, y_2)$.
\end{assumption}

This assumption precludes the two problematic cases. As we shall see, Assumption \ref{nontrivial} indeed suffices to prove dispersion.

5. The \textbf{nonlinear potential} $\epsilon \chi \ast |\psi|^2$ is a Hartree-type potential. The convolution kernel $\chi \in L^{\3/2,\infty}$ (weak $L^{\3/2}$) is even and $\epsilon<<1$ is a small coupling constant.

\subsection{Main results} For equation (\ref{random_nonlin}), the \textbf{energy} of a solution $\psi_\omega(x, t)$ is defined by
$$
E_\omega[\psi_\omega](t) = \frac 1 2 \int_{\R^\3} |\dl \psi_\omega|^2 + V_\omega(x, t) |\psi_\omega|^2 \dd x + \frac \epsilon 4 \int_{\R^\3\times \R^\3} \chi(x-y) |\psi_\omega(x)|^2 |\psi_\omega(y)|^2 \dd x \dd y.
$$
Because the random potential is time-dependent, energy is not constant. At best, one can prove that it is conserved up to a constant factor, see Theorem \ref{main_nonlinear_result}.

Our main nonlinear result is the following:
\begin{theorem}[Main result]\lb{main_nonlinear_result} For each initial data $\psi_0 \in L^2$ and $|\epsilon| < \epsilon_0(\|\psi_0\|_{L^2})$, equation (\ref{random_nonlin}),
$$
i (\psi_\omega)_t -\Delta \psi_\omega +V_\omega(x, t) \psi_\omega + \epsilon (\chi\ast |\psi_\omega|^2) \psi_\omega = 0,\ \psi_\omega(0)= \psi_0(x, \omega(0)) \in L^2(\R^\3 \times Y),
$$
has, with probability one, a global-in-time unique mild solution $\psi_\omega$, in the sense of semigroups:
$$
\psi_\omega(t) = e^{-it\Delta} \psi_\omega(0) + i \int_0^t e^{-i(t-s)\Delta} [V_\omega(x, s) \psi_\omega(s) + \epsilon (\chi \ast |\psi_\omega(s)|^2) \psi_\omega(s)] \dd s.
$$
The solution depends Lipschitz-continuously on the initial data, for $t \geq 0$, in the $L^2_\omega L^2_t L^{6, 2}_x$ Strichartz norm:
$$
\E \big\{\|\psi_\omega\|_{L^2_t L^{6, 2}_x}^2\big\} \les \|\psi_0\|_{L^2_y L^2_x}^2.
$$
Also, due to the unitarity of the evolution, along each path $\omega \in \Omega$ $\|\psi_\omega\|_{L^\infty_t L^2_x} = \|\psi_0(x, \omega(0))\|_{L^2_x}$.

If in addition $A [hV] \in L^\infty_y L^{\3/2, \infty}_x$ and $h^{1/2} \nabla \psi_0 \in L^2_\omega L^2_x$, where $h$ is the ground state of $A$, then the energy remains bounded on average:
$$
\E \big\{h(\omega(t)) E_\omega[\psi_\omega](t)\big\} \les \|h^{1/2} \dl \psi_0\|_{L^2_y L^2_x}^2 + \|\psi_0\|_{L^2_y L^2_x}^2.
$$
Moreover, the average of the energy converges as $t \to \infty$:
$$
\lim_{t \to \infty} \E \big\{h(\omega(t)) E_{\omega}[\psi_\omega](t)\big\} \text{ exists}.
$$
\end{theorem}

The random initial data have to be adapted to the random process, so $\psi_0(x, \omega)$ can only depend on $\omega(0) \in Y$ and $\|\psi_0(x, \omega)\|_{L^2(\R^\3 \times \Omega)} = \|\psi_0(x, \omega(0))\|_{L^2(\R^\3 \times Y)}$. There are no other restrictions on $L^2$ initial data. It could be a non-uniformly-bounded ensemble of initial data --- or the same for all random trajectories.

The initial distribution of $\omega$ can be a Dirac measure, i.e.\;for some $y_0 \in Y$ $\set P(\omega(0) = y_0) = 1$, so $\omega(0) = y_0$ for almost all $\omega \in \Omega$. In this case, the initial condition reduces to $\psi_0 \in L^2_x$.

The proof of Theorem \ref{main_nonlinear_result} relies on essentially optimal linear estimates described below. Consider the inhomogeneous linear equation with random time-dependent potential
\be\lb{random_lin}
i \partial_t \psi_\omega - \Delta \psi_\omega + V_\omega \psi_\omega = \Psi_\omega(x, t),\ \psi_\omega(0):= \psi_0(x, \omega(0)) \in L^2(\R^\3 \times Y),
\ee
where the random time-dependent potential $V_\omega$ is as in (\ref{def_v}).

For this linear equation, the energy of a solution is defined as
$$
E_{\omega\,lin}[\psi_\omega](t) = \frac 1 2 \int_{\R^\3} |\dl \psi_\omega|^2 + V_\omega(x, t) |\psi_\omega|^2 \dd x.
$$
As for the non-linear equation, this energy need not be constant --- even if $\Psi=0$ --- because the potential is time-dependent.

Let $\mc F_t \subset \mc P(\Omega)$ be the least $\sigma$-algebra with respect to which $X_s$ is measurable for every $s \in [0, t]$. Then $(\mc F_t)_{t \geq 0}$ is an increasing filtration of $\mc F$. A process $(F_t(\omega))_{t \geq 0}$ is said to be \textbf{adapted} to this filtration if $F_t$ is $\mc F_t$-measurable for each $t \geq 0$.

Our main result on the linear equation is then the following theorem:
\begin{theorem}\lb{main_result} Suppose that $V \in C_y(L^1_x \cap L^\infty_x)$ is real-valued and satisfies Assumption \ref{nontrivial} and that $\Psi_\omega(t)$ is adapted to $(\mc F_t)_{t \geq 0}$. Then the solution $\psi_\omega$ of (\ref{random_lin}) fulfills the estimates
\be\lb{est_lin}\begin{aligned}
\E \big\{\|\psi_\omega\|_{L^2_t L^{6, 2}_x}^2\big\} &\les \|\psi_0\|_{L^2_y L^2_x}^2 + \E\|\Psi_\omega\|_{L^2_t L^{6/5, 2}_x}^2
\end{aligned}\ee
and, for all $t \geq 0$,
\be\lb{est_\32}\begin{aligned}
\E \big\{\|\psi_\omega(t)\|_{L^2_x + L^\infty_x}^2\big\} &\les \langle t \rangle^{-\3} \big[ \|\psi_0\|_{L^2_y (L^1_x \cap L^2_x)}^2 + \sup_{s \geq 0} \langle s \rangle^{\3} \E \big\{\|\Psi_\omega(s)\|_{L^1_x \cap L^2_x}^2\big\} \big].
\end{aligned}\ee
Moreover, if $h(y)$ is the zero-energy eigenstate of $A$ and if $A [h(y) V(x, y)] \in L^\infty_y L^{\3/2, \infty}_x$, then energy remains bounded on average and its average converges as $t \to \infty$:
$$
\E \big\{h(\omega(t)) E_{\omega\,lin}[\psi_\omega](t)\big\} \les \|h^{1/2} \dl \psi_0\|_{L^2_{y, x}}^2 + \|\psi_0\|_{L^2_{y,x}}^2,\ \lim_{t \to \infty} \E \big\{h(\omega(t)) E_{\omega\,lin}[\psi_\omega](t)\big\} \text{ exists}.
$$
\end{theorem}

As a consequence of Strichartz estimates we observe that the time-dependent randomness of $V_{\omega}$ makes all possible bound states of the operators $-\Delta+V(\cdot,y)$, for any $y\in Y$, decay along almost all random trajectories $\omega$, when the randomness is nontrivial in the sense of Assumption \ref{nontrivial}.


For easier comparison with \cite{FLLS}, see below, we also state a straightforward generalization of this result to trace-class operators. A rank-one initial condition, $P_\omega := \vert \psi_\omega\rangle \langle \psi_\omega \vert$, can be replaced by a density matrix, i.e.\;a positive trace-class operator of trace $1$. Then Theorem \ref{main_result} implies a trace-class dispersive bound for a Liouville-type equation with a random time-dependent Hamiltonian.

The Schatten--von Neumann classes are denoted by $\frak S_\alpha$, $1 \leq \alpha< \infty$, with $\frak S_1$ the trace class, $\frak S_2$ the Hilbert--Schmidt class, and $\frak S_\infty$ the space of compact operators. Using a singular-value decomposition, any normal operator $\gamma \in \frak S_\alpha$ can be written as
$$
\gamma = \sum_j n_j u_j \otimes \ov u_j,
$$
where $(n_j)_j \in \ell^\alpha$ are the eigenvalues of $\gamma$, and the functions $u_j(x)$ form an orthonormal basis in $L^2$. The \textbf{density} of $\gamma$ is defined by
\be\lb{density}
\rho_{\gamma}(x) = \sum_j n_j |u_j(x)|^2.
\ee

\begin{corollary}\lb{cor_liouville} Consider the Liouville-type equation
\be\lb{random_liouville}
i \partial_t f_\omega = [-\Delta+V_\omega(t), f_\omega] + iF_\omega, f_\omega(0)=f_0(\omega(0)) \in L^1_y \B(L^2_x),
\ee
where $f_\omega$ and $F_\omega$ are random time-dependent families of self-adjoint trace-class operators, and $F_\omega$ is adapted to $(\mathcal{F}_t)_{t\geq 0}$. Then, under the hypotheses of Theorem \ref{main_result},
$$
\big\|\E \rho_{f(t)}\big\|_{L^1_t L^{\3, 1}_x} \les \E \|\rho_{f_0}\|_{L^1_x} + \int_\R \|\E  \rho_{F(t)}\|_{L^1_x} \dd t = \|\E f_0\|_{\frak S_1} + \int_\R \|\E F(t)\|_{\frak S_1} \dd t.
$$
\end{corollary}

Related results for Hilbert--Schmidt operators and more general Schatten--von Neumann classes will be presented in another paper.

\subsection{Motivation and history of the problem}

The stochastic differential equation (\ref{random_nonlin}) is used to model and describe various physical phenomena, such as the following ones:

I. Equations such as Eq. \eqref{random_nonlin} describe the mean-field limit of systems of $N$ interacting bosons interacting through two-body potentials $\epsilon\, \chi$ and subject to a time-dependent external potential $V_\omega$; (number $N$ of particles going to infinity, coupling constant $\epsilon$ going to zero, with $N\cdot g$ kept constant).

II. Consider a heavy molecule with internal structure coupled to a noisy environment. The electrons in its outer shells may then satisfy the linear version ($\varepsilon=0$) of equation \eqref{random_nonlin}.

If the molecule interacts with a reactive material/medium then the effective dynamics of the electrons is influenced by the material, and the non-linear term in the NLS \eqref{random_nonlin} may represent a natural way of capturing the back-reaction of the material.
This model allows some understanding of the effect of the material/medium on the life-time of the molecule, unstable because of the noisy environment.

III. The non-linear Schr\"{o}dinger equation with random potential $V_\omega$ is sometimes used in nonlinear optics: The noise in one-dimensional nonlinear optical fibres can be ``time-dependent'' because of impurities, with the spatial $z$-coordinate now playing the role of ``time'' in the NLS \eqref{random_nonlin}.

IV. From a purely theoretical perspective, it is of interest to estimate the effect of a non-linearity on the lifetime of unstable states, namely bound states of Schr\"{o}dinger operators rendered unstable by time-dependent noise.

An early study of stochastic nonlinear evolution equations belongs to Wadati \cite{wadati}, who proved that solitons of the completely integrable Korteweg--de Vries equation are preserved, up to isometries, after adding external Gaussian random noise to the equation.

The linear version of our problem, equation (\ref{random_lin}), 
was first studied by Pillet in \cite{pillet2}, \cite{pillet}. \cite{pillet2} established a general framework, including the Feynman--Kac-type formula (\ref{feynman}) and several other identities we use throughout the paper. \cite{pillet} showed that wave operators are unitary with probability one, if the random perturbation essentially lives in a compact region and has a certain amount of smoothness.

Cheremshantsev \cite{cherem}, \cite{cherem2} extended Pillet's wave operator results to the case of a potential undergoing Brownian motion on the whole space. The potential $V \in L^2(\R^\3)$ was assumed to decay like $|x|^{-5/2}$ at infinity.

Erdogan--Killip--Schlag \cite{eks} proved that the energy of solutions to Schr\"{o}dinger's equation on the torus with a random, time-dependent potential grows on average at a rate of $\langle t \rangle^{1/2}$.

In \cite{beso}, Strichartz estimates --- among other results --- have been established for short-range potentials $V \in L^{\3/2, 1}$, translated by sufficiently rapid Brownian paths in $\R^\3$.

A related, but distinct problem is to consider random, spatially uncorrelated time-dependent potentials without decay at spatial infinity. This problem has been studied with $\mathbb{R}^{\3}$ replaced by a lattice, i.e.,\ for the Anderson model with time-dependent random potential. Solutions exhibit diffusive behavior, spreading at a rate of $\sqrt t$, instead of scattering behavior, as in our case.

Ovchinnikov and \'{E}rikhman \cite{over} showed this to be true for a Gaussian white noise potential. Kang and Schenker \cite{kasc} obtained similar results for a more general class of (Markovian) random potentials.
Che\-rem\-shan\-tsev \cite{cherem\3}, \cite{cherem4} proved that the averaged momentum exhibits diffusive scaling, up to logarithmic factors, for a general random potential. For certain models of a quantum particle with an internal degree of freedom hopping on a three- or higher-dimensional lattice and interacting with a simple quantum-mechanical thermal reservoir, a central limit theorem, diffusive scaling for the mean-square displacement and a Maxwellian distribution law have been established in \cite{fdp, fd, dk}.

%
%

Many results are known for nonlinear Schr\"{o}dinger equations with additive or multiplicative white noise terms. This problem was considered in \cite{bang1}, \cite{bang2}, \cite{bang\3}, \cite{bang4}, \cite{bouard1}, \cite{bouard2}, \cite{bouard\3}, \cite{gautier}, \cite{barbu}, \cite{barbu2}, and many other papers. The authors studied coherence times, the global existence and uniqueness of solutions, blow-up phenomena, and the large deviation principle.


Our linear results can be compared with those of Frank--Lewin--Lieb--Seiringer \cite{FLLS}, who prove the following: consider the Liouville-type equation
\be\begin{aligned}\lb{equation}
& i \partial_t \gamma(t)=[-\Delta ,\gamma(t)]+i \Gamma(t),\ \gamma(0)=\gamma_0 \in \B(L^2).
\end{aligned}\ee
Here the source term $\Gamma(t)$ is also a family of self-adjoint operators on $L^2(\mathbb{R}^d)$, $d \geq 1$, and $\gamma(t)$ is a time-dependent family of normal operators in some Schatten--von Neumann class $\mathfrak{S}^{\frac{2q}{q+1}}$, for $q$ as below.

The density (\ref{density}) satisfies the following dispersive bound:
\begin{theorem}[\cite{FLLS}]
Assume that $p,q,d \ge1$ satisfy 
$$
1\leq q\le 1+2/d \quad \text{and}\quad 2/p+d/q=d,
$$
and let $\gamma(t)$ solve the Liouville equation (\ref{equation}). Then
$$
\|\rho_{\gamma(t)}\|_{L_t^p L_x^q} \les \|\gamma_0\|_{\mathfrak{S}^{\frac{2q}{q+1}}} + \int_\R \|\Gamma(t)\|_{\mathfrak{S}^{\frac{2q}{q+1}}} \dd t.
$$
\end{theorem}

This is further improved to $1\leq q< 1+2/(d-1)$ in \cite{FrSa}.

Corollary \ref{cor_liouville} intersects with the non-dispersive endpoint case $(p=\infty,\ q=1)$ in \cite{FLLS}, concerning the conservation of density. In addition to this conservation law, however, we have trace-class dispersive estimates as well.




\subsection{Notations} $a \les b$ means that $|a| \leq C |b|$ for some constant $C$.

We denote various positive constants, not always the same, by $C$.

We denote Lorentz spaces by $L^{p, q}$; see \cite{bergh} for the definition and properties.

Let $f_1 \otimes f_2$ mean $f_1(x_1) f_2(x_2)$.

\subsection{Conditions on the potential}

For the Schr\"{o}dinger-type equation (\ref{average}) and throughout the paper, the sharp condition for Strichartz estimates on the potential $V$ is given by real interpolation:
\be\lb{weak_cond}
V \in (C_y L^1_x, C_y L^\infty_x)_{1/\3, 1}.
\ee
However, such interpolation spaces do not have a simple concrete representation, see Note 5.8.6, p.\;1\30, in \cite{bergh}, as well as the original reference \cite{cwikel}.

A simpler, but too strong replacement condition would be
\be\lb{strong_cond}
V \in L^{\3/2, 1}_x C_y.
\ee
Indeed, since $L^1_x C_y \subset C_y L^1_x$ and $L^p_x C_y \subset C_y L^p_x$ for arbitrarily large $p$, by interpolation we get that
$$
L^{\3/2, 1}_x C_y \subset (C_y L^1_x, C_y L^\infty_x)_{1/\3, 1}.
$$
All results in this paper are simpler to prove assuming (\ref{strong_cond}), but (\ref{strong_cond}) does not hold when $V_\omega = V(x-B_t)$, where $B_t$ is Brownian motion, an application in which we are interested.

Results where the norms are in the opposite order ($x$ on the inside, $y$ on the outside) are slightly harder to prove, but have more general applicability. In this paper, to keep our proofs reasonably simple, we use the following condition: $V \in C_y (L^1_x \cap L^\infty_x)$. This suffices for most applications we wish to model through (\ref{random_nonlin}).

In general, the same proofs work when $V \in C_y (L^p_x \cap L^q_x)$ for $1 \leq p < \3/2 < q \leq \infty$, but lead to a decay rate of $\min(t^{-\3/(2p)}, t^{-\3/(2q)})$ instead of $\langle t \rangle^{-\3/2}$.

We also consider the following norm, for square roots of $V$:
$$
\|v\|_{\v} = \sum_{n \in \Z} \sup_{y \in Y} \|\chi_{|v| \in [2^n, 2^{n+1}]} v\|_{L^\3_x}.
$$
The finiteness of this norm means that $v(x, y)$ is uniformly, in some sense, in $L^{\3,1}_x$ for all $y \in Y$.

\subsection{Outline of proofs}\subsubsection{Reduction to a deterministic equation}
A crucial idea, introduced by Pillet in \cite{pillet}, is that if $\psi$ satisfies a stochastic equation then its average will satisfy some related deterministic equation.

This technique is analogous to the Feynman--Kac formula. In its simplest form, the Feynman--Kac formula states that, given a Brownian motion $B_t$ and a function $\phi \in L^p(\R^d)$, $\phi$'s average value when sampled at $B_t$
$$
\gamma(x, t) := \E\big(\phi(x-B_t)\big)
$$
will satisfy the parabolic equation
$$
\partial_t \gamma = \Delta \gamma,\ \gamma(0) = \phi.
$$

In general, a random term in the stochastic equation will lead to a dissipative term in the deterministic equation satisfied by the average of the solution to the stochastic equation. For equation (\ref{random_nonlin}), the randomness leads to the dissipative term $-A$ in (\ref{average}) and (\ref{average_liouville}).

For the Schr\"{o}dinger-type equation (\ref{random_lin}), consider the average
$$
g(x, y, t):=\E\big\{\psi_\omega(x, t) \mid X_t = \omega(t) = y\big\}.
$$
Then $g$ will satisfy the mixed-type (Schr\"{o}dinger and parabolic) equation
\be\lb{average}\begin{aligned}
i \partial_t g - \Delta_x g + i A g + V g = G,\ g(0) = \E (\psi_0(x, \omega) \mid \omega(0) = y),
\end{aligned}\ee
where $G$ is an appropriate average of the inhomogenous term $\Psi$:
$$
G(x, y, t):=\E\big\{\Psi_\omega(x, t) \mid X_t = y\big\}.
$$
We prove dispersive estimates for $g$, see Proposition \ref{prop26}. However, even though such bounds are useful, they do not directly lead to similar bounds on $\psi_\omega$ itself, because when averaging $\psi_\omega$ to obtain $g$ there will be cancellations.

Thus Pillet, and we following him, further considered the average of the density matrix $\psi_\omega \otimes \ov \psi_\omega$:
\be\lb{def_avg}
f(x_1, x_2, y, t):=\E\big\{\psi_\omega(x_1, t) \ov \psi_\omega(x_2, t) \mid X_t = \omega(t) = y\big\}.
\ee
The density matrix $\psi_\omega \otimes \ov \psi_\omega = \psi_\omega(x_1, t) \ov \psi_\omega(x_2, t)$ satisfies the Liouville equation
\be\lb{lio}
\begin{aligned}
&i \partial_t [\psi_\omega \otimes \ov \psi_\omega] + (- \Delta_{x_1} + \Delta_{x_2} + V_\omega \otimes 1 - 1 \otimes V_\omega) [\psi_\omega \otimes \ov \psi_\omega] = \Psi_\omega \otimes \ov \psi_\omega - \psi_\omega \otimes \ov \Psi_\omega,
\end{aligned}
\ee
where $[\psi_\omega \otimes \ov \psi_\omega](0) = \psi_0(x_1, \omega(0)) \ov \psi_0(x_2, \omega(0))$.

The density matrix $\psi_\omega \otimes \ov \psi_\omega$ is Hermitian and positive semi-definite, hence so will be its average $f$, which satisfies the following mixed-type (Liouville and dissipative) equation:
\be\lb{average_liouville}\begin{aligned}
&i \partial_t f - \Delta_{x_1} f + \Delta_{x_2} f + i A f + (V \otimes 1 - 1 \otimes V) f = F,\\
&f(0) = \E \big\{\psi_0(x_1, \omega) \ov \psi_0(x_2, \omega) \mid \omega(0) = y\big\}.
\end{aligned}\ee
The source term $F$ comes from averaging the tensor product of $\Psi_\omega$ and of the solution $\psi_\omega$:
$$
F(x_1, x_2, y, t):=\E\big\{\Psi_\omega(x_1, t) \ov \psi_\omega(x_2, t) - \psi_\omega(x_1, t) \ov \Psi_\omega(x_2, t) \mid X_t = y\big\}.
$$
Due to the positive semi-definiteness of the density matrix, bounds on $f$ will translate into average (probabilistic) bounds for the density matrix.

Both the density matrix, being a rank-one operator, and its average $f$ belong to the trace class $\frak S_1$; for more details, see Section \ref{sec_trace}. The density matrices' traces are given by
$$
\tr (\psi_\omega(t) \otimes \ov \psi_\omega(t)) = \int_{\R^\3} |\psi_\omega(x, t)|^2 \dd x,
$$
for each $t$; respectively, for each $y$ and $t$,
$$
\tr f(y, t) = \int_{\R^\3} f(x, x, y, t) \dd x = \int_{\R^\3} \E\big\{|\psi_\omega(x, t)|^2 \mid X_t=y\big\} \dd x = \E \Big\{\int_{\R^\3} |\psi_\omega(x, t)|^2 \dd x \mid X_t = y \Big\}.
$$

Due to the unitarity of the evolution, $\tr (\psi_\omega \otimes \ov \psi_\omega) = \|\psi_\omega(t)\|_{L^2}^2 = \|\psi_0(x, \omega(0))\|_{L^2_x}^2$ is constant for each trajectory $\omega \in \Omega$ and likewise for the average
$$
\int_Y \tr f(y, t) \dd y = \|\psi_0\|_{L^2_{y, x}}^2.
$$

However, in order to bound the nonlinear term's contribution, unitarity is not enough; we need to prove that the solution decays in some sense. To obtain dispersive estimates we employ the following Feynman--Kac-type formula, proved by Pillet \cite{pillet2}:
\begin{proposition} For solutions $\psi_\omega$ to (\ref{random_lin}), $f$ defined by (\ref{def_avg}), and any $t \geq 0$
\be\lb{feynman}
\E \int_{\R^\3} |\psi_\omega(x, t)|^2 |V_\omega(x, t)| \dd x = \int_{Y \times \R^\3} |V(x, y)| f(x, x, y, t) \dd x \dd y.
\ee
\end{proposition}

After integrating in time,
\be\lb{feynman'}
\E \int_{\R^{\3+1}} |\psi_\omega(x, t)|^2 |V_\omega(x, t)| \dd x \dd t = \int_{Y \times \R^{\3+1}} |V(x, y)| f(x, x, y, t) \dd x \dd y \dd t.
\ee
Hence the finiteness of the right-hand side in (\ref{feynman}) will imply a weighted space-time bound on the average trace of the density matrix.

Strichartz and pointwise decay estimates for $\psi_\omega$ will follow from it, by bootstrapping in the Duhamel formula
$$
\psi_\omega(t) = e^{-it\Delta} \psi_0 + i\int_0^t e^{-i(t-s)\Delta} V_\omega \psi_\omega(s) \dd s.
$$

The problem hence reduces to obtaining dispersive estimates for equation (\ref{average_liouville}):
$$\begin{aligned}
&i \partial_t f - \Delta_{x_1} f + \Delta_{x_2} f + i A f + (V \otimes 1 - 1 \otimes V) f = F,\\
&f(0) = \psi_0(x_1, \omega(0)) \ov \psi_0(x_2, \omega(0)).
\end{aligned}$$

The validity of dispersive estimates for (\ref{average_liouville}) depends on the spectrum of the evolution operator
\be\lb{ham}
H=-\Delta_{x_1} + \Delta_{x_2} + iA + (V \otimes 1 - 1 \otimes V).
\ee

Absent dissipation, represented here by $-A$, this operator may have bound states corresponding to those of $-\Delta+V$. However, Pillet \cite{pillet} proved that the operator (\ref{ham}) has no eigenstates; also see Lemma \ref{invertibility_liouville} for a similar result in our context.

Indeed, under Assumption \ref{nontrivial} on the Markov process generator $A \geq 0$, which guarantees the non-triviality of the random part of the potential, all the bound states present in the deterministic case become exponentially decaying complex resonances in the positive half-plane.

Getting rid of bound states simplifies the analysis --- and is one reason to consider random perturbations. On the other hand, equation (\ref{average_liouville}) becomes more complicated because $H$ is no longer the product of two commuting terms corresponding to the two space variables, unlike in the Liouville equation (\ref{lio}).


\subsubsection{A model case: the scalar equation} Consider the scalar Schr\"{o}dinger-type equation (\ref{average}) first, which helps understanding (\ref{average_liouville}). Here the perturbed and free evolution operators, acting on functions on $\R^\3 \times Y$, are
$$
-\Delta + V(x, y) + iA \text{ and } -\Delta + iA.
$$
The Laplacian is in $x \in \R^\3$ and $A$ acts on $y \in Y$. The free resolvent and perturbed resolvents are
$$
R^{iA}_0(\lambda) := (-\Delta + i A -\lambda)^{-1},\ R^{iA}_V(\lambda) := (-\Delta + i A + V -\lambda)^{-1}.
$$

The dissipative term $iA$ breaks the time symmetry. Infinitely many horizontal lines in the upper half-plane $\Im \lambda \geq 0$, more precisely those given by $\Im \lambda \in \sigma (A)$, will be in the spectrum of $-\Delta+iA$. Hence the spectral analysis takes place in the lower half-plane.

Consider a symmetric decomposition of the potential $V = v_1 v_2$, where each of $v_1$ and $v_2$ has size $\sim |V|^{1/2}$, but need not commute. Let the symmetric Kato--Birman operator be
\be\lb{katobirman}
KB(\lambda): = I + v_2 R^{iA}_0(\lambda) v_1.
\ee
This perturbation of the identity is compact on $L^p_y L^2_x$, $1 \leq p \leq \infty$, for $\Imim \lambda \leq 0$, by Lemma \ref{lemma_compact}. The resolvent identity can be written in symmetric form as
\be\lb{symmetric_resolvent}
(I + v_2 R^{iA}_0(\lambda) v_1) (I - v_2 R^{iA}_V(\lambda) v_1) = I.
\ee

Due to the dissipativity of $-A$, $KB(\lambda)$ is invertible whenever $\Im \lambda \leq 0$. This is proved by Fredholm's alternative, see Proposition \ref{fredholm} and Lemma \ref{invertibility}. If the equation 
$$
f = - v_2 (-\Delta+iA+V-\lambda)^{-1} v_1 f
$$
had some nonzero solution $f \in L^2$, $f \ne 0$, by Agmon's argument the solution must have the form $f(x, y) = v_2(x, y) h(y) g(x)$, where $h$ is the ground state of $A$ and $g(x)$ is a solution of
$$
(-\Delta+V(x, y)-\lambda) h(y) g(x) = 0.
$$
But then the non-triviality Assumption \ref{nontrivial} implies that $g$ vanishes on an open set and the existence of such bound states is precluded by unique continuation.

Furthermore, the resolvent is the Fourier transform of the evolution operator:
$$
\int_0^\infty e^{it(-\Delta_x +i A)} e^{-it\lambda} \dd \lambda = -(-i\Delta-A-i\lambda)^{-1} = i (-\Delta+iA-\lambda)^{-1} = i R^{iA}_0(\lambda)
$$
and
$$
\int_0^\infty e^{it(-\Delta_x +i A+V)} e^{-it\lambda} \dd \lambda = i R^{iA}_V(\lambda).
$$
The Kato--Birman operator is the Fourier transform of the integral kernel
$$
\big(I + KB\big)^\vee(t) = \one + i T(t),\ \one = \delta_{t=0} I,\ T(t) := \chi_{t \geq 0}(t) v_2 e^{it(-\Delta+iA)} v_1
$$
and likewise in the perturbed case:
$$
\big(I-v_2 R^{iA}_V(\lambda)v_1\big)^\vee(t) = \one - i T_V(t),\ T_V(t) := \chi_{t \geq 0}(t) v_2 e^{it(-\Delta+iA+V)} v_1.
$$
Duhamel's identity for equation (\ref{average}) is
\be\lb{duhamel_v}
g(t) = e^{it(-\Delta+iA)} g(0) -i \int_0^t e^{i(t-s)(-\Delta+iA)} (G(s)-Vg(s)) \dd s.
\ee
Setting $G=0$, the Fourier transform of (\ref{duhamel_v}) is the resolvent identity:
$$
R^{iA}_V = R^{iA}_0 - R^{iA}_0 V R^{iA}_V.
$$
The symmetric resolvent identity (\ref{symmetric_resolvent}) is the Fourier transform of the following symmetric version of Duhamel's identity:
\be\lb{duhamel_T}
(\one + i T) \ast (\one - iT_V) = \one.
\ee

This correspondence is useful both ways. Starting from known properties of the free evolution, we obtain the compactness of the free resolvent in Lemma \ref{lemma_compact}. Conversely, an analysis of the perturbed resolvent is crucial in our proof of the dispersive estimates, which is based on a variant of Wiener's Theorem, Theorem \ref{thm:Wiener}.

Given a Banach space $X$, for some arbitrary $p>1$ let
$$
\W_X:=\langle t \rangle^{-p} L^\infty_t \B(X).
$$
$\W_X$ is a Banach algebra, with the algebra operation given by convolution in the $t$ variable and composition of operators in $\B(X)$. To $\W_X$ we can adjoin an identity element, $\one = \delta_{t=0} I$, to form the unital algebra $\ov \W_X = \W_X \oplus \C \one$.

Elements of $\W_X$ are families of $X$-bounded operators $T=T(t)$, indexed by $t \in \R$. If $T \in \W_X$, then $T(\cdot-\delta) \in \W_X$ as well, for any $\delta \in \R$, albeit with non-uniformly bounded norm.

\begin{theorem}[Wiener's Theorem] \label{thm:Wiener}
Fix $p>1$ and let $T$ be an element of $\W_X$ with the property that
\be
\begin{aligned}
\label{translation} &\lim_{\delta \to 0}
  \|T - T(\cdot-\delta)\|_{\W_X} = 0
\end{aligned}
\ee
or more generally for some $N \geq 1$
\be\tag{\ref{translation}'}
\begin{aligned}
\label{translation_high} &\lim_{\delta \to 0}
  \|T^N - T^N(\cdot-\delta)\|_{\W_X} = 0,
\end{aligned}
\ee
where $T^N$ refers to the $N$-th power of $T$ under the $\W_X$ algebra operation.

If $I + \widehat{T}(\lambda)$ is an invertible element of $\B(X)$ for every
$\lambda \in \R$, then $\one + T$ possesses an inverse in
$\overline{\W}_X$ of the form $\one + S$, with $S \in \W_X$.
\end{theorem}

Thus, when $\one + i T$ is invertible in an appropriate operator algebra, inverting it leads to decay estimates for the perturbed evolution $T_V$, due to (\ref{duhamel_T}). This is the method we use to prove all dispersive estimates in this paper.

For the scalar case, the main result obtained in this manner is as follows:
\begin{proposition}\lb{prop26}
For $V \in C_y (L^1_x \cap L^\infty_x)$ real-valued and satisfying Assumption \ref{nontrivial}, consider the equation
\be\lb{eq\3D}\begin{aligned}
i \partial_t f - \Delta_x f + i A f + V(x, y) f = F,\ f(0) = f_0.
\end{aligned}\ee
Then, for $1 \leq p \leq \infty$ and $F=0$,
\be\lb{weaker_bound}
\|e^{it(-\Delta+V+iA)}f_0\|_{L^p_y (L^2_x + L^\infty_x)} \les \langle t \rangle^{-\3/2} \|f_0\|_{L^p_y(L^1_x \cap L^2_x)}.
\ee
\end{proposition}

\subsubsection{The main estimate} Next, consider the dissipative Liouville-type equation (\ref{average_liouville}).

One cannot expect better results in the perturbed case than for the free equation, which shall serve as our model:
\be\lb{free_l}
i \partial_t F - \Delta_{x_1} F + \Delta_{x_2} F + iAF = G,\ F(0)=F_0.
\ee
For trace-class initial data, the $x_1$ and $x_2$ components decouple and the estimates we obtain are products of two separate estimates, one for each component. In particular, the rate of decay in time should be the square of the $\3$-dimensional one.

We prove two kinds of trace-class estimates in Proposition \ref{prop27} and Corollary \ref{cor_strichartz}: with an average ($L^1_t$) rate of decay and with a pointwise-in-$t$ rate of decay. Both have straightforward analogues in the free case:
$$
\|F\|_{L^1_t L^1_y (L^{6, 2}_{x_1} \otimes L^{6, 2}_{x_2})} \les \|F_0\|_{L^1_y \frak S_1}
$$
and
$$
\|F(t)\|_{L^1_y ((L^2_{x_1}+ L^\infty_{x_1}) \otimes (L^2_{x_2} + L^\infty_{x_2}))} \les \langle t \rangle^{-\3} \|F_0\|_{L^1_y [(L^1_{x_1} \cap L^2_{x_1}) \otimes (L^1_{x_2} \cap L^2_{x_2})]}.
$$
For $L^1_{x_1, x_2} = L^1_{x_1} \otimes L^1_{x_2}$ initial data, one also gets the sharp $t^{-\3}$ decay rate. Proving this estimate in the perturbed case is left for a future paper.

The free equation (\ref{free_l}) can also serve as a model for the inhomogenous estimates that can be proved for equation (\ref{average_liouville}). It turns out that for technical reasons the best choice of source term $G$ in either (\ref{free_l}) or (\ref{average_liouville}) will be one that allows the Schr\"{o}dinger evolution in the two variables $x_1$ and $x_2$ to decouple and take place for different lengths of time:
$$
G(y, s, x_1, x_2) = \int_0^s e^{i(s-\tilde s)\Delta_{x_2}} \tilde G(y, s, \tilde s, x_1, x_2) \dd \tilde s.
$$
This is akin to appending an extra interval of Schr\"{o}dinger evolution in just one of the variables, thus making them of different lengths.


%
%

%


We aim to get such bounds in the presence of a potential and source terms, i.e.\;for equation (\ref{average_liouville}).

It is not obvious that the potential $V \otimes 1 - 1 \otimes V$ has a symmetric decomposition like the one in (\ref{katobirman}), but we write it as the matrix product $V \otimes 1 - 1 \otimes V = V_1 V_2$, where
\be\begin{aligned}\lb{more_notations}
&V_2 := \bpm V_{21} \\ V_{22} \epm := \bpm v_2 \otimes 1 \\ -1 \otimes v_2 \epm,\ V_1 := \bpm V_{11} & V_{12} \epm := \bpm v_1 \otimes 1 & 1 \otimes v_1 \epm,
\end{aligned}\ee
and
$$\begin{aligned}
v_1(x, y):= |V|^{1/2}(x, y),\ v_2(x, y):= |V|^{1/2}(x, y) \sgn V(x, y).
\end{aligned}$$
Also for $\Im \lambda \leq 0$ let
$$
R^{iA}(\lambda):=(-\Delta_{x_1}+\Delta_{x_2}+iA_y-\lambda)^{-1},
$$
$$\begin{aligned}
T(t)&:= \bpm T_{11} & T_{12} \\ T_{21} & T_{22} \epm := \chi_{t>0} V_2 e^{it(-\Delta_{x_1}+\Delta_{x_2}+iA_y)} V_1.
\end{aligned}$$
The symmetric Kato--Birman perturbation is
$$\begin{aligned}
i\widehat T(\lambda) &= \bpm i\widehat T_{11}(\lambda) & i\widehat T_{12}(\lambda) \\ i\widehat T_{21}(\lambda) & i\widehat T_{22}(\lambda) \epm = V_2 R^{iA}(\lambda) V_1.
\end{aligned}$$

The main issue is that $V \otimes 1 - 1 \otimes V$ does not vanish at infinity in all directions, making it in fact a multi-channel potential. Hence there are two kinds of terms in $i\widehat T(\lambda)$: diagonal ones, with weights in the same variable at both ends, and off-diagonal cross-terms, whose weights involve different channels at each end.

Diagonal terms are easier to study and essentially reduce to the scalar case, equation (\ref{average}). However, the cross-terms are not compact, making $\widehat T(\lambda)$ not compact either.

To address this issue and be able to use Fredholm's alternative, we employ the following decomposition, along the lines of Pillet's proof of Lemma 5.2, p.\;11 in \cite{pillet}:
$$\begin{aligned}
I + i\widehat T &= \bpm I + i\widehat T_{11} & 0 \\ 0 & I + i\widehat T_{22} \epm\ \bpm I & i(I + i\widehat T_{11})^{-1} \widehat T_{12} \\ i(I + i\widehat T_{22})^{-1} \widehat T_{21} & I \epm \\
& = (I + i\widehat T_{diag}) (I + i(I + i\widehat T_{diag})^{-1} \widehat S_0) \\
&= (I + i\widehat T_{diag}) (I + i\widehat S),
\end{aligned}$$
where we let
$$
\widehat T_{diag} = \bpm \widehat T_{11} & 0 \\ 0 & \widehat T_{22} \epm,\ \widehat S_0 = \bpm 0 & \widehat T_{12} \\ \widehat T_{21} & 0 \epm,
$$
and
$$\begin{aligned}
\widehat S := (I + i\widehat T_{diag})^{-1} \widehat S_0 = \bpm 0 & (I + i \widehat T_{11})^{-1} \widehat T_{12} \\ (I + iT_{22})^{-1} \widehat T_{21} & 0 \epm.
\end{aligned}$$
Then
$$
(I + i\widehat T)^{-1} = (I + i \widehat S)^{-1} (I + i\widehat T_{diag})^{-1}.
$$
Hence, if we could invert $I + i \widehat S$, we could also invert the Kato--Birman operator.

It turns out that $I + i \widehat S$ is not compact either, in the appropriate operator space; so we instead consider
$$
I + \widehat S^2 = (I+i\widehat S)(I-i\widehat S).
$$
Even though $\widehat S^2$ is still not compact, $\widehat S^2$ admits a further decomposition
$$
\widehat S(\lambda)^2 = U_1 B_\lambda U_2,
$$
such that $B_\lambda U_2 U_1$ is compact. Using Fredholm's alternative, one can invert $I + B_\lambda U_2 U_1$ and thus eventually obtain an inverse for the Kato--Birman operator as well.

In order to get the optimal rate of decay of $\langle t \rangle^{-\3}$, we use Wiener's Theorem \ref{thm:Wiener} in this rather complicated setting, while also keeping track of the initial and final segments $U_1$ and $U_2$. The main results are Proposition \ref{prop27} and Corollary \ref{cor_strichartz}.


\section{Proofs}

\subsection{Spectral analysis}
Our goal in this section is to prove the validity of the spectral condition for the scalar averaged equation (\ref{average})
$$
i \partial_t g - \Delta_x g - i A g + V g = G.
$$

For the sake of completeness, we define complete continuity and prove the appropriate version of Fredholm's alternative.

\begin{definition} Given a Banach space of operators $W \subset \B(X)$ and $B \in W$, we say that $B$ is completely continuous in $W$ if $B$ can be approximated by finite-rank operators in the $W$-norm.
\end{definition}

\begin{lemma}[Fredholm's alternative]\lb{fredholm} Let $X$ be a Banach space and consider an operator $T \in \B(X)$ which can be approximated in the $\B(X)$ norm by finite-rank operators (i.e.\;is completely continuous in $\B(X)$). Then either the equation
$$
f = T f
$$
has a nonzero solution $f \in X$, $f \ne 0$, or $(I-T)^{-1} \in \B(X)$.

In the latter case, if $T$ has finite rank and belongs to a (not necessarily closed) Banach operator subalgebra $\widehat X \subset \B(X)$, then
$$
(I-T)^{-1} - I \in \widehat X.
$$

Finally, consider a subalgebra $\widehat X$ such that $\widehat X \B(X) \widehat X \subset \widehat X$, i.e.\;if $T_1, T_2 \in \widehat X$ and $T \in \B(X)$, then $T_1 T T_2 \in \widehat X$ and
\be\lb{proper}
\|T_1 T T_2\|_{\widehat X} \les \|T_1\|_{\widehat X} \|T\|_{\B(X)} \|T_2\|_{\widehat X}.
\ee

Then, if $T \in \widehat X$ is completely continuous in $\widehat X$, it follows that $(I-T)^{-1} - I \in \widehat X$.
\end{lemma}
\begin{proof}
First consider a finite-rank operator
$$
\tilde T = \sum_{n=1}^N f_n \otimes g_n,
$$
where $f_n \in X$ and $g_n \in X^*$. The invertibility of $I-\tilde T$ is determined by its behavior on the finite-dimensional vector space spanned by $\{f_n: 1 \leq n \leq N\}$, which in turn is completely characterized by the $N \times N$ determinant
$$
\det(\delta_{mn}+\langle g_m, f_n\rangle).
$$
If the determinant is zero, then $I+\tilde T$ cannot be invertible, because it has a nontrivial kernel. If the determinant is nonzero, then one can explicitly construct an inverse, which is again a finite-rank perturbation of the identity whose support and range are spanned by the same vectors:
\be\lb{finite_rank}
(I-\tilde T)^{-1} = I + \sum_{m, n=1}^N c_{mn} f_m \otimes g_n.
\ee
So Fredholm's alternative is true for finite-rank perturbations of the identity.

Next, consider a sequence of finite-rank operators $T_n \to T$. Applying Fredholm's alternative to each, one possibility is that the inverses do not exist or are not uniformly bounded. Then there exist sequences $n_k$ and $f_{n_k}$, $g_{n_k}$ such that $(I-T_{n_k}) f_{n_k} = g_{n_k}$, but $\|f_{n_k}\|/\|g_{n_k} \to \infty$. Normalizing them so that $\|f_{n_k}\|=1$, we get that $\|g_{n_k}\| \to 0$.

Using a ``diagonalization'' argument one can find a subsequence $n_{k_\ell}$, $T_{n_{k_\ell}} f_{n_{k_\ell}} \to T f$, so $(I-T)f=0$.

The other possibility is that the inverses exist and are uniformly bounded: if $(I-T_n) f_n = g$ then $\|f_n\| \les \|g\|$ with a constant that does not depend on $n$.

Again, using a ``diagonalization'' argument one can find a subsequence $n_k$ such that $T_{n_k} f_{n_k} \to T f$ and $\|f\| \les \|g\|$. Then $(I-T) f = g$ and we let $f = (I-T)^{-1} g$.

Moreover, this $f$ is uniquely defined, because otherwise $I-T_n$ would have a non-unique inverse for sufficiently large $n$, which contradicts the inverse being bounded.

Hence $I-T$ must have a bounded inverse as well.

If $T$ is of finite rank, let $T \in \widehat X$ 
act on the subalgebra $\langle T \rangle \subset \widehat X \subset \B(X)$ spanned by its powers $T^n$, $n \geq 0$. This subalgebra is finite-dimensional by the Cayley--Hamilton theorem, of dimension no larger than the rank of $T$.



Applying Fredholm's alternative in this context yields two cases. Either $I-T$ is invertible in $\B(\langle T \rangle)$, in which case there exists $S \in \B(\langle T \rangle)$ such that
$$
(I-T)S = S(I-T) = I_{\B(\langle T \rangle)}.
$$
Applying this identity to $I \in \langle T \rangle$, we get that
$$
(I-T)S(I) = S(I)(I-T) = I_X,
$$
so
$$
(I-T)^{-1} = S(I) \in \langle T \rangle \subset \widehat X.
$$
Or there exists some $s \in \langle T \rangle$, $s \ne 0$, such that
$$
(I-T)s = 0.
$$
But if $s \ne 0$ then there must exist some $f \in X$ such that $sf \ne 0$, so $(I-T)(sf) = 0$ for $sf \in X$, $sf \ne 0$. This contradicts the invertibility of $I-T$.

Thus, when $I-T$ is invertible and $T$ has finite rank, its inverse must belong to $\widehat X$ as well.

Finally, if $\widehat X \subset \B(X)$ has property (\ref{proper}), let $T \in \widehat X$ be completely continuous, such that $I-T$ is invertible in $\B(X)$. Then
$$
(I-T)^{-1} = I + T + T^2 + T [(I-T)^{-1}-I] T,
$$
so by property (\ref{proper}) it follows that $(I-T)^{-1} \in \widehat X$ as well.
\end{proof}

\subsection{Free evolution and free resolvent bounds}

Returning to the mixed-type equation (\ref{average}), define the free resolvent
$$
R^{iA}_0(\lambda) := (-\Delta + i A -\lambda)^{-1}.
$$

To characterize the resolvent, we commence with a statement that holds under fewer conditions concerning $A$. In particular, when $p=2$, it suffices to assume that $A$ is self-adjoint.
\begin{lemma} Fix $1 \leq p \leq \infty$. Let $v_1$ and $v_2$ belong to $\v \subset (C_y L^\infty_x, C_y L^2_x)_{2/\3, 1}$ and suppose that $e^{-tA}$ is uniformly bounded on $L^p_y$:
$$
\sup_{t \geq 0} \|e^{-tA}\|_{\B(L^p_y)} < \infty.
$$
Then
\be\lb{dispersive}
\|v_2 e^{it (-\Delta_x+iA)} v_1 f\|_{L^1_t L^p_y L^2_x} \les \|v_2\|_{\v} \|v_1\|_{\v} \|f\|_{L^p_y L^2_x}.
\ee
Hence, for $\lambda \in \C$ with $\Im \lambda \leq 0$, the operator
$$
v_2(x, y) R^{iA}_0(\lambda) v_1(x, y)
$$
is uniformly bounded on $L^p_y L^2_x$ and
$$
\sup_{\Im \lambda \leq 0}  \|v_2(x, y) R^{iA}_0(\lambda) v_1(x, y) f\|_{L^p_y L^2_x} \les \|v_2\|_{\v} \|v_1\|_{\v} \|f\|_{L^p_y L^2_x}.
$$
\end{lemma}
The proof is based on the correspondence between the time-independent and the time-dependent settings, given by the Fourier transform, which we shall also use in several subsequent proofs.

\begin{proof}
We use the real interpolation method, especially Theorem 5.6.1, p.\;122, from \cite{bergh}.

Since, for each $t \geq 0$, $e^{-tA}$ is a contraction on $L^p_y$, $1 \leq p \leq \infty$,
$$\begin{aligned}
\|v_2 \chi_{[2^n, 2^{n+1}]}(t) e^{it (-\Delta_x+iA)} v_1 f\|_{\ell^\infty_n L^\infty_t L^\infty_y L^2_x} &\les \|v_2\|_{C_y L^\infty_x} \|\chi_{t \in [2^n, 2^{n+1}]} e^{it(-\Delta_x+iA)} v_1 f\|_{\ell^\infty_n L^\infty_t L^\infty_y L^2_x} \\
&\les \|v_2\|_{C_y L^\infty_x} \|\chi_{t \in [2^n, 2^{n+1}]} e^{-it\Delta_x} v_1 f\|_{\ell^\infty_n L^\infty_t L^\infty_y L^2_x} \\
&\les \|v_2\|_{C_y L^\infty_x} \|v_1\|_{C_y L^\infty_x} \|f\|_{L^\infty_y L^2_x}
\end{aligned}$$
and
$$\begin{aligned}
\|v_2 \chi_{[2^n, 2^{n+1}]}(t) e^{it (-\Delta_x+iA)} v_1 f\|_{2^{-\3n/2} \ell^\infty_n L^\infty_t L^\infty_y L^2_x} &\les \|v_2\|_{C_y L^2_x} \|\chi_{t \in [2^n, 2^{n+1}]} e^{it (-\Delta_x+iA)} v_1 f\|_{2^{-\3n/2} \ell^\infty_n L^\infty_t L^\infty_y L^\infty_x} \\
&\les \|v_2\|_{C_y L^2_x} \|\chi_{t \in [2^n, 2^{n+1}]} e^{-it\Delta_x} v_1 f\|_{2^{-\3n/2} \ell^\infty_n L^\infty_t L^\infty_y L^\infty_x} \\
&\les \|v_2\|_{C_y L^2_x} \|v_1\|_{C_y L^2_x} \|f\|_{L^\infty_y L^2_x}.
\end{aligned}$$


By real interpolation (see Problem 5, p.\;76, from \cite{bergh}) we then obtain that
$$
\|v_2 \chi_{[2^n, 2^{n+1}]}(t) e^{it (-\Delta_x+iA)} v_1 f\|_{2^{-n} \ell^1_n L^\infty_t L^\infty_y L^2_x} \les \|v_2\|_{(C_y L^\infty_x, C_y L^2_x)_{2/\3, 1}} \|v_1\|_{(C_y L^\infty_x, C_y L^2_x)_{2/\3, 1}} \|f\|_{L^\infty_y L^2_x}.
$$
This immediately implies (\ref{dispersive}) for $p=\infty$.
%
%

With minor modifications, this proof applies to all of $L^p_y L^2_x$, $1 \leq p \leq \infty$, as well. The key is using the same Lebesgue exponent for the $t$ and $y$ norms and having these two norms next to each other, so that they can be swapped as needed.

As above, one shows that
$$
\|v_2 \chi_{[2^n, 2^{n+1}]}(t) e^{it (-\Delta_x+iA)} v_1 f\|_{2^{n/p} \ell^\infty_n L^p_t L^p_y L^2_x} \les \|v_2\|_{C_y L^\infty_x} \|v_1\|_{C_y L^\infty_x} \|f\|_{L^p_y L^2_x}
$$
and
$$
\|v_2 \chi_{[2^n, 2^{n+1}]}(t) e^{it (-\Delta_x+iA)} v_1 f\|_{2^{n/p-\3n/2} \ell^\infty_n L^p_t L^p_y L^2_x} \les \|v_2\|_{C_y L^2_x} \|v_1\|_{C_y L^2_x} \|f\|_{L^p_y L^2_x},
$$
leading to
$$
\|v_2 \chi_{[2^n, 2^{n+1}]}(t) e^{it (-\Delta_x+iA)} v_1 f\|_{2^{n(-1+1/p)} \ell^1_n L^p_t L^p_y L^2_x} \les \|v_2\|_{\v} \|v_1\|_{\v} \|f\|_{L^p_y L^2_x}.
$$
This implies (\ref{dispersive}) for $1 \leq p \leq \infty$.

By taking the Fourier transform of (\ref{dispersive}) in the variable $t$, we obtain that for $\Im \lambda \leq 0$
\be\lb{marginire}
\|v_2 R^{iA}_0(\lambda) v_1\|_{\B(L^p_y L^2_x)} \les \|v_2\|_{\v} \|v_1\|_{\v}.
\ee
\end{proof}

Another more precise characterization of $v_2 R^{iA}_0 v_1$ is based on the following result.

\begin{lemma}\lb{uniform_lemma} Consider $v_1, v_2 \in L^{\3, 1}$. Then
$$
\int_\R \|v_2 e^{-it\Delta} v_1\|_{\B(L^2)} \dd t \les \|v_1\|_{L^{\3, 1}} \|v_2\|_{L^{\3, 1}}.
$$
More generally, consider two families $(v_1^i(t))_{i \in I}$ and $(v_2^j(t))_{j \in J}$, which belong to $L^{\3, 1}$, uniformly for all $y \in Y$, in the following sense:
\be\lb{uniformL\31}
\sum_{k \in \Z} \sup_{i \in I} \sup_t \|(\chi_{[2^k, 2^{k+1}]} \circ v_1^i(t)) v_1^i(t)\|_{L^\3} < \infty,\ \sum_{k \in \Z} \sup_{i \in I} \sup_t \|(\chi_{[2^k, 2^{k+1}]} \circ v_2^i(t)) v_2^i(t)\|_{L^\3} < \infty.
\ee
Then the following estimate holds uniformly for all $v_1^i(t)$ and $v_2^j(t)$:
\be\lb{estxy}
\sup_{i \in I, j \in J} \int_\R \|v_2^i(t) e^{-it\Delta} v_1^i(t)\|_{\B(L^2_x)} \dd t \les \|v_1\|_{L^{\3, 1}} \|v_2\|_{L^{\3, 1}}.
\ee
\end{lemma}
Hence the same result holds when e.g.\;$v_1(t)$ and $v_2(t)$ are two time-dependent weights and we replace $\|v_1\|_{L^{\3, 1}}$ and $\|v_2\|_{L^{\3, 1}}$ by
$$
\sum_{n \in \Z} \sup_t \|(\chi_{[2^k, 2^{k+1}]} \circ v_1) v_1\|_{L^\3},\ \sum_{n \in \Z} \sup_t \|(\chi_{[2^k, 2^{k+1}]} \circ v_2) v_2\|_{L^\3}.
$$
This is a quantitative way of stating that $v_1(t)$ and $v_2(t)$ are uniformly in $L^{\3, 1}$, for all $t$.
\begin{proof} We again proceed by real interpolation. If $v_1, v_2 \in L^2$, then
\be\lb{L2}
\|v_2 e^{-it\Delta} v_1\|_{\B(L^2_x)} \les t^{-\3/2} \|v_2\|_{L^2} \|v_1\|_{L^2},
\ee
hence
$$
\|v_2 e^{-it\Delta} v_1\|_{2^{-k/2} \ell^\infty_k L^1_t([2^k, 2^{k+1}]) B(L^2_x)} \les \|v_2\|_{L^2} \|v_1\|_{L^2}.
$$
If $v_1, v_2 \in L^\infty$, then
\be\lb{Linfty}
\|v_2 e^{-it\Delta} v_1\|_{B(L^2_x)} \les \|v_2\|_{L^\infty} \|v_1\|_{L^\infty},
\ee
hence
$$
\|v_2 e^{-it\Delta} v_1\|_{2^k \ell^\infty_k L^1_t([2^k, 2^{k+1}]) B(L^2_x)} \les \|v_2\|_{L^\infty} \|v_1\|_{L^\infty}.
$$
So by real interpolation (see Problem 5, p.\;76, from \cite{bergh})
$$
\|v_2^n e^{-it\Delta} v_1^m\|_{\ell^1_k L^1_t([2^k, 2^{k+1}]) B(L^2_x)} \les \|v_2^n\|_{L^{\3, 1}} \|v_1^m\|_{L^{\3, 1}}.
$$
Summing over $k$ we get the desired conclusion.

To understand the role played by real interpolation, consider a dyadic atomic decomposition of the $v_1$ and $v_2$ weights:
$$
v_1 = \sum_{m \in \Z} v_1^{(m)},\ v_2 = \sum_{n \in \Z} v_2^{(n)}.
$$
For each of these atoms, by (\ref{L2}) and (\ref{Linfty})
$$
\int \|v_2^{(n)} e^{-it\Delta} v_1^{(m)}\|_{\B(L^2_x)} \dd t \les [\|v_2^{(n)}\|_{L^2} \|v_1^{(m)}\|_{L^2}]^{2/\3} [\|v_2^{(n)}\|_{L^\infty} \|v_1^{(m)}\|_{L^\infty}]^{(\3-2)/\3} = \|v_2^{(n)}\|_{L^\3} \|v_1^{(m)}\|_{L^\3},
$$
uniformly for all $n$ and $m$. Then, summing over all scales $n$ and $m$, we get
$$
\int \|v_2 e^{-it\Delta} v_1\|_{\B(L^2_x)} \dd t \les \|v_2^n\|_{L^{\3, 1}} \|v_1^m\|_{L^{\3, 1}}.
$$

Thus, when dealing with the more general families $v_1^i(t)$ and $v_2^j(t)$, as long as they are uniformly in $L^{\3, 1}$ in the sense of (\ref{uniformL\31}), the same proof applies and we get the uniform bound (\ref{estxy}).
\end{proof}

Using Lemma \ref{uniform_lemma}, we next give a more precise characterization of $v_2 R^{iA} v_1$ in Lemma \ref{bdd}. We first analyze the time evolution $v_2 e^{it(-\Delta+iA)} v_1$, then use this to characterize $v_2 R^{iA} v_1$ by using the Fourier transform in the $t$ variable.

\begin{lemma}\lb{bdd} Suppose that
$$
\sup_{t > 0} \|e^{-tA}(y_1, y_2)\|_{L^\infty_{y_2} L^1_{y_1} \cap L^\infty_{y_2} L^1_{y_1}} < \infty.
$$
If $v_1, v_2 \in \v$ then
$$
\int_0^\infty \sup_{y_1, y_2 \in Y} \|v_2(x, y_1) e^{-it\Delta} v_1(x, y_2)\|_{\B(L^2)} \dd t \les \|v_2\|_{\v} \|v_1\|_{\v},
$$
hence
$$
\int_0^\infty \|v_2 e^{it(-\Delta+iA)} v_1\|_{(L^\infty_{y_1} L^1_{y_2} \cap L^\infty_{y_2} L^1_{y_1}) \B(L^2_x)} \dd t \les \|v_2\|_{\v} \|v_1\|_{\v}.
$$
So for $\lambda \in \C$ with $\Im \lambda \leq 0$, the operator $v_2 R_0^{iA}(\lambda) v_1$ is uniformly bounded on $L^p_y L^2_x$, $1 \leq p \leq \infty$, and
\be\lb{R0A}
\sup_{\Im \lambda \leq 0} \|v_2 R_0^{iA} v_1\|_{(L^\infty_{y_1} L^1_{y_2} \cap L^\infty_{y_2} L^1_{y_1}) \B(L^2_x)} \les \|v_2\|_{\v} \|v_1\|_{\v}.
\ee
\end{lemma}
See (\ref{syx}) for the notation.

\begin{proof}
Due to (\ref{estxy})
$$
\int \sup_{y_1, y_2 \in Y} \|v_2(x, y_1) e^{-it\Delta} v_1(x, y_2)\|_{\B(L^2_x)} \dd t \les \|v_2\|_{\v} \|v_1\|_{\v}.
$$

In particular,
$$
\int \|v_2(x, y_1) [e^{-tA}(y_1, y_2)] e^{-it\Delta} v_1(x, y_2)\|_{(L^\infty_{y_1} L^1_{y_2} \cap L^\infty_{y_2} L^1_{y_1}) \B(L^2_x)} \dd t \les \|v_2\|_{\v} \|v_1\|_{\v},
$$
which is the claimed conclusion.

The only fact we used about $A$ is that
$$
\sup_t \sup_{y_1} \int |e^{-tA}|(y_1, y_2) \dd y_2 < \infty,\ \sup_t \sup_{y_2} \int |e^{-tA}|(y_1, y_2) \dd y_1 < \infty,
$$
which follows from properties of Markov chains. Both integrals are always $1$.

Finally, taking the Fourier transform of $v_2 e^{it(-\Delta+iA)} v_1$, we obtain (\ref{R0A}) by Minkowski's inequality.
\end{proof}

We next prove that the Kato--Birman operator $KB(\lambda)$, see (\ref{katobirman}), is a compact perturbation of the identity for each $\lambda$ in the lower half-plane, $\Im \lambda \leq 0$, so that we can apply Fredholm's alternative, Lemma \ref{fredholm}.

In fact, this is already true for the time-dependent operator $v_2 e^{it(-\Delta+iA)} v_1$, which we show is completely continuous in a space involving the following norm:

\begin{definition}
For $R>0$, consider the Banach space
$$
{\frak L_R} := \{f \in L^1_{loc} \mid \|\chi_{|t|\leq R}(t) f(t)\|_{L^2_t} + \|\chi_{|t|\geq R}(t) f(t)\|_{\langle t \rangle^{-\3/2} L^\infty_t} < \infty \}
$$
with the norm
\be\lb{eser}
\|f\|_{{\frak L_R}} := \|f\|_{L^2([-R, R])} + \|\langle t \rangle^{\3/2} f(t)\|_{L^\infty(\R \setminus [-R, R])}.
\ee
\end{definition}

For fixed $R$, this space can serve as a slightly weaker replacement for $\langle t \rangle^{-\3/2} L^\infty_t$, because, for each $R>0$, $\langle t \rangle^{-\3/2} L^\infty_t \subset {\frak L_R}$ and ${\frak L_R} \ast {\frak L_R} \subset \langle t \rangle^{-\3/2} L^\infty_t$:
$$
\|f_1 \ast f_2\|_{\frak L_R} \les \|f_1 \ast f_2\|_{\langle t \rangle^{-\3/2} L^\infty_t} \les \|f_1\|_{{\frak L_R}} \|f_2\|_{{\frak L_R}},\ \|f_1 \ast f_2\|_{\langle t \rangle^{-\3/2} L^\infty_t} \les \|f_1\|_{{\frak L_R}} \|f_2\|_{\langle t \rangle^{-\3/2} L^\infty_t}.
$$

The reason for replacing $\langle t \rangle^{-\3/2} L^\infty_t$ by ${\frak L_R}$ is that when $f \in {\frak L_R}$
$$
\lim_{\epsilon \to 0} \|\chi_{|t| \leq \epsilon}(t) f(t)\|_{{\frak L_R}} = 0,
$$
so $f$ is well approximated by discarding some small interval near zero in the ${\frak L_R}$ norm.

This is not the case in $\langle t \rangle^{-\3/2} L^\infty_t$, so this $L^\infty$-based space is less suitable for proving complete continuity.

Any $\frak L_R$ with $R>0$ is appropriate. From here on, we shall take $R=1$.

%
%
%
%
%

\begin{lemma}\lb{lemma_compact} Assume that the generator $A$ satisfies conditions C1-C5. For $v_1, v_2 \in C_y(L^2_x \cap L^\infty_x)$ and any $\epsilon>0$, there exist $N$ and $f_n(t)$, $g_n(y)$, $\tilde g_n(y)$, $h_n(y)$, $\tilde h_n(y)$, $1 \leq n \leq N$, such that
$$
\|\chi_{t \geq 0}(t) v_2 e^{it(-\Delta+iA)} v_1 - \sum_{n=1}^N f_n(t) g_n(y_1) \tilde g_n(y_2) [h_n(x) \otimes \tilde h_n(x)]\|_{(\frak L_1)_t (L^\infty_{y_1} L^1_{y_2} \cap L^\infty_{y_2} L^1_{y_1}) \B(L^2_x)} < \epsilon.
$$
Here $\frak L_1$ is defined by (\ref{eser}) for $R=1$. Same is true when $v_1, v_2 \in \v$, if one replaces ${\frak L_1}$ by $L^1$.

Consequently, for $v_1, v_2 \in \v$ and $\Im \lambda \leq 0$, the family of operators
$$
v_2(x, y) R^{iA}_0(\lambda) v_1(x, y) \in (L^\infty_{y_1} L^1_{y_2} \cap L^\infty_{y_2} L^1_{y_1}) \B(L^2_x) \subset \B(L^p_y L^2_x)
$$
is norm-continuous as a function of $\lambda$, goes to zero in this norm as $\lambda \to \infty$, and can be approximated by finite-rank operators for $\Im \lambda \leq 0$ in this norm.
\end{lemma}

This lemma does not prove the compactness of the time-dependent operator $\chi_{t \geq 0}(t) v_2 e^{it(-\Delta+iA)} v_1$, when it acts by convolution in $t$. Indeed, this would be impossible, due to its time-translation invariance. However, it proves that $v_2 R^{iA}_0(\lambda) v_1$ is compact on $L^p_y L^2_x$, $1 \leq p \leq \infty$.

Note that the functions $f_n(t)$ can all be chosen to be the characteristic functions of bounded intervals or of the form $\chi_{t \geq R} \langle t \rangle^{-\3/2}$.

\begin{proof}
The weights $v_1, v_2 \in \v$ can be approximated in norm by functions in $C_y (L^1_x \cap L^\infty_x)$. This can be done because $v_1$ and $v_2$ are uniformly in $L^{\3, 1}$; it suffices to keep a finite range of dyadic atoms in $x$, uniformly over $y \in Y$.

In such a case, the operator $\chi_{t \geq 0}(t) v_2 e^{it(-\Delta+iA)} v_1$ will have a definite rate of decay of $\langle t \rangle^{-\3/2}$ and will belong to
$$
\chi_{t \geq 0}(t) v_2 e^{it(-\Delta+iA)} v_1 \in \langle t \rangle^{-\3/2} L^\infty_t (L^\infty_{y_1} L^1_{y_2} \cap L^\infty_{y_2} L^1_{y_1}) \B(L^2_x).
$$

We can discard the initial and final portions for a good approximation in the $L^1_t$ norm, so it suffices to consider the restriction to $t \in [\epsilon, R]$:
$$
\chi_{t \geq 0}(t) v_2 e^{it(-\Delta+iA)} v_1 \sim \chi_{[\epsilon, R]}(t) v_2 e^{it(-\Delta+iA)} v_1.
$$
For $t>0$, $e^{-tA}(y_1, y_2) \in L^\infty_{y_1} L^1_{y_2} \cap L^\infty_{y_2} L^1_{y_1}$ is norm-continuous as a function of $t$, so for $t \in [\epsilon, R]$ it is uniformly continuous. Then
$$
\chi_{[\epsilon, R]}(t) v_2 e^{it(-\Delta+iA)} v_1 \sim \chi_{[\epsilon, R]}(t) \sum_{n=1}^N \chi_n(t) e^{-t_n A}(y_1, y_2) [v_2(x_1, y_1) e^{-it\Delta} v_1(x_2, y_2)].
$$
Next, for each fixed $t_n>0$, the integral kernel $e^{-t_n A}(y_1, y_2)$ can be approximated in the $L^\infty_{y_2} L^1_{y_1} \cap L^\infty_{y_1} L^1_{y_2}$ norm by the tensor product of some bounded, compactly supported functions:
$$
e^{-t_n A}(y_1, y_2) \sim \sum_{m=1}^N g_m^*(y_1) \tilde g_m^*(y_2).
$$
Since $v_1(\cdot, y)$ and $v_2(\cdot, y)$ are uniformly continuous in $y$ on compact sets, we can approximate them by finitely many samples: for each $t_n$, after further subdividing $\supp g_m^*$ and $\supp \tilde g_m^*$,
$$
e^{-t_n A}(y_1, y_2) v_2(x_1, y_1) e^{-it\Delta} v_1(x_2, y_2) \sim \sum_{m=1}^N g_m(y_1) \tilde g_m(y_2) v_2(x_1, y_1^m) e^{-it\Delta} v_1(x_2, y_2^m).
$$
Finally, for each $y_1, y_2 \in Y$, $v_1(\cdot, y_1)$ and $v_2(\cdot, y_2)$ can be approximated in $L^2_x \cap L^\infty_x$ by functions with compact support. Then for $t \in [\epsilon, R]$, $v_2(x_1, y_1^m) e^{-it\Delta} v_1(x_2, y_2^m)$ is norm-continuous in $\B(L^2_x)$, hence uniformly continuous, so
$$
\chi_{[\epsilon, R]}(t) v_2(x_1, y_1^m) e^{-it\Delta} v_1(x_2, y_2^m) \sim \sum_{n=1}^N \chi_n(t) v_2(x_1, y_1^m) e^{-it_n\Delta} v_1(x_2, y_2^m).
$$
Each operator $v_2(x_1, y_1^m) e^{-it_n\Delta} v_1(x_2, y_2^m)$ is Hilbert--Schmidt for $t_n>0$, so it is compact and can be approximated by finite-rank operators in $\B(L^2_x)$.

In the $({\frak L_1})_t$ norm we can still discard the initial interval $[0, \epsilon]$ whenever $\epsilon<1$, so the previous reasoning works for $t \in [0, R_0]$ and any $R_0>0$. Still, we need to find a separate approximation as $t \to \infty$.  Assuming that
$$
\lim_{t \to \infty} e^{-tA}(y_1, y_2) = e^{-\infty A}(y_1, y_2)
$$
exists in the $L^\infty_{y_2} L^1_{y_1} \cap L^\infty_{y_1} L^1_{y_2}$ norm, then
$$
\chi_{[R, \infty)}(t) v_2 e^{it(-\Delta+iA)} v_1 \sim \chi_{[R, \infty)}(t) e^{-\infty A}(y_1, y_2) [v_2(x_1, y_1) e^{-it\Delta} v_1(x_2, y_2)].
$$
Following the same steps as above, since when restricted to a bounded domain in $\R^d$
$$
e^{-it\Delta} \sim (4\pi i t)^{-\3/2} + O(t^{-(\3+2)/2}),
$$
we get that
$$
\chi_{[R_0, \infty)}(t) v_2(x_1, y_1^m) e^{-it\Delta} v_1(x_2, y_2^m) \sim \chi_{[R_0, \infty)}(t) (4\pi i t)^{-\3/2} v_2(x_1, y_1^m) v_1(x_2, y_2^m)
$$
for sufficiently large $R_0$. For each $m$ this is a simple function, so overall it still is a simple function.

Since all previous approximations can be achieved with any desired precision, the first conclusion is proved.

The conclusion pertaining to the resolvent $v_2 R^{iA}_0(\lambda) v_1$ follows immediately from this as well, via the standard argument used in the proof of the Riemann--Lebesgue lemma.
\end{proof}

\subsection{The scalar Kato--Birman operator}
Consider the perturbed resolvent
$$
R^{iA}_V(\lambda) := (-\Delta_x + i A + V -\lambda)^{-1}
$$
and recall that $V = v_1 v_2$, $v_1=|V|^{1/2}$, $v_2=|V|^{1/2} \sgn V$.

If the Kato--Birman operator $KB(\lambda)$, see (\ref{katobirman}), is invertible for some $\lambda \in \C$, then $R^{iA}_V(\lambda)$ is a bounded operator between the appropriate spaces, by the symmetric resolvent identity:
\be\lb{sym_res}
R^{iA}_V(\lambda) = R^{iA}_0(\lambda) - R^{iA}_0(\lambda) v_1 KB(\lambda) v_2 R^{iA}_0(\lambda).
\ee
Also,
\be\lb{inverse_res}
I - v_2 R_V^{iA}(\lambda) v_1 = KB(\lambda)^{-1} = (I + v_2 R_0^{iA}(\lambda) v_1)^{-1}.
\ee
In the next lemma, we invert the Kato--Birman operator by means of the Fredholm alternative, Lemma \ref{fredholm}, and use this to prove some useful properties of $R^{iA}_V$.
%

\begin{lemma}\lb{invertibility}
Let $v_1, v_2 \in \v$. Then $v_2 R^{iA}_V(\lambda) v_1$ belongs to $[L^\infty_{y_1} L^1_{y_2} \cap L^\infty_{y_2} L^2_{y_1}] \B(L^2_x)$, hence is bounded on $L^p_y L^2_x$, $1 \leq p \leq \infty$, uniformly for $\Im \lambda \leq 0$:
\be\lb{RiA}
\sup_{\Im \lambda \leq 0} \|v_2 R^{iA}_V(\lambda) v_1\|_{[L^\infty_{y_1} L^1_{y_2} \cap L^\infty_{y_2} L^1_{y_1}] \B(L^2_x)} < \infty.
\ee
\end{lemma}

For the resolvent $R_V^{iA}$ itself, by the symmetric resolvent identity (\ref{inverse_res}), (\ref{RiA}) implies that
$$
\sup_{\Im \lambda \leq 0} \|R^{iA}_V(\lambda)\|_{[L^\infty_{y_1} L^1_{y_2} \cap L^\infty_{y_2} L^1_{y_1}] \B(L^{6/5, 2}_x, L^{6, 2}_x)} < \infty.
$$

Note that $[L^\infty_{y_1} L^1_{y_2} \cap L^2_{y_1} L^2_{y_2}] \subset \B(L^p_y)$ for $1 \leq p \leq 2$ and $[L^\infty_{y_1} L^1_{y_2} \cap L^\infty_{y_2} L^1_{y_1}] \subset \B(L^p_y)$ for $1 \leq p \leq \infty$.


\begin{proof}
For each $\lambda \in \C$, $\Im \lambda \leq 0$, if the Kato--Birman operator $KB(\lambda)$ is invertible in $[L^\infty_{y_1} L^1_{y_2} \cap L^\infty_{y_2} L^1_{y_1}] \B(L^2_x) \oplus \C I$, then by (\ref{sym_res}) and (\ref{inverse_res})
$$
v_2 R_V^{iA}(\lambda) v_1 \in [L^\infty_{y_1} L^1_{y_2} \cap L^\infty_{y_2} L^1_{y_1}] \B(L^2_x) \subset \B(L^p_y L^2_x).
$$

Because the perturbation is continuous and goes to zero in norm, $KB(\lambda)$ is always invertible for large $\lambda$ and its inverse approaches the identity as $\lambda \to \infty$. Furthermore, if it is invertible for all $\Im \lambda \leq 0$, then the inverse will also be norm-continuous. Hence the inverse will be bounded, proving (\ref{RiA}).

It remains to be shown that $KB(\lambda)$ is invertible for all $\Im \lambda \leq 0$. The other possibility would be that $KB(\lambda)$ is not invertible for some $\Im \lambda \leq 0$.

Since $v_2 R^{iA}(\lambda) v_1$ is a completely continuous perturbation of the identity, in particular it is compact in $\B(L^p_y L^2_x)$.

By Fredholm's alternative, Lemma \ref{fredholm}, its non-invertibility implies the existence of some $f \in L^p_y L^2_x$, $f \ne 0$, $1 \leq p \leq \infty$, for which
\be\lb{eqf}
f = - v_2 R_0^{iA}(\lambda-i0) v_1 f.
\ee
In the following we show that such $f$ cannot exist for any value of $p \in [1, 2]$.

The (standard) proof depends on the fact that $V$ is real-valued, so $\sgn V = \pm 1$, hence $v_2 \sgn V = v_1$. 

Suppose that $f \in L^p_y L^2_x$, $f \ne 0$, solves (\ref{eqf}).

If $p<2$, we show that there exists some $\epsilon>0$ such that, whenever $f \in L^p_y L^2_x$, then $f \in L^{\tilde p}_y L^2_x$ as well, where
$$
\frac 1 {\tilde p} = \frac 1 p - \epsilon
$$
when $p<1/\epsilon$ and $\tilde p = \infty$ when $p \geq 1/\epsilon$. Then we can reach some value above $2$ in finitely many steps, hence $2$ as well, by interpolation.

We represent $R_0^{iA}(\lambda-i0)$ as the Fourier transform in $t$ of the time evolution. For
$$
T(t) = \chi_{t>0}(t) e^{it(-\Delta_x+iA)},
$$
its Fourier transform is given by
$$
i \widehat T(\lambda) = (i\chi_{t>0}(t) e^{it(-\Delta_x+iA)})^\wedge(\lambda) = R_0^{iA}(\lambda-i0).
$$
For arbitrary $\epsilon>0$, we split this kernel into two pieces, such that the Fourier transform of the second piece is smoothing in $y$ and the Fourier transform of the first piece is small:
$$
i\chi_{t>0}(t) T(t) = i\chi_{t>0}(t) \eta(t/\epsilon) T(t) + i\chi_{t>0}(t) (1-\eta(t/\epsilon)) T(t),
$$
where $\eta$ is a standard cutoff function: $\eta(t)=1$ for $|t| \leq 1$ and $\eta(t)=0$ for $|t| \geq 2$.

Then the same decomposition holds for its Fourier transform: when $\Im \lambda \leq 0$,
$$
R_0^{iA}(\lambda-i0) = [i \chi_{t>0} \eta(t/\epsilon) T(t)]^\wedge(\lambda) + [i \chi_{t>0} (1-\eta(t/\epsilon)) T(t)]^\wedge(\lambda).
$$

We rewrite the equation of $f$ as
$$
(I + v_2 [i \chi_{t>0} \eta(t/\epsilon) T(t)]^\wedge(\lambda) v_1) f = -v_2 [i \chi_{t>0} (1-\eta(t/\epsilon)) T(t)]^\wedge(\lambda) v_1 f.
$$

The last term is smoothing in the $y$ variable. Indeed, assuming that
$$
\||e^{-tA}|\|_{L^p \to L^{\tilde p}} = \|e^{-tA}\|_{L^p \to L^{\tilde p}} < \infty
$$
is finite for each $t \geq 0$ (and hence non-increasing), then
$$
\|v_2 i \chi_{t>0} (1-\eta(t/\epsilon)) T(t) v_1\|_{L^1_t \B(L^{\tilde p}_y L^2_x)} \les C(\epsilon) \|v_1\|_\v \|v_2\|_\v,
$$
so its Fourier transform in $t$ has finite norm, uniformly in $\lambda$:
$$
\|v_2 [i \chi_{t>0} (1-\eta(t/\epsilon)) T(t)]^\wedge(\lambda) v_1\|_{\B(L^p_y L^2_x, L^{\tilde p}_y L^2_x)} \leq C(\epsilon) \|v_1\|_\v \|v_2\|_\v < \infty.
$$

The other piece is not smoothing, but it becomes small in norm as $\epsilon \to 0$, due to Lemma \ref{bdd}.
%
%
%
Consequently, for sufficiently small $\epsilon$, the inverse can be realized as a converging geometric series:
$$
\big(I - v_2 [i\chi_{t>0}(t) \eta(t/\epsilon) e^{it(-\Delta_x+iA_y)}]^\wedge(\lambda)\, v_1\big)^{-1}
$$
is a bounded operator on $L^{\tilde p}_y L^2_x$ for any $\tilde p \in [1, \infty]$. Then
$$
f = -(I + v_2 [i\chi_{t>0}(t) \eta(t/\epsilon) e^{it(-\Delta_x+iA)}]^\wedge(\lambda) v_1)^{-1} v_2 (i\chi_{t>0}(t) (1-\eta(t/\epsilon)) e^{it(-\Delta_x+iA)})^\wedge(\lambda) v_1 f
$$
must be in $L^{\tilde p}_y L^2_x$.

After several such bootstrapping steps, we get that $f \in L^2_y L^2_x$.

Then by (\ref{eqf}) the pairing
$$
\langle \sgn V f, f \rangle = -\langle R_0^{iA}(\lambda-i0) v_1 f, v_1 f \rangle
$$
is finite and real-valued, because $\sgn V = \pm 1$ is real-valued and $(\sgn V) v_2 = v_1$. However, then
\be\lb{spectral_projection}\begin{aligned}
0 = -\Im \langle R_0^{iA}(\lambda-i0) V g, V g \rangle &= c \chi_{\lambda \geq 0} \lambda^{(\3-2)/2} \int_{S^{\3-1}} |\widehat {P_0(Vg)}(\sqrt \lambda \omega)|^2 \dd\omega + \\
&+ \langle A P_{>0} R_0^{iA}(\lambda-i0) v_1 f, P_{>0} R_0^{iA}(\lambda-i0) v_1 f \rangle,
\end{aligned}\ee
for some constant $c>0$, where $P_0$ is the projection on the zero energy states of $A$ and $P_{>0} = I - P_0$. The first term disappears when $\lambda < 0$ because the spectrum of the free Laplacian is supported on $\lambda \geq 0$.

To prove (\ref{spectral_projection}), we need no particular assumption about the self-adjoint operator $A \geq 0$ or its spectrum. Let $P_{\geq \lambda}$, for $\lambda \geq 0$, be spectral projections for $A$ and $E_\lambda = dP_{\geq \lambda}$ be the spectral density of $A$. Since $A$ and $-\Delta$ commute,
$$\begin{aligned}
-\Im \langle R_0^{iA}(\lambda-i0) v_1 f, v_1 f \rangle &= -\int_0^\infty \Im \langle R_0^{iA}(\lambda-i0) \dd P_{\geq \eta} v_1 f, v_1 f \rangle \\
&= -\Im \langle R_0^{iA}(\lambda-i0) P_0 v_1 f, P_0 v_1 f \rangle - \int_{(0, \infty)} \Im \langle R_0^{iA}(\lambda-i0) \dd P_{\geq \eta} v_1 f, P_{\geq \eta} v_1 f \rangle. 
\end{aligned}$$

The first therm is given by the spectral density of the free Laplacian, where $h=P_0 v_1 f$:
$$
-\Im \langle R_0^{iA}(\lambda-i0) h, h \rangle = -\Im \langle R_0(\lambda-i0) h, h \rangle = c \langle E_\lambda(-\Delta) h, h \rangle = c \chi_{\lambda \geq 0} \lambda^{(\3-2)/2} \int_{S^{\3-1}} |\widehat h(\sqrt \lambda \omega)|^2 \dd \omega.
$$
Here $E_\lambda(-\Delta)$ is the spectral density of the Laplacian, which has the explicit formula given here in terms of the Fourier transform.

Concerning the higher energy states, the other term is the integral of
$$
\begin{aligned}
\Im \langle R_0^{iA}(\lambda-i0) E_\eta v_1 f, E_\eta v_1 f\rangle
&= \Im \langle R_0(\lambda-i\eta) E_\eta v_1 f, E_\eta v_1 f\rangle \\
&= \frac 1 {2i}
\langle \big(R_0(\lambda-i\eta) - R_0(\lambda+i\eta)\big) E_\eta v_1 f, E_\eta v_1 f \rangle \\
&= - \langle R_0(\lambda+i\eta) \eta R_0(\lambda-i\eta) E_\eta v_1 f, E_\eta v_1 f \rangle
\end{aligned}
$$
by the resolvent identity. This gives the second term in (\ref{spectral_projection}).

For any $\epsilon>0$, since $A P_{\geq \epsilon}$ is a positive, coercive operator, it follows from (\ref{spectral_projection}) that
$$
P_{\geq \epsilon} R_0^{iA}(\lambda-i0) v_1 f = 0,
$$
so $P_{>0} R_0^{iA}(\lambda-i0) v_1 f = 0$.

We infer that $R_0^{iA}(\lambda-i0) v_1 f = P_0 R_0^{iA}(\lambda-i0) v_1 f$. Under our hypotheses, $A$ has at most one zero energy bound state, call it $h$, hence $R_0^{iA}(\lambda-i0) v_1 f$ is of the form $g(x) h(y)$ for all $x$.

Let
$$
g(x) h(y) = R_0^{iA}(\lambda-i0) v_1 f.
$$
Then by (\ref{eqf})
$$
g (x) h(y) = - R_0^{iA} V(x, y) (g(x) h(y))
$$
so $g(x)$ is a distributional solution of the equation
$$
(-\Delta+V(x, y)-\lambda) g(x) = 0
$$
for almost all $y \in \supp h$.

Due to the nontrivial-randomness Assumption \ref{nontrivial}, by giving different values to $y \in Y$ we get that $g(x)=0$ on an open set. By bootstrapping in the equation we also obtain that $g \in L^\3_x \cap L^\infty_x$. Furthermore,
$$
(-\Delta +1) g = V g + (\lambda +1) g,
$$
which implies that $g \in W^{2, \3}_x$. Moreover, $g$ solves the equation
$$
(-\Delta + V) g = \lambda g.
$$
Due to the unique continuation property of this equation (see \cite{carleman} or \cite{kochtataru} for a proof), since $g$ vanishes on an open set it must be everywhere zero.

But then $f = -v_2 g$ must also be zero. The contradiction shows that $KB(\lambda)$ must be invertible in $\B(L^p_y L^2_x)$, $1 \leq p \leq 2$.

For $p>2$, we simply perform the same reasoning for the adjoint of the operator we wish to understand,
$$
KB^*(\lambda) = -v_1 R^{-iA}_0(\lambda+i0) v_2,
$$
in the upper half-plane, and use duality.

Next, we show that the inverse has the specified structure. Due to Lemma \ref{lemma_compact}, $v_2 R^{iA}_0(\lambda) v_1$ is completely continuous and can be approximated by finite-rank operators in the $[L^\infty_{y_2} L^1_{y_1}]\B(L^2_x)$ norm.

Furthermore, $[L^\infty_{y_2} L^1_{y_1}]\B(L^2_x)$ has property (\ref{proper}) within $\B(L^1_y L^2_x)$:
$$
\big\{[L^\infty_{y_2} L^1_{y_1}]\B(L^2_x)\big\} \B(L^1_y L^2_x) \big\{[L^\infty_{y_2} L^1_{y_1}]\B(L^2_x)\big\} \subset [L^\infty_{y_2} L^1_{y_1}]\B(L^2_x).
$$
Consequently,
$$
KB(\lambda)^{-1} \in [L^\infty_{y_2} L^1_{y_1}]\B(L^2_x)
$$
and a similar reasoning shows that $KB(\lambda)^{-1} \in [L^\infty_{y_1} L^1_{y_2}]\B(L^2_x)$ as well.
%
\end{proof}

This concludes our spectral analysis of the scalar averaged equation (\ref{average}).

\subsection{Wiener's theorem} When proving dispersive estimates in $L^1$ in time, we employ an instance of Wiener's theorem; see \cite{bec} and, on Banach spaces, \cite{bego}. For convenience, we state this instance of Wiener's theorem below, exactly as we shall use it in the proof.

Given a Banach space $X$, let $\V_X=\B(X, L^1_t X)$.

Define the Fourier transform on $\V_X$ as follows:
$$
\widehat T(\lambda) f = \int_{-\infty}^\infty T(\rho) f e^{-i\rho\lambda} \dd \rho
$$
and the algebra operation on $\V_X$ by
$$
[(T_1 \ast T_2)f](\rho_0) = \int_{-\infty}^\infty T_1(\rho) T_2(\rho_0-\rho)f \dd \rho.
$$
If $T \in \V_X$, then for each $\lambda \in \R$ $\widehat T \in \B(X)$ and
$$
\sup_{\lambda \in \R} \|\widehat T(\lambda)\|_{\B(X)} \leq \|T\|_{\V_X}.
$$
The same holds for $\Im \lambda \leq 0$, i.e.\;in a half-plane, assuming that $\supp T = \{t: \exists x \in X\ T(x)(t) \ne 0\} \subset [0, \infty)$.

Furthermore, for any such $\lambda$
$$
\widehat{T_1 \ast T_2}(\lambda) = \widehat T_1(\lambda) \widehat T_2(\lambda).
$$

Let $\widehat {\V_X} \subset C_\lambda \B(X)$ be the algebra of Fourier transforms of kernels in $\V_X$.

$\V_X$ is a Banach algebra under convolution and consequently $\widehat {V_X}$ is one under pointwise composition of operators, but they lack a unit element. To remedy this, we adjoin to $\V_X$ the operator $\one \in \B(X, \mc M_t X)$ given by
$$
\one(x) := \delta_0(t) x,
$$
thus forming the Banach space $\ov \V_X := \V_X \oplus \C I \subset \B(X, L^1_t(m+\delta_0)X)$.

Here $L^1_t(m+\delta_0)X$ is the space of $X$-valued functions which are Lebesgue integrable with respect to the measure $m+\delta_0$, where $m$ is the Lebesgue measure on $\R$ and $\delta_0$ is the Dirac measure at $0$.

Clearly $\one$ is the identity under convolution. Its Fourier transform is constantly $I \in \B(X)$: $\widehat \one(\lambda) = I$ for every $\lambda$, so $\ov{\widehat\V_X} = \widehat {V_X} \oplus \C I$ is also a unital algebra.

Each element $T \in \V_X$ determines a bounded operator on $L^1_t X$ by acting as follows:
\be\lb{conv_act}
[T F](\rho_0) = \int_{-\infty}^\infty T(\rho) F(\rho_0-\rho) \dd \rho.
\ee

\begin{theorem} \label{thm:Wiener'}
Suppose $T$ is an element of $\V_X$ with the following properties:
\begin{align}
\label{translation'} &\lim_{\delta \to 0}
  \|T(\rho) - T(\rho-\delta)\|_{\V_X} = 0\\
\label{locality'} &\lim_{R \to \infty}
   \|\chi_{|\rho| \ge R}T\|_{\V_X} = 0.
\end{align}
If $I + \widehat{T}(\lambda)$ is an invertible element of $\B(X)$ for every
$\lambda \in \R$, then $\one + T$ possesses an inverse in
$\overline{\V}_X$ of the form $\one + S$, with $S \in \V_X$.
\end{theorem}

The proof of this statement can be found in \cite{bec} or \cite{bego}.

Note that Theorem \ref{thm:Wiener'} does not guarantee a specific rate of decay, but only integrability in $t$. However, in many cases we need a specific uniform rate of decay. This corresponds to using a Beurling-type algebra instead of a Wiener algebra.

For any Banach algebra $\mc A \subset L^1_t$, such as ${\frak L_1}$ defined by (\ref{eser}), $\mc A_t \B(X)$ is a subalgebra of $\V_X$.

Given a Banach space $X$, for any $p>1$ let
$$
\W_X:=\langle t \rangle^{-p} L^\infty_t \B(X).
$$
For any $T \in \W_X$, the map $f \mapsto T(t)f$ is an element of $\V_X$ that can be identified with $T$, making $\W_X$ a subalgebra of $\V_X$. The algebra operation reduces to
$$
[T_1 \ast T_2](\rho_0) = \int_{-\infty}^\infty T_1(\rho) T_2(\rho_0-\rho) \dd \rho.
$$

Using this sort of pointwise, as opposed to average, decay rate at infinity poses its own problems. The main issue is that for a typical element $T \in W_X$
$$
\lim_{R \to \infty} \|\chi_{[R, \infty)}(t) T(t)\|_{\W_X} \ne 0.
$$
Indeed, for this limit to be zero, one would need a $o(t^{-p})$ rate of decay instead of $O(t^{-p})$. Below we explain how to overcome this difficulty.

Again, $\ov \W_X := \W_X \oplus \C \one$ is a unital algebra under convolution.

Let $\widehat{\W_X}$ be the algebra of Fourier transforms of kernels in $\W_X$. In this paper we use a novel version of Wiener's Theorem, Theorem \ref{thm:Wiener}, to get pointwise decay estimates in this setting. Note that we do not need condition (\ref{locality'}) in this setting, since it is in some sense implied by the definition of $\W_X$. The proof is presented below.

It can also be useful to know when the inverse belongs to some specific Banach subalgebra $\ov \W$ of $\ov \W_X$, such as for example the subalgebra obtained by replacing $\B(X)$ with $\mc K(X)$, the space of compact operators, or with some other subspace of operators. Such properties are encapsulated in the following remarks.

\begin{proposition}\lb{subalg} Consider a Banach subalgebra ${\frak X} \subset \B(X)$ and for some fixed $p>1$ let
$$
\W = \W_{{\frak X}} = \langle t \rangle^{-p} L^\infty_t {\frak X}.
$$
If, in addition, 
$T \in \W$ has property (\ref{translation_high}) within $\W$,
$$
\lim_{\delta \to 0} \|T^N(t) - T^N(t-\delta)\|_{\W} = 0,
$$
and $(I + \widehat T(\lambda))^{-1} \in {\frak X}$ for all $\lambda \in \R$, then $S = (\one + T)^{-1} - \one \in \W$.

For $\tilde p > 1$ and an operator space $\tilde {\frak X}$, consider a Banach algebra $\tilde \W = \W \cap \langle t \rangle^{-\tilde p} L^\infty_t \tilde {\frak X}$, which satisfies the following generalized Leibniz inequality with respect to $\W$:
\be\lb{leib}
\|T_1 \ast T_2\|_{\tilde \W} \les \|T_1\|_{\tilde \W} \|T_2\|_\W + \|T_1\|_{\W} \|T_2\|_{\tilde \W}.
\ee
If $T \in \tilde \W$ and $(I+\widehat T(\lambda))^{-1} \in \tilde \W$ for all $\lambda \in \R$ as well, then $S = (\one+T)^{-1}-\one \in \tilde \W$ and
$$
\|S\|_{\tilde \W} \leq C \|T\|_{\tilde W},
$$
with a constant that depends only on $\|T\|_{\W}$ and on $\sup_{\lambda \in \R} \|(I+\widehat T(\lambda))^{-1}\|_{\tilde \W}$.

Furthermore, consider the function algebra ${\frak L_1}$, defined by (\ref{eser}), with the norm
$$
\|f\|_{{\frak L_1}} := \|f\|_{L^2([-1, 1])} + \|\langle t \rangle^{\3/2} f(t)\|_{L^\infty(\R \setminus [-1, 1])}.
$$
Suppose that $T$ is completely continuous in $({\frak L_1})_t {\frak X}$: for every $\epsilon>0$ there exist some functions $f_n \in {\frak L_1}$ and operators $g_n \in {\frak X}$, $1 \leq n \leq N$, such that
$$
\Big\|T - \sum_{n=1}^N f_n(t) g_n\Big\|_{(\frak L_1)_t {\frak X}} < \epsilon.
$$
Then $S = (\one+T)^{-1}-\one$ is also completely continuous in $({\frak L_1})_t {\frak X}$.
\end{proposition}
All these operators act by convolution in the $t$ variable, following (\ref{conv_act}).

This notion of complete continuity does not imply compactness of $T$ itself, or of $S$. Indeed, due to the fact that the action of $T$ commutes with translations, i.e.\;$T(f(\cdot+t_0)) = T(f)(\cdot+t_0)$, $T \ne 0$ cannot be compact when acting by convolution on Banach spaces of the form $\langle t \rangle^{-\alpha} L^q_t$.

However, if the operators $g_n \in \frak X$ in the approximation above are compact, then both $\widehat T(\lambda)$ and $\widehat S(\lambda) \in \B(X)$ are compact for $\Im \lambda \leq 0$.

\begin{proof}[Proof of Theorem \ref{thm:Wiener} and Proposition \ref{subalg}]
We prove that $(I + \widehat{T}(\lambda))^{-1}-I$ is the Fourier transform
of an element  $S \in \W$, under the conditions listed in Proposition \ref{subalg}.

Obviously $\B(X)$ is a subalgebra of itself, so the proof will apply to the whole space, $\W = \W_X$, as well, thus proving Theorem \ref{thm:Wiener}.

%

Let $\eta \in C^\infty(\R)$ be a standard cutoff function, $\eta(\lambda)=1$ for $|\lambda| \leq 1$ and $\eta(\lambda)=0$ for $|\lambda| \geq 2$.  For $L \in \R$, the restriction
$(1 - \eta(\lambda/L))\widehat{T}(\lambda)$ is the Fourier transform of
$$
S_L(\rho) = \big(T - L\widecheck{\eta}(L\,\cdot\,) * T\big)(\rho) =
\int_\R L \widecheck{\eta}(L\sigma) [T(\rho) - T(\rho-\sigma)]\,d\sigma.
$$

Since convolution with a sufficiently rapidly decaying function preserves $\W$, clearly $S_L \in \W$.

If condition (\ref{translation}) is satisfied within $\W$, then
the $\W$ norm of the right-hand side vanishes as $L \to \infty$, because
$$
\|S_L(\rho)\|_{\W} \leq \int_\R L |\widecheck{\eta}(L\sigma)| \, \|T(\rho) - T(\rho-\sigma)\|_{\W} \,d\sigma = \int_\R |\widecheck{\eta}(\sigma)| \, \|T(\rho) - T(\rho-\sigma/L)\|_{\W} \,d\sigma.
$$
Thus, for some large number $L$,
\be\lb{serie'}
\sum_{k=0}^\infty (-S_L(\rho))^k = \one + \sum_{k=1}^\infty (-S_L(\rho))^k
\ee
is a convergent series in $\ov \W = \W \oplus \C \one$, where the powers of $-S_L(\rho)$ denote repeated convolution in the $\rho$ variable. The Fourier transform of $\sum_{k=0}^\infty (-S_L(\rho))^k$ as a function of $\rho$ is
\begin{equation*}
\sum_{k=0}^\infty (-1)^k \Big(\big(1 - \eta(\lambda/L))\widehat{T}(\lambda)\Big)^k
= \Big(I + \big(1-\eta(\lambda/L)\big)\widehat{T}(\lambda)\Big)^{-1}
\end{equation*}
which agrees with $(I + \widehat{T}(\lambda))^{-1}$ for all $\lambda > 2L$. Therefore
$$
[\big(1-\eta(\lambda/2L)\big) (I + \widehat{T}(\lambda))^{-1}]^\vee = [\big(1-\eta(\lambda/2L)\big) (I + (1-\eta(\lambda/L))\widehat{T}(\lambda))^{-1}]^\vee \in \ov \W
$$
and more specifically
\be\lb{infinit}
[\big(1-\eta(\lambda/2L)\big) (I + \widehat{T}(\lambda))^{-1}]^\vee - \one \in \W.
\ee

Assuming the weaker condition (\ref{translation_high}) instead of (\ref{translation}) only guarantees in like manner that
$$
[\big(1-\eta(\lambda/2L)\big) (I \pm \widehat{T}^N(\lambda))^{-1}]^\vee - \one \in \W,
$$
since $\widehat{T^N} = \widehat T^N$, but then one can write
$$
[\big(1-\eta(\lambda/2L)\big) (I + \widehat{T}(\lambda))^{-1}]^\vee = [\big(1-\eta(\lambda/2L)\big) (I - (-1)^N \widehat{T}^N(\lambda))^{-1}]^\vee \ast \bigg(\sum_{k=0}^{N-1} (-1)^k T^k\bigg) \in \ov \W. 
$$
So the more general condition (\ref{translation_high}) holding within $\W$ also implies (\ref{infinit}).

We then decompose $(I + \widehat{T}(\lambda))^{-1} - I$ into
\begin{align*}
(I + \widehat{T}(\lambda))^{-1} - I &= \eta(\lambda/2L) (I + \widehat{T}(\lambda))^{-1} +
\big(1-\eta(\lambda/2L)\big) (I + \widehat{T}(\lambda))^{-1} \\
&= \eta(\lambda/2L) (I + \widehat{T}(\lambda))^{-1} + (1 - \eta(\lambda/2L))
\big(I + (1-\eta(\lambda/L))\widehat{T}(\lambda)\big)^{-1} - I.
\end{align*}

Knowing by (\ref{infinit}) that the second term belongs to $\widehat  {\W}$, we next prove that the first term also belongs to $\widehat {\W}$.

We construct a local inverse for $I + \widehat{T}(\lambda)$ in the neighborhood
of
any $\lambda_0 \in \R$.  With no loss of generality let $\lambda_0 = 0$ and let $A_0 = I + \widehat{T}(0)$.
One can write $\eta(\lambda/\epsilon)(I + \widehat{T}(\lambda) - A_0)$
as the Fourier transform of
\be\lb{seps}\begin{aligned}
S_\epsilon(\rho) &= \epsilon\widecheck{\eta}(\epsilon\,\cdot\,) * T(\rho) +
\epsilon\widecheck{\eta}(\epsilon\rho)(I-A_0)
\\
&= \int_\R \epsilon\big(\widecheck{\eta}(\epsilon(\rho-\sigma))-
\widecheck{\eta}(\epsilon\rho)\big) T(\sigma)\,d\sigma.
\end{aligned}\ee
The second equality used the fact that
$$
I - A_0 = -\widehat{T}(0) = -\int_\R T(\sigma)\,d\sigma.
$$
As we shall see below, $S_\epsilon$ is uniformly bounded in $\W$.

For any smooth function $\phi$ supported in $[-\frac{\epsilon}2,\frac{\epsilon}2]$, consider the series expansion
\be\lb{serie}\begin{aligned}
\phi(\lambda)(I + \widehat{T}(\lambda))^{-1}
&= \phi(\lambda)\big(A_0 + \eta(\lambda/\epsilon)(1 + \widehat{T}(\lambda) -A_0)\big)^{-1}
\\
&= \phi(\lambda) A_0^{-1}\big(I + \widehat{S_\epsilon}(\lambda)A_0^{-1}\big)^{-1} \\
&= \phi(\lambda) A_0^{-1} (I - \widehat{S_\epsilon}(\lambda)A_0^{-1})\sum_{k=0}^\infty (-1)^k \big(\widehat{S_\epsilon}(\lambda)A_0^{-1}\widehat{S_\epsilon}(\lambda)A_0^{-1}\big)^k.
\end{aligned}\ee
Here $A_0^{-1} = (I + \widehat T(0))^{-1} \in {\frak X}$ by hypothesis and $\widehat {S_\epsilon}$ is the Fourier transform of $S_\epsilon$, which will be in $L^\infty_\lambda {\frak X}$ because $S_\epsilon \in \W$.

To prove that $S_\epsilon \in \W$ is uniformly bounded, we estimate the integral in (\ref{seps}) differently in the two regions $|\rho|\geq 2|\sigma|$ and $|\rho|<2|\sigma|$ to get $|\rho|^{-p}$ decay. For each $\sigma \in \R$, one has that
$$
\widecheck{\eta}(\epsilon(\rho-\sigma)) - \widecheck{\eta}(\epsilon\rho) = -\epsilon \int_0^\sigma (\widecheck{\eta})'(\epsilon(\rho-\tilde \sigma)) \dd \tilde \sigma. 
$$
Since $\widecheck{\eta}$ and $(\widecheck{\eta})'$ have rapid decay,
$$
\widecheck{\eta}(\epsilon \rho) \les (\epsilon |\rho|)^{-1},\ (\widecheck{\eta})'(\epsilon\rho) \les (\epsilon |\rho|)^{-2}.
$$
Thus, the integrand in (\ref{seps}) is bounded by 
$|\rho|^{-1}$ and $|\sigma| |\rho|^{-2}$ for $|\rho| \geq 2|\sigma|$.

Consequently, for $p<2$
\be\lb{estp2}
\bigg\|\int_{|\sigma| \leq |\rho|/2} \epsilon\big(\widecheck{\eta}(\epsilon(\rho-\sigma))-
\widecheck{\eta}(\epsilon\rho)\big) T(\sigma)\,d\sigma\bigg\|_{{\frak X}} \les \|T\|_{\W} |\rho|^{-2} \int_{|\sigma| \leq |\rho|/2} |\sigma| \langle \sigma \rangle^{-p} \dd \sigma \les |\rho|^{-p} \|T\|_{\W}.
\ee
Note that for $p>2$ this expression still decays no faster than $|\rho|^{-2}$, so a different approach is needed.

In the other region $|\rho|<2|\sigma|$, we estimate each term in (\ref{seps}) separately:
$$
\bigg\|\int_{|\sigma| > |\rho|/2} \epsilon\widecheck{\eta}(\epsilon(\rho-\sigma)) T(\sigma)\,d\sigma\bigg\|_{{\frak X}} \les \|\epsilon\widecheck{\eta}(\epsilon(\rho-\sigma)\|_{L^1_{\sigma}} \sup_{|\sigma| > |\rho|/2} \|T(\sigma)\|_{{\frak X}} \les |\rho|^{-p} \|T\|_{\W}
$$
and
$$
\bigg\|\int_{|\sigma| > |\rho|/2} \epsilon
\widecheck{\eta}(\epsilon\rho) T(\sigma)\,d\sigma\bigg\|_{{\frak X}} \les |\rho|^{-1} \int_{|\sigma| > |\rho|/2} \|T(\sigma)\|_{{\frak X}} \,d\sigma \les |\rho|^{-p} \|T\|_{\W}.
$$
Putting these estimates together, we get that uniformly for all $\epsilon>0$
\be\lb{rhop}
\|S_\epsilon(\rho)\|_{{\frak X}} \les |\rho|^{-p} \|T\|_{\W}.
\ee

In addition, the integrand in (\ref{seps}) is of size $\epsilon$, so $\|S_\epsilon(\rho)\|_{{\frak X}} \les \epsilon \|T\|_{\W}$ and consequently
\be\lb{bdint}
\|\chi_{|\rho| \leq R}(\rho) S_\epsilon(\rho)\|_{\W} \les \sup_{|\rho| \leq R} \langle \rho \rangle^p \|S_\epsilon(\rho)\|_{{\frak X}} \les \langle R \rangle^p \epsilon \|T\|_{\W}.
\ee
For $\epsilon \leq 1$, considering the previous $|\rho|^{-p}$ decay estimate, we get that $S_\epsilon \in \W$ is uniformly bounded:
$$
\|S_\epsilon\|_{\W} \les \|T\|_{\W}.
$$

Since $A_0^{-1} \in {\frak X}$, it is also true that
$$
\|S_\epsilon A_0^{-1} \|_{\W} \les \|T\|_{\W} \|A_0^{-1}\|_{{\frak X}}
$$
and
$$
\|\chi_{|\rho| \leq R}(\rho) S_\epsilon(\rho)A_0^{-1} \|_{\W} \les \langle R \rangle^p \epsilon \|T\|_{\W} \|A_0^{-1}\|_{{\frak X}}.
$$

The inverse Fourier transform of the series in (\ref{serie}) is
\be\lb{inv_F}
\sum_{k=0}^\infty (-1)^k \big(S_\epsilon A_0^{-1} \ast S_\epsilon A_0^{-1}\big)^k,
\ee
where the powers now denote repeated convolution.

Neither $S_\epsilon$ nor $A_0^{-1} S_\epsilon$ goes to zero in norm as $\epsilon \to 0$; only their restrictions to bounded intervals do, due to (\ref{bdint}). However, their square does, making the series (\ref{inv_F}) converge in $\W$ for any sufficiently small $\epsilon < \epsilon_0$, where $\epsilon_0$ depends on $\|T\|_{\W}$ and on $\|(I+\widehat T(\lambda_0))^{-1}\|_{{\frak X}}$.

Indeed, note that, uniformly for bounded $\tilde T \in \W$,
\be\lb{prop\3}
\lim_{R \to \infty} \|(\chi_{|t| \geq R}(t) \tilde T(t)) \ast (\chi_{|t| \geq R}(t) \tilde T(t))\|_{\W} = 0.
\ee
This is shown by the following computation: for any $\tilde T \in \W$,
$$\begin{aligned}
&\langle \rho \rangle^p \|[(\chi_{|\sigma| \geq R}(\sigma) \tilde T(\sigma)) \ast (\chi_{|\sigma| \geq R}(\sigma) \tilde T(\sigma))](\rho)\|_{{\frak X}} \leq \\
&\leq \langle \rho \rangle^p \|\tilde T\|_{\W}^2 |(\chi_{|\sigma| \geq R} \langle \sigma \rangle^{-p}) \ast (\chi_{|\sigma| \geq R} \langle \sigma \rangle^{-p})(\rho)| \\
&\les \|\tilde T\|_{\W}^2 \big[|(\chi_{|\sigma| \geq R}) \ast (\chi_{|\sigma| \geq R} \langle \sigma \rangle^{-p})(t)| + |(\chi_{|\sigma| \geq R} \langle \sigma \rangle^{-p}) \ast (\chi_{|\sigma| \geq R})(\rho)|\big] \\
&\les \langle R \rangle^{1-p} \|\tilde T\|_{\W}^2.
\end{aligned}$$
Consequently, for any $\epsilon$ and $R>0$, writing
$$
S_\epsilon A_0^{-1} = \chi_{|\rho| \leq R}(\rho) S_\epsilon(\rho) + \chi_{|\rho| \geq R}(\rho) S_\epsilon(\rho)
$$
and using (\ref{bdint}), we get
\be\lb{mainest}\begin{aligned}
&\|S_\epsilon A_0^{-1} \ast S_\epsilon A_0^{-1}\|_{\W} \les \\
&\les \|\chi_{|\rho| \leq R}(\rho) S_\epsilon(\rho) A_0^{-1} \ast \chi_{|\rho| \leq R}(\rho) S_\epsilon(\rho) A_0^{-1}\|_{\W} + \|\chi_{|\rho| \leq R}(\rho) S_\epsilon(\rho) A_0^{-1} \ast \chi_{|\rho| \geq R}(\rho) S_\epsilon(\rho) A_0^{-1}\|_{\W} \\
& + \|\chi_{|\rho| \geq R}(\rho) S_\epsilon(\rho) A_0^{-1} \ast \chi_{|\rho| \leq R}(\rho) S_\epsilon(\rho) A_0^{-1}\|_{\W} + \|\chi_{|\rho| \geq R}(\rho) S_\epsilon(\rho) A_0^{-1} \ast \chi_{|\rho| \geq R}(\rho) S_\epsilon(\rho) A_0^{-1}\|_{\W} \\
&\les (\langle R \rangle^p \epsilon)^2 \|T\|_{\W}^2 \|A_0^{-1}\|_{{\frak X}}^2 + 2 \langle R \rangle^p \epsilon \|T\|_{\W} \|A_0^{-1}\|_{{\frak X}} \|S_\epsilon A_0^{-1}\|_{\W} +  \langle R \rangle^{1-p} \|S_\epsilon A_0^{-1}\|_{\W}^2\\
&\les \|T\|_{\W}^2 \|A_0^{-1}\|_{{\frak X}}^2 (\langle R \rangle^{2p} \epsilon^2 + \langle R \rangle^p \epsilon + \langle R \rangle^{1-p}).
\end{aligned}\ee

For sufficiently large $R$ and correspondingly small chosen $\epsilon<\epsilon_0(R)=\langle R \rangle^{-p}$, the right-hand side can be made arbitrarily small. Then $\|S_\epsilon A_0^{-1} \ast S_\epsilon A_0^{-1}\|_{\W}$ will be small as well, making the series (\ref{inv_F}) converge in $\W$ and (\ref{serie}) converge in $\widehat{\W}$.

Under assumption (\ref{leib}), whenever $T \in \tilde \W$, for $N \geq 1$ one has that
$$
\|T^N\|_{\tilde \W} \les N C^{N-1} \|T\|_{\W}^{N-1} \|T\|_{\tilde \W}.
$$
Consequently, the geometric series (\ref{serie'}) and (\ref{serie}) can also be made to converge in $\tilde W$, for radii $L$ and $\epsilon_0$ that do not depend on $\|T\|_{\tilde W}$, but only on $\|(I+\widehat T(0))^{-1}\|_{\frak X \cap \tilde {\frak X}}$, in the latter case.



The expression $(I+\widehat T(\lambda_0))^{-1}$ being norm-continuous in ${\frak X}$, it follows that for each compact interval $[-2L, 2L]$ there is a nonzero lower bound on the length $\epsilon$ required to make the series (\ref{serie}), which represents $\phi(\lambda) (I+\widehat T(\lambda))^{-1}$ for some $\phi$ supported on $(\lambda_0-\epsilon/2, \lambda_0+\epsilon/2)$, convergent in $\widehat{\W}$ for any $\lambda_0 \in [-2L, 2L]$.


Choose a finite covering of the compact set $[-2L, 2L]$ and a subordinated
partition of the unity $(\phi_j)_j$ with
$\sum_j \phi_j = \eta(\lambda/2L)$, such that for each $j$ the local inverse
$\phi_j(\lambda) \big(I + \widehat
T(\lambda)\big)^{-1} \in \widehat {\W}$ is given
by an explicit series as above. Then $\eta(\lambda/2L) \big(I + \widehat
T(\lambda)\big)^{-1}$ is the sum of finitely many elements of $\widehat {\W}$, so it belongs to $\widehat {\W}$.

Taking into account (\ref{infinit}), consequently $\big(\one + T\big)^{-1} - \one$ belongs to $\W$, as well as to $\tilde W$ if $T \in \tilde W$, $(I+\widehat T(0))^{-1} \in \tilde {\frak X}$ for $\lambda \in \R$, and (\ref{leib}) holds.

Summarizing previous computations, the inverse is given by an expression of the following type:
\be\lb{inverse_form}
\big(\one + T\big)^{-1} = (\one - \eta) \ast \sum_{k=0}^\infty (\tilde \eta \ast T - T)^k + \sum_{j=1}^J \widecheck{\phi_j} \ast (I+\widehat T(\lambda_j))^{-1} \sum_{k=0}^\infty [[\eta_j(\rho) \widehat T(\lambda_j) - \eta_j \ast (e^{-\lambda_j \rho} T(\rho))] (I+\widehat T(\lambda_j))^{-1}]^k,
\ee
where $\eta$, $\tilde \eta$, $\phi_j$, and $\eta_j$ are Schwartz-class functions and the sums converge in the $\W \subset ({\frak L_1})_t {\frak X}$ norm. The right-hand side contains an identity term.

We can approximate $S = (\one + T)^{-1}-\one$ in $\W$ with arbitrary precision by truncating each infinite sum in (\ref{inverse_form}) to some large, but finite, order. Since $\W \subset ({\frak L_1})_t {\frak X}$, this also provides an arbitrarily good approximation in $({\frak L_1})_t {\frak X}$.

Likewise, we can also truncate the cutoff functions to compact sets and approximate them by simple functions. Since $T$ is completely continuous in $({\frak L_1})_t {\frak X}$, it follows that all its powers are too.

By (\ref{inverse_form}), then, $S = \big(\one + T\big)^{-1} - \one$ is also completely continuous in $({\frak L_1})_t {\frak X}$.

When $p>2$, the proof changes because estimate (\ref{estp2}) does not hold. Instead
$$
\bigg\|\int_{|\sigma| \leq |\rho|/2} \epsilon\big(\widecheck{\eta}(\epsilon(\rho-\sigma))-
\widecheck{\eta}(\epsilon\rho)\big) T(\sigma)\,d\sigma\bigg\|_{{\frak X}} \les \|T\|_{\W} |\rho|^{-2-\delta} \epsilon^{-\delta} \int_{|\sigma| \leq |\rho|/2} |\sigma| \langle \sigma \rangle^{-p} \dd \sigma \les |\rho|^{-2+\delta} \epsilon^{-\delta} \|T\|_{\W}.
$$
So there is an $\epsilon^{-\delta}$ loss. For $2+\delta=p$ we obtain, analogously to (\ref{mainest}),
$$
\|S_\epsilon A_0^{-1} \ast S_\epsilon A_0^{-1}\|_{\W} \les \|T\|_{\W}^2 \|A_0^{-1}\|_{{\frak X}}^2 \langle R \rangle^p \epsilon^{\3-p} + \langle R \rangle^{1-p} \epsilon^{2(2-p)}).
$$
Setting $\langle R \rangle = \epsilon^{-\alpha}$, we get powers of $2-2p\alpha$, $\3-p-p\alpha$, and $2(2-p)+(p-1)\alpha$ on the right-hand side. For all of them to be positive, the conditions
$$
\alpha<1/p,\ \alpha<(\3/p)-1,\ \alpha>2(p-2)/(p-1)
$$
produce two inequalities. The first inequality
$$
\frac 1 p > \frac {2(p-2)}{p-1}
$$
gives $2p^2-5p+1<0$, so $p<(5+\sqrt{17})/4 \simeq 2.28$, while the other
$$
\frac {\3-p} p > \frac {2(p-2)}{p-1}
$$
gives $\3p^2-8p+\3<0$, so $p<(4+\sqrt 7)/\3 \simeq 2.21$. So it is possible to go slightly above $p=2$.

Next, we aim to get to $p<3$. For this purpose, we replace $\eta(\lambda/\epsilon)(I+\widehat T(\lambda)-A_0)$ by $\eta(\lambda/\epsilon)(I+\widehat T(\lambda)-A_0-\lambda\tilde A_0)$, where $\tilde A_0=[\partial_\lambda \widehat T](0)$ is well-defined when $p>2$ as the Fourier transform of $\rho T(\rho)$.

In other words, here we can use one more term in the Taylor expansion. $S_\epsilon$ is replaced by
$$\begin{aligned}
\tilde S_\epsilon(\rho) &= \epsilon \widecheck \eta(\epsilon \cdot) \ast T(\rho) + \epsilon \widecheck \eta(\epsilon \rho)(I-A_0) + \epsilon^2 \widecheck \eta'(\epsilon \rho) \tilde A_0 \\
&= \int_\R \epsilon \big(\widecheck \eta(\epsilon(\rho-\sigma)) - \widecheck \eta(\epsilon \rho) + \epsilon \sigma (\widecheck \eta)'(\epsilon \rho) \big) T(\sigma) \dd \sigma.
\end{aligned}$$

Now the parenthesis is bounded by $|\rho|^{-3} |\sigma|^2$ for $|\sigma| \leq |\rho|/2$. Consequently, the analogue of (\ref{estp2}) is now valid for $p<3$:
$$
\bigg\|\int_{|\sigma| \leq |\rho|/2} \epsilon\big(\widecheck{\eta}(\epsilon(\rho-\sigma))-
\widecheck{\eta}(\epsilon\rho)\big) T(\sigma)\,d\sigma\bigg\|_{{\frak X}} \les \|T\|_{\W} |\rho|^{-3} \int_{|\sigma| \leq |\rho|/2} |\sigma|^2 \langle \sigma \rangle^{-p} \dd \sigma \les |\rho|^{-p} \|T\|_{\W}.
$$

In the region $|\sigma| \geq |\rho|/2$, the two old terms behave as before and the new term is also bounded:
$$
\bigg\|\int_{|\sigma| > |\rho|/2} \epsilon^2
(\widecheck{\eta})'(\epsilon\rho) T(\sigma)\,d\sigma\bigg\|_{\B(X)} \les |\rho|^{-2} \int_{|\sigma| > |\rho|/2} \|T(\sigma)\|_{\B(X)} \,d\sigma \les |\rho|^{-p}.
$$

One final difference is that now we are perturbing around $A_0+\lambda\tilde A_0$ instead of simply $A_0$, so we have to bound $(A_0+\lambda\tilde A_0)^{-1}$. It is convenient to use a cutoff and estimate
\be\lb{newexp}
\|\eta(\lambda/\delta)(A_0+\lambda\tilde A_0)^{-1}\|_{\widehat \W}.
\ee

If $A_0$ is invertible, then so is $A_0 + \lambda \tilde A_0$ for all sufficiently small $\lambda$. After factoring out $A_0^{-1}$, bounding (\ref{newexp}) reduces to bounding $\eta(\lambda/\epsilon)(I-\lambda B)^{-1}$, where $B = -\tilde A_0 A_0^{-1}$.

Consider a $C^\infty$ operator family $F(\lambda)$; then for $\delta<1$
$$
\|[\eta(\lambda/\delta) F(\lambda)]^\vee\|_{\langle \rho \rangle^{-1} L^\infty_{\rho} {\frak X}} \les \|[\partial_\lambda (\eta(\lambda/\delta) F(\lambda))]^\vee\|_{L^\infty_{\rho} {\frak X}} + \|[\eta(\lambda/\delta) F(\lambda)]^\vee\|_{L^\infty_{\rho} {\frak X}} \les 1.
$$
More generally, for any $n \geq 0$ and $\epsilon \leq 1$
$$
\|[\eta(\lambda/\delta) \lambda^n F(\lambda)]^\vee\|_{\langle \rho \rangle^{-n-m-1} L^\infty_{\rho} {\frak X}} \les_n \delta^{-m}.
$$
Then, for all sufficiently small $\delta$ such that $I-\lambda B$ is invertible whenever $|\lambda|<\delta$,
$$
\|\eta(\lambda/\delta)(I-\lambda B)^{-1}\|_{\langle \rho \rangle^{-n} L^\infty_\rho {\frak X}} \les_n \delta^{1-n}.
$$
Hence, clearly, $\eta(\lambda/\delta)(I-\lambda B)^{-1} \in \W$ and is completely continuous in $({\frak L_1})_t {\frak X}$ for small $\delta$.
%

Since blow-up is not advantageous, we freeze $\delta$ at some small value for which the previous expression has finite norm.

Likewise, when analyzing the case $p<2$, we could have used that
$$
\|\eta(\lambda/\delta) A_0^{-1}\|_{\langle \rho \rangle^{-n} L^\infty_\rho \B(X)} \les_n \delta^{1-n},
$$
with $\delta$ fixed to some suitable small value. A similar bound holds in all cases of interest.

With these modifications, for $p<3$ and small $\epsilon$, the previous bound (\ref{mainest}) becomes
$$
\|S_\epsilon A_0^{-1} \ast S_\epsilon A_0^{-1}\|_{\W} \les \|T\|_{\W}^2 \|\eta(\lambda/\delta)(A_0 + \tilde A_0)^{-1}\|_{\widehat{\W}} (\langle R \rangle^{2p} \epsilon^2 + \langle R \rangle^p \epsilon + \langle R \rangle^{p-1}).
$$
Again, for sufficiently large $R$ and correspondingly small chosen $\epsilon<\epsilon_0(R)=\langle R \rangle^{-p}$, the right-hand side can be made arbitrarily small.

In order to go slightly above $3$, using the same modification as previously leads to a bound of
$$
\|S_\epsilon A_0^{-1} \ast S_\epsilon A_0^{-1}\|_{\W} \les \langle R \rangle^{2p} \epsilon^2 + \langle R \rangle^p \epsilon^{4-p} + \langle R \rangle^{p-1} \epsilon^{2(3-p)}.
$$
If $\langle R \rangle = \epsilon^{-\alpha}$, then obtaining an arbitrarily small norm requires the conditions
$$
\frac 1 p > \frac {2(p-3)}{p-1},\ \frac {4-p} p > \frac {2(p-3)}{p-1},
$$
which are met for some value of $p>3$.

The same idea, using a Taylor expansion, applies when $3$ is replaced with an arbitrarily large integer $n$, so in fact there is no upper limit for $p$.

\end{proof}

\subsection{Dispersive estimates in the scalar case}

We initially apply Theorem \ref{thm:Wiener} to prove decay estimates for the averaged solution (\ref{average}) of the Schr\"{o}dinger equation (\ref{random_lin}). This is both easier than and a prerequisite to studying the Liouville-type equation (\ref{average_liouville}).

More specifically, we prove that the integral kernel $e^{it(-\Delta+V+iA)}(y_1, y_2)$ decays like $\langle t \rangle^{-\3/2}$ and its norm forms a bounded kernel from $L^p_{y_2}$ to $L^p_{y_1}$, for all $p \in [1, \infty]$:
\begin{lemma}\lb{compl_cont_1var}
\begin{multline}\lb{kernel_estimate}
\sup_{y_2} \int_Y \|e^{it(-\Delta+V+iA)}(y_1, y_2)\|_{\B(L^1_x \cap L^2_x, L^2_x + L^\infty_x)} \dd y_1 + \\
+ \sup_{y_1} \int_Y \|e^{it(-\Delta+V+iA)}(y_1, y_2)\|_{\B(L^1_x \cap L^2_x, L^2_x + L^\infty_x)} \dd y_2 \les \langle t \rangle^{-\3/2}.
\end{multline}
Furthermore, the operator
$$
v_2(x, y) e^{it(-\Delta+V+iA)} v_1(x, y)
$$
is completely continuous in $({\frak L_1})_t [L^\infty_{y_2} L^1_{y_1} \cap L^\infty_{y_1} L^1_{y_2}] \B(L^2_x)$, i.e.\;for any $\epsilon>0$ there exist $N$ and $f_n$, $g_n$, $\tilde g_n$, $h_n$, $\tilde h_n$, $1 \leq n \leq N$, such that
$$
\bigg\|v_2(x, y) e^{it(-\Delta+V+iA)} v_1(x, y) - \sum_{n=1}^N f_n(t) g_n(y_1) \tilde g_n(y_2) [h_n(x) \otimes \tilde h_n(x)] \bigg\|_{({\frak L_1})_t [L^\infty_{y_2} L^1_{y_1} \cap L^\infty_{y_1} L^1_{y_2}] \B(L^2_x)} < \epsilon.
$$
\end{lemma}
\begin{proof}[Proofs of Proposition \ref{prop26} and Lemma \ref{compl_cont_1var}]
For $1 \leq p \leq \infty$, let $X=L^p_y L^2_x$ and consider the operator subalgebra ${\frak X} \subset \B(X)$ of operators given by kernels $T$ with the property that
$$
\sup_{y_2} \int_Y \|T(y_1, y_2)\|_{\B(L^2_x)} \dd y_1 + \sup_{y_1} \int_Y \|T(y_1, y_2)\|_{\B(L^2_x)} \dd y_2 < \infty
$$
or using a notation in the form (\ref{syx})
$$
{\frak X} = [L^\infty_{y_2} L^1_{y_1} \cap L^\infty_{y_1} L^1_{y_2}] \B(L^2_x).
$$
Let $\W$ be the Banach space of families of operators
$$
\W = \W_{{\frak X}} = \{(T(t))_{t \geq 0} \mid \|T(t)\|_{{\frak X}} \les \langle t \rangle^{-\3/2}\} = \langle t \rangle^{-\3/2} L^\infty_t [L^\infty_{y_2} L^1_{y_1} \cap L^\infty_{y_1} L^1_{y_2}] \B(L^2_x).
$$
Then $\W=\W_{{\frak X}}$ is a subalgebra of $\W_X = \langle t \rangle^{-\3/2} L^\infty_t \mc L(X)$.

In particular, let
$$
T(t) = \chi_{t>0}(t) v_2 e^{it(-\Delta+iA)} v_1.
$$
Clearly $T \in \W$ and $\|T\|_\W \les \|v_2\|_{C_y (L^2_x \cap L^\infty_x)} \|v_1\|_{C_y (L^2_x \cap L^\infty_x)} = \|V\|_{C_y (L^2_x \cap L^\infty_x)}$. Indeed, the integral kernel of $T$ has enough decay because
\be\lb{decay}
\|e^{it\Delta} f\|_{L^2_x + L^\infty_x} \les \langle t \rangle^{-\3/2} \|f\|_{L^1_x \cap L^2_x}
\ee
and the diffusion $e^{-tA}$ in the $y$ variable, by itself, has the property that
$$
\sup_{t > 0} \sup_{y_2} \int_Y |e^{-tA}(y_1, y_2)| \dd y_1 + \sup_{t > 0} \sup_{y_1} \int_Y |e^{-tA}(y_1, y_2)| \dd y_2 < \infty.
$$
In our case, both quantities equal $1$ due to properties of Markov chains.

In particular, as a consequence of the $\langle t \rangle^{-\3/2}$ decay in (\ref{decay}), $T(t)$ is integrable in $t$ and therefore its Fourier transform is bounded:
$$
\sup_{\Im \lambda \leq 0} \|\widehat T(\lambda)\|_{\B(L^p_y L^2_x)} \leq \sup_{\Im \lambda \leq 0} \|\widehat T(\lambda)\|_{{\frak X}} < \infty.
$$

After this reduction, we still have to prove that Wiener's Theorem \ref{thm:Wiener} and Proposition \ref{subalg} apply in this case, by checking that their hypotheses hold.


We show that the condition (\ref{translation_high}) from the theorem's statement holds in $\W$, hence also in $\W_X$.
If (\ref{translation}) holds for the cutoff kernels $\chi_{t \in [1/R, R]}(t) T(t)$ and each $R>0$, then (\ref{translation_high}) holds for $T$ with $N=2$, because the square of the tail at infinity is small in norm, see (\ref{prop\3}). In fact, it suffices to prove the following similar statement:
\be\lb{concl}\begin{aligned}
&\lim_{\epsilon \to 0} \|\chi_{t \in [1/R, R]}(t) \big(T(t+\epsilon) - T(t)\big)\|_{L^\infty_t [L^\infty_{y_2} L^1_{y_1} \cap L^\infty_{y_1} L^1_{y_2}] \B(L^2_x)} = 0.
\end{aligned}\ee
This follows immediately from Lemma \ref{lemma_compact} (complete continuity).

By Proposition \ref{invertibility} we know that $I+i\widehat T(\lambda)$ is invertible and
$$
(I+i\widehat T(\lambda))^{-1}-I \in {\frak X}.
$$


Theorem \ref{thm:Wiener} (Wiener's Theorem) and Proposition \ref{subalg} applied to $\one +iT$, where ${\frak X} \subset \B(X)$ is defined above and $X=L^p_y L^2_x$, 
now imply that
\be\lb{first_bound}
\chi_{t\geq 0}(t) v_2 e^{it(-\Delta+V+iA)} v_1 \in \W,
\ee
since
$$\begin{aligned}
\big(\one + i\chi_{t \geq 0}(t) v_2 e^{it(-\Delta+iA)} v_1\big)^{-1} = \one - i\chi_{t \geq 0}(t) v_2 e^{it(-\Delta+V+iA)} v_1.
\end{aligned}$$

Complete continuity in $({\frak L_1})_t {\frak X}$ also follows from Lemma \ref{lemma_compact} combined with Proposition \ref{subalg}.

We can use (\ref{first_bound}) to bound the corresponding factor $(\one + iT)^{-1}$ in Duhamel's formula
\be\lb{du}\begin{aligned}
&\chi_{t>0}(t) e^{it(-\Delta_x+iA_y+ V(x))} =\chi_{t>0}(t) e^{it(-\Delta_x+iA_y)} + \\
&+i\int_{t>s_1>s_2>0}e^{i(t-s_1)(-\Delta_x+iA_y)} v_1 (\one +iT)^{-1}(s_1-s_2) v_2 e^{is_2(-\Delta_x+iA_y)} \dd s_1 \dd s_2.
\end{aligned}\ee
The other terms and factors in (\ref{du}) can be estimated as follows:
$$\begin{aligned}
\|e^{it(-\Delta+iA)} f\|_{L^p_y (L^2_x+L^\infty_x)} \les \langle t \rangle^{-\3/2} \|f\|_{L^p_y (L^1_x \cap L^2_x)},\\
\|v_2 e^{it(-\Delta+iA)} f\|_{L^p_y L^2_x} \les \langle t \rangle^{-\3/2} \|f\|_{L^p_y (L^1_x \cap L^2_x)},\\
\|e^{it(-\Delta+iA)} v_1 f\|_{L^p_y (L^2_x+L^\infty_x)} \les \langle t \rangle^{-\3/2} \|f\|_{L^p_y L^2_x}.
\end{aligned}$$
This suffices to prove (\ref{weaker_bound}).

However, 
the other terms in (\ref{du}) also obey the stronger bounds
$$
\sup_{y_2} \int_Y \|[v_2 e^{it(-\Delta+iA)}](y_1, y_2)\|_{\B(L^1_x \cap L^2_x)} \dd y_1 + \sup_{y_1} \int_Y \|[v_2 e^{it(-\Delta+iA)}](y_1, y_2)\|_{\B(L^1_x \cap L^2_x)} \dd y_2 \les \langle t \rangle^{-\3/2},
$$
$$
\sup_{y_2} \int_Y \|[e^{it(-\Delta+iA)} v_1](y_1, y_2)\|_{\B(L^2_x + L^\infty_x)} \dd y_1 + \sup_{y_1} \int_Y \|[e^{it(-\Delta+iA)} v_1](y_1, y_2)\|_{\B(L^2_x + L^\infty_x)} \dd y_2 \les \langle t \rangle^{-\3/2},
$$
and
$$
\sup_{y_2} \int_Y \|e^{it(-\Delta+iA)}(y_1, y_2)\|_{\B(L^1_x \cap L^2_x, L^2_x + L^\infty_x)} \dd y_1 + \sup_{y_1} \int_Y \|e^{it(-\Delta+iA)}(y_1, y_2)\|_{\B(L^1_x \cap L^2_x, L^2_x + L^\infty_x)} \dd y_2 \les \langle t \rangle^{-\3/2}.
$$
This implies the stronger statement (\ref{kernel_estimate}) and completes the proof.
\end{proof}

%


\subsection{Trace-class dispersive estimates}\lb{sec_trace}

We begin by recalling some basic notions about the trace class and tensor products of Banach and Hilbert spaces.
\begin{definition}
The \textbf{trace class} $\frak S_1$ on a Hilbert space $H$ (in this paper, $H=L^2$) is defined as the set of bounded operators $K \in \B(H)$ such that
\be\lb{tracedef}
\sum_{n=1}^\infty |\langle K e_n, e_n \rangle| < \infty
\ee
for some orthonormal basis $(e_n)_{n \geq 1}$ of $H$.
\end{definition}
If the sum (\ref{tracedef}) is finite for some orthonormal basis, then it is finite for every orthonormal basis. Then the sum without absolute values
$$
\sum_{n=1}^\infty \langle K e_n, e_n \rangle
$$
is absolutely convergent and does not depend on the chosen orthonormal basis. Its value is called $\tr K$, the \textbf{trace} of $K$.

Equivalently, trace-class operators have kernels
\be\lb{tracedef2}\begin{aligned}
\frak S_1 := \{K(x_1, x_2) \in L^2_{x_1, x_2} \mid\ & K(x_1, x_2) = \sum_{n=1}^\infty f_n(x_1) g_n(x_2),\ \sum_{n=1}^\infty \|f_n\|_{L^2_{x_1}} \|g_n\|_{L^2_{x_2}} < \infty\}.
\end{aligned}\ee
In particular, $\frak S_1 \subset L^2_{x_1, x_2} \equiv \frak S_2$, the Hilbert--Schmidt class.

The trace class $\frak S_1$ is a Banach space, with the norm given by the infimum of (\ref{tracedef}) over all orthonormal bases or equivalently
$$
\|K\|_{\frak S_1} = \tr |K| = \tr [(K K^*)^{1/2}].
$$
Another equivalent norm is the infimum of the sum in (\ref{tracedef2}) over all possible representations.

The trace-class norm of the rank-one operator $f \otimes \ov f$ is
$$
\|f \otimes \ov f\|_{\frak S_1} = \|f\|_{L^2}^2.
$$

Both $\frak S_1$ and $\frak S_2$ are Banach spaces of compact operators and are ideals within $\B(L^2)$, whose closure in the operator norm is the set of compact operators. In addition, $\frak S_2$ is a Hilbert space.


In general, we define the \textbf{tensor product of two Banach spaces} $B_1$ and $B_2$ as the following Banach space. Let their algebraic tensor product $B_1 \otimes B_2$ be endowed with the norm
$$
\|K\|_{B_1 \otimes_p B_2} = \inf \Big\{\Big(\sum_{n=1}^N \|a_n\|_{B_1}^p \|b_n\|_{B_2}^p\Big)^{1/p} \mid \sum_{n=1}^N a_n \otimes b_n = K\Big\}.
$$
Then $B_1 \otimes_p B_2$ is defined as the completion of $B_1 \otimes B_2$ under this norm, i.e.\;the complete Banach space of equivalence classes of Cauchy sequences of elements of $B_1 \otimes B_2$ that converge together in this norm.

Let $\otimes = \otimes_1$. For any Banach spaces $B_1$ and $B_2$, an element of $B_1 \otimes B_2$ defines an operator from $B_2^*$ to $B_1$ or from $B_1^*$ to $B_2$. In particular, $\frak S_1 \equiv L^2_{x_1} \otimes_1 L^2_{x_2}$. Also, $L^1_x \otimes B_2 \equiv L^1_x B_2$ whenever $B_2$ is separable.

For other values of $p$, $B_1 \otimes_p B_2$ need not always have a simple realization. In general, though, if $p \leq q$ and $L^q$ is separable, then $L^p \otimes_p L^q \equiv L^p L^q$. In particular, $L^2 \otimes_p L^2 \equiv \frak S_p$ for $1 \leq p \leq 2$.

Next, we prove decay estimates for equation (\ref{average_liouville}) with trace-class initial data.


\begin{proposition}[Trace-class pointwise decay]\lb{prop27}\lb{prop28} Suppose that $V \in C_y (L^1_x \cap L^\infty_x)$ is real-valued and satisfies Assumption \ref{nontrivial}. Then the solution $f$ to the homogenous equation (\ref{average_liouville})
$$\begin{aligned}
i \partial_t f - \Delta_{x_1} f + \Delta_{x_2} f + i A f + (V \otimes 1 - 1 \otimes V) f = 0,\ f(0) = f_0
\end{aligned}$$
satisfies the estimate
\be\lb{point}
\|f(t)\|_{L^1_y (L^2_{x_1}+L^\infty_{x_1}) \otimes (L^2_{x_2}+L^\infty_{x_2})} \les \langle t \rangle^{-\3}
\|f_0\|_{L^1_y (L^1_{x_1}\cap L^2_{x_1}) \otimes (L^1_{x_2} \cap L^2_{x_2})}.
\ee
Moreover, consider the inhomogenous version of this equation:
$$\begin{aligned}
i \partial_t f - \Delta_{x_1} f + \Delta_{x_2} f + i A f + (V \otimes 1 - 1 \otimes V) f = F,\ f(0) = 0,
\end{aligned}$$
where
$$
F(s, y, x_1, x_2) = \int_{s \geq s_1, s_2} e^{-i(s-s_1)\Delta_{x_1} + i(s-s_2)\Delta_{x_2}} G(s, s_1, s_2, y, x_1, x_2) \dd s_1 \dd s_2.
$$
Given some fixed $s \geq s_1, s_2$,
\be\lb{decay_est'}\begin{aligned}
\left\|e^{i(t-s)(-\Delta_{x_1}+\Delta_{x_2}+iA+1\otimes V - V\otimes 1)} e^{-i(s-s_1)\Delta_{x_1} + i(s-s_2)\Delta_{x_2}} G \right\|_{L^1_y (L^2_{x_1} +L^\infty_{x_1}) \otimes (L^2_{x_2}+L^\infty_{x_2})} \les \\
\les \langle t - s_1 \rangle^{-\3/2} \langle t - s_2 \rangle^{-\3/2} \|G\|_{L^1_y (L^1_{x_1} \cap L^2_{x_1}) \otimes (L^1_{x_2} \cap L^2_{x_2})}.
\end{aligned}\ee
Furthermore, for fixed $t_1, t_2 \geq t$,
\be\lb{decay_est''}\begin{aligned}
\left\|e^{-i(t_1-t)\Delta_{x_1} + i(t_2-t)\Delta_{x_2}} e^{i(t-s)(-\Delta_{x_1}+\Delta_{x_2}+iA+1\otimes V - V\otimes 1)} e^{-i(s-s_1)\Delta_{x_1} + i(s-s_2)\Delta_{x_2}} G \right\|_{L^1_y (L^2_{x_1} +L^\infty_{x_1}) \otimes (L^2_{x_2}+L^\infty_{x_2})} \les \\
\les \langle t_1 - s_1 \rangle^{-\3/2} \langle t_2 - s_2 \rangle^{-\3/2} \|G\|_{(L^1_{x_1}\cap L^2_{x_1}) \otimes (L^1_{x_2} \cap L^2_{x_2})}.
\end{aligned}\ee
\end{proposition}

The proof is again based on Wiener's Theorem \ref{thm:Wiener}. The most important ingredient is a suitable Banach operator space, which we construct step by step below in a sequence of lemmas.

The point of estimates (\ref{decay_est'}) and (\ref{decay_est''}) is that one can put the free evolution in either variable, not necessarily for the same lengths of time, to both the beginning and the end of the perturbed evolution, and still get the expected rate of decay.

\begin{observation} These explicit bounds permit including a small deterministic time-dependent potential into the equation. If we perturb $V \otimes 1 - 1 \otimes V$ by $\tilde V_1 \otimes 1 + 1 \otimes \tilde V_2$, where
$$
\|\tilde V_1\|_{L^\infty_t L^\infty_y (L^1_x \cap L^\infty_x)} + \|\tilde V_2\|_{L^\infty_t L^\infty_y (L^1_x \cap L^\infty_x)} << 1,
$$
then the same estimates (\ref{point}-\ref{decay_est''}) hold.

Moreover, in equation (\ref{average_liouville}), the potential $V \otimes 1 - 1 \otimes V$ can be replaced by $V_1 \otimes 1 - 1 \otimes V_2$: the two potentials need not be the same.
\end{observation}

Pointwise-in-time decay estimates, especially (\ref{decay_est''}), immediately lead to Strichartz estimates as well.

\begin{corollary}[Trace-class local decay]\lb{cor_strichartz}
Suppose that $V \in C_y (L^1_x \cap L^\infty_x)$ is real-valued and satisfies Assumption \ref{nontrivial}. Then, for any weights $w_1, w_2 \in L^\infty_y L^{\3, 1}_x$, the solution $f$ to the homogenous equation (\ref{average_liouville})
$$\begin{aligned}
i \partial_t f - \Delta_{x_1} f + \Delta_{x_2} f + i A f + (V \otimes 1 - 1 \otimes V) f = 0,\ f(0) = f_0
\end{aligned}$$
satisfies the estimate
\be\lb{strichartz}
\|w_1(x_1) f w_2(x_2)\|_{L^1_y L^1_t \frak S_1} \les \|w_1\|_{L^\infty_y L^{\3, 1}_{x_1}} \|w_2\|_{L^\infty_y L^{\3, 1}_{x_2}}
\|f_0\|_{L^1_y \frak S_1}.
\ee
More generally, consider the inhomogenous version of the equation:
$$\begin{aligned}
i \partial_t f - \Delta_{x_1} f + \Delta_{x_2} f + i A f + (V \otimes 1 - 1 \otimes V) f = F,\ f(0) = 0,
\end{aligned}$$
where
$$
F(y, s, x_1, x_2) = \int_{s \geq s_2} \int_Y e^{i(s-s_2)\Delta_{x_2}} G(y, s, s_2, x_1, x_2) \dd s_2.
$$
Then
\be\lb{strichartz'}
\|w_1(x_1) f w_2(x_2)\|_{L^1_y L^1_t \frak S_1} \les \|w_1\|_{L^\infty_y L^{\3, 1}_{x_1}} \|w_2\|_{L^\infty_y L^{\3, 1}_{x_2}} \|G\|_{L^1_y (L^1_s L^2_{x_1} + L^2_s L^{6/5, 2}_{x_1}) \otimes (L^1_{s_2} L^2_{x_2} + L^2_{s_2} L^{6/5, 2}_{x_2})}.
\ee
\end{corollary}
More inhomogenous estimates are possible, but there are delicate endpoint issues.


\subsubsection{Setup}

The main difficulty is that both operators
$$
v_2(x_1, y) e^{it(-\Delta_{x_1}+\Delta_{x_2}+iA)}: L^p_y \frak S_1 \to L^p_y \frak S_1
$$
and
$$
e^{it(-\Delta_{x_1}+\Delta_{x_2}+iA)} v_1(x_1, y): L^p_y \frak S_1 \to L^p_y \frak S_1
$$
lack the requisite time decay properties: neither decays faster than $L^2$ in $t$.

Thus, the most one could prove directly by such decay estimates is that the evolution is bounded from $L^1_y \frak S_1$ to $L^2_t L^1_y \frak S_1$: for $V \in C_y (L^1_x \cap L^\infty_x)$,
$$
\|e^{it(-\Delta_{x_1}+\Delta_{x_2}+iA + V \otimes 1 - 1 \otimes V)} f\|_{L^2_t L^1_y \frak S_1} \les \|f\|_{L^1_y \frak S_1}.
$$

Likewise, the $t^{-\3}$ decay of the evolution cannot be proved in this manner, because, even if $v_2$ and $v_1$ had compact support, decay would be no faster than $\langle t \rangle^{-\3/2} L^2_t$. Then the best achievable result by direct estimates would be
$$
\|e^{it(-\Delta_{x_1}+\Delta_{x_2}+iA + V \otimes 1 - 1 \otimes V)} f\|_{\langle t \rangle^{-\3/2} L^2_t L^1_y (L^2_{x_1} + L^\infty_{x_1}) \otimes (L^2_{x_2} + L^\infty_{x_2})} \les \|f\|_{L^1_y (L^1_{x_1} \cap L^2_{x_1}) \otimes (L^1_{x_2} \cap L^2_{x_2})}.
$$
See (\ref{weak_est}) for a partial result in this direction.

The common deficiency in both cases is only using dispersion in only one variable and the unitarity of the evolution in the other variable. However, the equation is dispersive in both variables, which we can quantify by considering the structure of the operator $e^{it(-\Delta_{x_1}+\Delta_{x_2}+iA+V \otimes 1 - 1 \otimes V)}$ in more detail.

The main technical device is keeping track of the initial and final segments in $S_1 \ast S_2$:
$$\begin{aligned}{}
[S_1 \ast S_2](t) := &\int_{t>s_1>s_2>s_\3>0} (I+iT_{11})^{-1}(t-s_1) T_{12}(s_1-s_2) (I+iT_{22})^{-1}(s_2-s_\3) T_{21}(s_\3) \dd s_1 \dd s_2 \dd s_\3.
\end{aligned}$$
The computations have to be done in a suitable operator space. 
Here is our choice: let
\be\lb{U1}
[U_1 F](y_1, x_1, x_2) = \int_0^\infty \int_Y e^{is_1\Delta_{x_2}} v_1(x_2, y)  F(y_1, y, x_1, s_1, x_2) \dd y \dd s_1
\ee
and
\be\lb{U2}
[U_2 f](s_2, y_2, \tilde y, x_1, x_2) = v_2(x_2, \tilde y) \chi_{s_2 \geq 0} e^{is_2\Delta_{x_2}} f(y_2, x_1, x_2).
\ee
These are the initial and final segments we shall keep track of. $U_2$ introduces two extra variables, $s_2$ and $\tilde y$, and $U_1$ eliminates them. 

As proved below, each operator appearing in the computation can be written as $U_2 B U_1$, for some suitable choice of $B$.

The Banach spaces of operators we use in this proof have the following general structure:
\be\lb{syx}
SYX = \{T \mid \Big\|\big\|\|T\|_X\big\|_Y\Big\|_S < \infty\}.
\ee
The two outermost spaces $S$ and $Y$ will be \emph{Banach latices} of operator kernels, meaning that, whenever $T(s_1, s_2) \in S$ and $|\tilde T(s_1, s_2)| \leq |T(s_1, s_2)|$, then $\tilde T \in S$ as well.

In other words, $T \in SYX$ has a kernel $T(s_1, s_2)$ such that $T(s_1, s_2) \in YX$ for each $s_1, s_2 \geq 0$ and $\|T(s_1, s_2)\|_Y$ is a bounded kernel in the operator space $S$.

Likewise, $\tilde T \in YX$ means the following: $\tilde T$ has an integral kernel $\tilde T(y_1, y_2) \in X$ for each $y_1, y_2 \in Y$, such that $\|\tilde T(y_1, y_2)\|_X \in Y$.

For example, estimate (\ref{kernel_estimate}) from Proposition \ref{prop26} concerning $e^{it(-\Delta+iA)}$ can be restated as
$$
\chi_{t \geq 0} e^{it(-\Delta+iA+V)} \in \chi_{t \geq 0} \langle t \rangle^{-\3/2} L^\infty_t (L^\infty_{y_2} L^1_{y_1} \cap L^\infty_{y_1} L^1_{y_2}) \B(L^1_x \cap L^2_x, L^2_x + L^\infty_x).
$$

We shall use the following spaces of operators in the course of the proof:\\
1. For $s$: integral kernels in $L^\infty_s$, $\langle s \rangle^{-p} L^\infty_s$, $\langle s_1 + s_2 \rangle^{-p} L^\infty_{s_1, s_2}$, or $\mc M_{s_1, s_2}$.

We further distinguish the positive subspace, consisting of kernels supported on $\{(s_1, s_2) \in \R^2: s_1, s_2 \geq 0\}$. Positive support plays a part in the study of compactness and similar properties.

Assuming positive support, operators with kernels in $\chi_{s_1, s_2 \geq 0} \langle s_1 + s_2 \rangle^{-p} L^\infty_{s_1, s_2} \subset L^2_{s_1, s_2}$ are in the Hilbert--Schmidt class for $p>1$.

\noindent 2. For $y$: the spaces involved will be spaces of integral kernels acting on $Y$ or on $Y \times Y$, such as $L^\infty_{y_2} L^1_{y_1}$ and $L^\infty_{y_1} L^1_{y_2}$. For example,
$$
[K f](y_1) = \int_Y K(y_1, y_2) f(y_2) \dd y_2 
$$
is in $\B(L^1_y)$ if the corresponding integral kernel is in $L^\infty_{y_2} L^1_{y_1}$ and in $\B(L^\infty_y)$ if $K(y_1, y_2) \in L^\infty_{y_1} L^1_{y_2}$.

Likewise, $L^\infty_{y_2} L^1_{\tilde y, y_1, y} \subset \B(L^1_{y_2} L^\infty_{\tilde y}, L^1_{y_1} L^1_y)$ and $L^\infty_{y_1} L^1_{y, y_2, \tilde y} \subset \B(L^\infty_{y_2} L^\infty_{\tilde y}, L^\infty_{y_1} L^1_y)$.

\noindent 3. For $x$: $\B(L^2)$, $\B(L^2) \otimes \B(L^2)$, $\B(\frak S_1)$, $\C I$, etc.. Multiples of the identity operator appear in the computation.

The composition of operators from two such classes is also in such a class whenever the composition is well-defined. Indeed, consider two operators $T_1$ and $T_2$, such that $T_j \in S_j Y_j X_j$, $1 \leq j \leq 2$. Then
$$
T_1 T_2 \in (S_1 S_2) (Y_1 Y_2) (X_1 X_2),
$$
whenever the compositions of operators in the individual operator spaces $S_1$ and $S_2$, $Y_1$ and $Y_2$, and $X_1$ and $X_2$ make sense.

Due to endpoint Strichartz estimates, both the initial and the final segments are bounded as operators between the following spaces:
$$
\|U_1 F\|_{L^1_{y_1} \frak S_1} \les \|F\|_{L^1_{y_1} L^1_y (L^2_{x_1} \otimes L^2_{s_1} L^2_{x_2})}
$$
and
$$
\|U_2 f\|_{L^1_{y_2} L^\infty_{\tilde y} (L^2_{x_1} \otimes L^2_{s_2} L^2_{x_2})} \les \|f\|_{L^1_{y_2} \frak S_1}.
$$
These characterizations of $U_1$ and $U_2$ are not in the form (\ref{syx}), but here are two characterizations that are: with decay
\be\lb{bound1}
U_1(s_1) \in \chi_{s_1 \geq 0} \langle s_1 \rangle^{-\3/2} L^\infty_{s_1} L^\infty_{y_1, y} (\C I_{x_1} \otimes \B(L^2_{x_2}, L^2_{x_2}+L^\infty_{x_2}))
\ee
and
\be\lb{bound2}
U_2(s_2) \in \chi_{s_2 \geq 0} \langle s_2 \rangle^{-\3/2} L^\infty_{s_2} L^\infty_{y_2, \tilde y} (\C I_{x_1} \otimes \B(L^1_{x_2} \cap L^2_{x_2}, L^2_{x_2}))
\ee
and without decay
\be\lb{u1est}
U_1(s_1) \in \chi_{s_1 \geq 0} L^\infty_{s_1} L^\infty_{y_1, y} (\C I_{x_1} \otimes \B(L^2_{x_2})),
\ee
$$
U_2(s_2) \in \chi_{s_2 \geq 0} L^\infty_{s_2} L^\infty_{y_1, y} (\C I_{x_1} \otimes \B(L^2_{x_2})).
$$

Composed in the appropriate order,
$$
[U_2 U_1 F](y_2, \tilde y, x_1, s_1, x_2) = \chi_{s_1, s_2 \geq 0} \delta_{y_1 = y_2} v_2(x_2, \tilde y) e^{is_1\Delta_{x_2}} \int_0^\infty \int_Y e^{is_2\Delta_{x_2}} v_1(x_2, y) F(y_1, y, x_1, s_2, x_2) \dd y \dd s_2
$$
has extra decay properties that the individual components lack: it decays like $\langle s_1+s_2 \rangle^{-\3/2}$ and then the operator norm forms a bounded kernel from $L^1_{y_2} L^1_y$ to $L^1_{y_1} L^\infty_{\tilde y}$:
\be\lb{u2u1bound}
U_2 U_1 \in \chi_{s_1, s_2 \geq 0} \langle s_1+s_2 \rangle^{-\3/2} L^\infty_{s_1, s_2} (L^\infty_{y_2, y} L^1_{y_1} L^\infty_{\tilde y}) (\C I_{x_1} \otimes \B(L^2_{x_2})).
\ee

$S_1 \ast S_2$ is a sum of four terms:
$$
S_1 \ast S_2 = T_1 + T_2 + T_3 + T_4,
$$
where
$$
T_1 = T_{12} \ast T_{21},
$$
$$
T_2 = [(I+iT_{11})^{-1} - I] \ast T_{12} \ast T_{21},
$$
$$
T_3 = T_{12} \ast [(I+iT_{22})^{-1} - I] \ast T_{21},
$$
and
$$
T_4 = [(I+iT_{11})^{-1} - I] \ast T_{12} \ast [(I+iT_{22})^{-1} - I] \ast T_{21}.
$$


The first two terms $T_1$ and $T_2$ are singular, containing Dirac deltas, but have a specific form that is easy to analyze. The other two terms $T_\3$ and $T_4$ and every other expression appearing in the proof will be uniformly bounded and belong to one common Banach space, see Definition \ref{defw}.

Starting with $T_1$ as a simpler case, note that
$$
T_1(t) = [T_{12} \ast T_{21}](t) = U_1 B_1(t) U_2,
$$
where $B_1(t)$ is given by
\be\begin{aligned}[]\lb{def_b1}
[B_1(t) F](s_1, y_1, y, x_1, x_2) = \chi_{t \geq 0} \delta_{t-s_1}(s_2) \int_{t \geq s_2 \geq 0} &e^{-s_1 A}(y_1, y) \delta_{y = \tilde y} e^{-s_2 A}(\tilde y, y_2) [(v_2(x_1, y_1) e^{-it\Delta_{x_1}} v_1(x_1, y_2)) \otimes I_{x_2}] \\
&F(s_2, y_2, \tilde y, x_1, x_2) \dd \tilde y \dd y_2 \dd s_2.
\end{aligned}\ee

For each value of the $s$'s and $y$'s, this operator is in $\B(L^2_{x_1}) \otimes \C I_{x_2}$, since
$$
\sup_{y_1, y_2 \in Y} \|v_2(x_1, y_1) e^{-it\Delta_{x_1}} v_1(x_1, y_2)\|_{\B(L^2_{x_1})} < \infty.
$$
Its norm, considered as a function of the $y$'s, forms a kernel bounded from $L^1_{y_2} L^\infty_{\tilde y}$ to $L^1_{y_1} L^1_y$, as well as from $L^\infty_{y_2} L^\infty_{\tilde y}$ to $L^\infty_{y_1} L^1_y$. This is because
$$
\sup_{s_1, s_2 >0} \sup_{y_1 \in Y} \iint_{Y \times Y} |e^{-s_1 A}(y_1, y) e^{-s_2 A}(y, y_2)| \dd y \dd y_2 < \infty
$$
and
$$
\sup_{s_1, s_2 >0} \sup_{y_2 \in Y} \iint_{Y \times Y} |e^{-s_1 A}(y_1, y) e^{-s_2 A}(y, y_2)| \dd y \dd y_1 < \infty.
$$
In fact the expression is non-negative and the integrals always equal $1$, by properties of Markov chains.

Finally, as a function of $s$, $T_1$'s operator norm is dominated by a specific singular measure in $\mc M_{s_1, s_2}$ (the space of signed measures on $\R^2$):
$$
\chi_{t \geq 0} \chi_{s_1, s_2 \geq 0} \langle t \rangle^{-\3/2} \delta_{t-s_1}(s_2).
$$

We can abbreviate this whole characterization as
$$
B_1(t) \in \chi_{s_1, s_2 \geq 0} \mc M_{s_1, s_2} (L^\infty_{y_2} L^1_{\tilde y, y_1, y} \cap L^\infty_{y_1} L^1_{y, y_2, \tilde y}) (\B(L^2_{x_1}) \otimes \C I_{x_2}).
$$
The norm of $B_1(t)$ is of size $\langle t \rangle^{-\3/2}$:
$$
\|B_1(t)\|_{\mc M_{s_1, s_2} (L^\infty_{y_2} L^1_{\tilde y, y_1, y} \cap L^\infty_{y_1} L^1_{y, y_2, \tilde y}) (\B(L^2_{x_1}) \otimes \C I_{x_2})} \les \langle t \rangle^{-\3/2}.
$$

Next, consider the second term
$$
T_2 = [(I+iT_{11})^{-1} - I] \ast T_{12} \ast T_{21} = [v_2(x_1, y) e^{it(-\Delta_{x_1}+iA+V\otimes 1)} v_1(x_1, y)] \ast T_1.
$$
Note that for $t \geq 0$
$$
\big\|\|[v_2 e^{it(-\Delta+iA+V)} v_1](y_1, y_2)\|_{\B(L^2_x)}\big\|_{L^\infty_{y_2} L^1_{y_1} \cap L^\infty_{y_1} L^1_{y_2}} \les \langle t \rangle^{-\3/2},
$$
which in the notation (\ref{syx}) becomes
$$
\chi_{t \geq 0} [v_2 e^{it(-\Delta+iA+V)} v_1] \in \chi_{t \geq 0} \langle t \rangle^{-\3/2} L^\infty_t [L^\infty_{y_2} L^1_{y_1} \cap L^\infty_{y_1} L^1_{y_2}] \B(L^2_x).
$$

Thus $T_2(t)$ is also of the form $T_2(t) = U_1 B_2(t) U_2$, where $B_2(t)$ is given by
\be\begin{aligned}[]\lb{def_b2}
&[B_2(t) F](s_1, y_1, y, x_1, x_2) = \\
&\chi_{t \geq 0} \delta_{t-s_1}(s_2) \int_{t\geq s_2 \geq 0} [v_2 e^{is_1(-\Delta_{x_1}+iA+V)} v_1](y_1, y') (v_2(x_1, y') e^{-it\Delta_{x_1}} v_1(x_1, y_2)) \otimes I_{x_2}] \\
& e^{-(s_1-s_1') A}(y', y) \delta_{y = \tilde y} e^{-s_2 A}(\tilde y, y_2) F(s_2, y_2, \tilde y, x_1, x_2) \dd y' \dd s_1' \dd \tilde y \dd y_2 \dd s_2.
\end{aligned}\ee

By virtue of Proposition \ref{prop26}, we can describe $T_2$ in exactly the same way:
$$
T_2(t) = U_1 B_2(t) U_2,\ B_2(t) \in \chi_{s_1, s_2 \geq 0} \mc M_{s_1, s_2} (L^\infty_{y_2} L^1_{\tilde y, y_1, y} \cap L^\infty_{y_1} L^1_{y, y_2, \tilde y}) (\B(L^2_{x_1}) \otimes \C I_{x_2}).
$$
Its norm is also of size $\langle t \rangle^{-\3/2}$.

Likewise, $T_3(t)$ and $T_4(t)$ are both of the form $U_1 B(t) U_2$, where for example $B_3$ is given by
\be\begin{aligned}[]\lb{def_b3}
[B_3(t) F](s_1, y_1, y, x_1, x_2) = &\chi_{t \geq 0} \int_{\substack{t \geq s_1 + s_2\\
s_1, s_2\geq 0}} e^{-s_1 A}(y_1, y) v_2(x_2, y) e^{-i(t-(s_1+s_2))(-\Delta_{x_2}+iA+1\otimes V)} v_1(x_2, \tilde y) \\
&e^{-s_2 A}(\tilde y, y_2) [(v_2(x_1, y_1) e^{-it\Delta_{x_1}} v_1(x_1, y_2)) \otimes I_{x_2}] F(s_2, y_2, \tilde y, x_1, x_2) \dd \tilde y \dd y_2 \dd s_2.
\end{aligned}\ee

Both $B_3$ and $B_4$ are in the following space $X_t$:
\begin{definition}\lb{defx} Let
$$
X_t = \chi_{s_1, s_2 \geq 0} \langle t-(s_1 + s_2) \rangle^{-\3/2} L^\infty_{s_1, s_2} [L^\infty_{y_2} L^1_{\tilde y, y, y_2}] \B(\frak S_1)
$$
and
$$
\mc X_t = X_t \oplus \C (B_1(t) + B_2(t)).
$$
\end{definition}
Then $\|T_3(t)\|_{U_1 X_t U_2}+\|T_4(t)\|_{U_1 X_t U_2} \les \langle t \rangle^{-\3/2}$, due to the decay in the $x_1$ factor.

All $X_t$ norms, for $t \geq 0$, are comparable, albeit not uniformly; so the underlying Banach space is the same: for any $t_1, t_2 \geq 0$
\be\lb{norm_comp}
\|T\|_{X_{t_1}} \les \langle t_1-t_2 \rangle^{\3/2} \|T\|_{X_{t_2}}.
\ee
For $t \leq 0$, $X_t \subset X_0$ and $\|T\|_{X_0} \leq \langle t_1-t_2 \rangle^{\3/2} \|T\|_{X_t}$.

By contrast, the measures that enter the definition of $\mc X_t$ are mutually singular, so they cannot be compared in the same norm for different values of $t$. However, the singular part is absent from higher powers, because $\mc X_t U_2 U_1 \mc X_s \subset X_{t+s}$.

\begin{lemma} For any $s, t \geq 0$, $\mc X_t U_2 U_1 \mc X_s \subset X_{t+s}$.
\end{lemma}
\begin{proof} First, we prove that $(B_1(s) + B_2(s)) U_2 U_1 X_t \subset X_{t+s}$.

Note that for $p>1$
\be\lb{triangle}
\int_\R \langle t-a \rangle^{-p} \langle t-b \rangle^{-p} \dd t \les \langle a-b \rangle^{-p}.
\ee
Hence
$$
\int_0^\infty \langle a + t \rangle^{-p} \langle b - t \rangle^{-p} \dd t = \int_0^\infty \langle a + t \rangle^{-p} \langle t-b \rangle^{-p} \dd t \les \langle a+b \rangle^{-p}.
$$
Then the statement reduces to the following claim: for $s_1, s_4 \geq 0$
$$\begin{aligned}
&\int_{\substack{s_2, s_\3 \geq 0}} [\delta_{s-s_1}(s_2)] \langle s_2+s_\3 \rangle^{-\3/2} [\langle t- (s_\3 + s_4) \rangle^{-\3/2}] \dd s_2 \dd s_\3 =\\
&= \int_{\substack{s_\3, s-s_1 \geq 0}} \langle s-s_1+s_\3 \rangle^{-\3/2} \langle t-s_\3 - s_4 \rangle^{-\3/2} \dd s_\3 \dd s\\
&\les \langle s+t - s_1 - s_4 \rangle^{-\3/2}.
\end{aligned}$$

The statement that $(B_1(s) + B_2(s)) U_2 U_1 (B_1(t) + B_2(t)) \subset X_{t+s}$ reduces to the true statement
$$\begin{aligned}
&\int_{\substack{s_2, s_\3 \geq 0}} [\delta_{s-s_1}(s_2)] \langle s_2+s_\3 \rangle^{-\3/2} [\delta_{t-s_\3}(s_4)] \dd s_2 \dd s_\3 \les \langle t-s_1-s_4 \rangle^{-\3/2}.
\end{aligned}$$

Same goes for the remaining combinations.
\end{proof}

\begin{definition}\lb{defw} Let
$$
W = \{(B(t))_{t \in \R} \mid \|B(t)\|_{X_t} \les \langle t \rangle^{-\3/2}\}.
$$
\end{definition}

So $T_3, T_4 \in U_1 W U_2$. Moreover, even though $S_1 \ast S_2$ contains two singular terms, $T_1, T_2 \not \in U_1 W U_2$, higher powers of $S_1 \ast S_2$ contain only $U_1 W U_2$ terms.
\begin{corollary} Let $B_1, B_2$ be given by (\ref{def_b1}) and (\ref{def_b2}) and $B, \tilde B \in W$. Then for any $f, \tilde f \in L^\infty_t$
$$
\|[f (B_1+B_2) U_2] \ast [U_1 \tilde f (B_1 + B_2)]\|_W \les \|f\|_{L^\infty_t} \|\tilde f\|_{L^\infty_t}
$$
and likewise for all other combinations: $\|[f (B_1+B_2) U_2] \ast [U_1 B]\|_W \les \|f\|_{L^\infty_t} \|B\|_W$,
$$
\|B U_2 \ast [U_1 f (B_1+B_2)]\|_W \les \|f\|_{L^\infty_t} \|B\|_W,\ \|B U_2 \ast U_1 \tilde B\|_W \les \|B\|_W \|\tilde B\|_W.
$$
\end{corollary}

\begin{proof} This follows immediately from the listed properties of the $\mc X_t$ spaces.
\end{proof}

Hence the only exceptional term is $B_1+B_2$, corresponding to
$$
T_1 + T_2 = (\one + iT_{11})^{-1} \ast T_{12} \ast T_{21}.
$$
Consider the following algebra $\W$ and its completion $\ov \W = \W \oplus \C I$:
\begin{definition}\lb{defwalg} Let $\W$ be the algebra defined by
$$
\W = [L^\infty_t (B_1+B_2)] \oplus W = \{(B(t))_{t \in \R} \mid \|B(t)\|_{\mc X_t} \les \langle t \rangle^{-\3/2}\}.
$$
with the algebra operation $(B, \tilde B) \mapsto B U_2 \ast U_1 \tilde B$.
\end{definition}
Then $W$ is a bilateral ideal in $\W$ and $\W / W$ is a nilpotent algebra, $(\W / W)^2 = 0$.

The definition of $\mc X_t$ has the advantage that it includes, in a sense that will be clear later, an extra $\langle t \rangle^{-\3/2}$ decay rate, so the $\langle t \rangle^{-\3/2}$ weight in Definition \ref{defwalg} corresponds to an overall decay rate of $\langle t \rangle^{-\3}$.

However, certain steps of the proof, especially computation (\ref{mainest}) in Wiener's Theorem, cannot be carried out in this setting, because, for each $t$, $T(t)$ will involve a combination of several $\mc X_s$ spaces for multiple values of $s$. This is not detrimental to the proof, if small values of $s$ are penalized by appropriate weights. To account for this, we define the following Banach spaces:

\begin{definition}\lb{defxtilde} For $\alpha \geq 0$, $\beta \in [0, 1]$, consider the space
$$
\frak X_{\alpha, \beta} = \Big\{\frak B = \int_0^\infty B(\tau) \dd \tau \mid \int_0^\infty (1 + \beta \tau)^{-\alpha} \|B(\tau)\|_{\mc X_\tau} \dd \tau < \infty\Big\},
$$
with the norm given by the infimum of the above integral over all valid decompositions (possibly singular):
$$
\|\frak B\|_{\frak X_{\alpha, \beta}} = \inf \Big\{\int_0^\infty (1+\beta\tau)^{-\alpha} \|B(\tau)\|_{\mc X_\tau} \dd \tau \mid \frak B = \int_0^\infty B(\tau) \dd \tau \Big\}.
$$
Let $\frak X_0 = \frak X_{0, 1}$ and also let
$$
\frak W_\beta = \{(\frak B(t))_{t \in \R} \mid \|\frak B(t)\|_{\frak X_{\3/2, \beta}} \les \langle t \rangle^{-\3/2} (1+\beta t)^{-\3/2},\ \|\frak B(t)\|_{\frak X_0} \les \langle t \rangle^{-\3/2}\},\ \ov{\frak W}_\beta = \frak W_\beta \oplus \C \one,
$$
with the appropriate norm:
$$
\|\frak B\|_{\frak W} = \sup_{t \in \R} \langle t \rangle^{\3/2} (1+\beta t)^{\3/2} \|\frak B(t)\|_{\frak X_{\3/2, \beta}} + \sup_{t \in \R} \langle t \rangle^{\3/2} \|\frak B(t)\|_{\frak X_0}.
$$
\end{definition}


Clearly $\frak X_{\alpha, \beta} \subset \frak X_{\tilde \alpha, \tilde \beta}$ whenever $\alpha \leq \tilde \alpha$, $\beta \geq \tilde \beta$. It is easy to check that for $\alpha \geq 0$, $\beta \in [0, 1]$,
$$
[\chi_{\tau \geq 0} (1+\beta\tau)^\alpha L^1_\tau] \ast [\chi_{\tau \geq 0} L^1_\tau] \subset [\chi_{\tau \geq 0} \langle \tau \rangle^\alpha L^1_\tau].
$$
Consequently, and since $\mc X_t U_2 U_1 \mc X_s \subset \mc X_{t+s}$,
$$
\frak X_{\alpha, \beta} U_2 U_1 \frak X_0 \subset \frak X_{\alpha, \beta}.
$$

Moreover, $\frak X_{\alpha, \beta}$ contains all $\mc X_t$ spaces, each represented through a Dirac measure in $\tau$, for $t \geq 0$:
$$
\|B\|_{\frak X_{\alpha, \beta}} \les (1+\beta t)^{-\alpha} \|B\|_{\mc X_t}.
$$
For $\3>2$, also note that
$$
\|f \ast g\|_{\langle t \rangle^{-\3/2} L^\infty_t} \les \|f\|_{\langle t \rangle^{-\3/2} L^\infty_t} \|g\|_{\langle t \rangle^{-\3/2} L^\infty_t}
$$
and, uniformly for $\beta \in [0, 1]$,
$$
\|f \ast g\|_{\langle t \rangle^{-\3/2} (1+\beta t)^{-\3/2} L^\infty_t} \les \|f\|_{\langle t \rangle^{-\3/2} (1+\beta t)^{-\3/2} L^\infty_t} \|g\|_{\langle t \rangle^{-\3/2} L^\infty_t} + \|f\|_{\langle t \rangle^{-\3/2} L^\infty_t} \|g\|_{\langle t \rangle^{-\3/2} (1+\beta t)^{-\3/2} L^\infty_t}.
$$
Therefore, for each $\beta \in [0, 1]$, $\W \subset \frak W_\beta$, $\frak W_0 \subset \frak W_\beta$, $\frak W_\beta$ is also a Banach algebra, i.e.\;$\frak W_\beta U_2 \ast U_1 \frak W_\beta \subset \frak W_\beta$, and furthermore $\frak W_\beta$ has the generalized Leibniz property (\ref{leib}) with respect to $\frak W_0$:
$$
\|\frak B\|_{\frak W_\beta} \leq \|\frak B\|_{\W},\ \|\frak B_1 U_2 \ast U_1 \frak B_2\|_{\frak W_\beta} \les \|\frak B_1\|_{\frak W_\beta} \|\frak B_2\|_{\frak W_0} + \|\frak B_1\|_{\frak W_0} \|\frak B_2\|_{\frak W_\beta}.
$$

For future reference, we also record the following useful property:
\begin{lemma}\lb{pro} The algebra $\chi_{s_1, s_2 \geq 0} \langle s_1 + s_2 \rangle^{-\3/2} L^\infty_{s_1, s_2}$ has property (\ref{proper}) as a subalgebra of $\B(L^1_s) \cap \B(L^\infty_s)$.
\end{lemma}
\begin{proof} The statement is equivalent to asking that for any $T \in \B(L^1_{s_2}) \cap \B(L^\infty_{s_2})$
$$
\langle \langle s_1 + s_2 \rangle^{-\3/2}, T \langle s_2 + s_3 \rangle^{-\3/2} \rangle \les \|T\|_{\B(L^1_{s_2}) \cap \B(L^\infty_{s_2})} \langle s_1 + s_3 \rangle^{-\3/2}.
$$
Indeed,
$$
\langle \langle s_1 + s_2 \rangle^{-\3/2}, T \langle s_2 + s_3 \rangle^{-\3/2} \rangle \les \|\langle s_1 + s_2 \rangle^{-\3/2}\|_{L^1_{s_2}} \|T\|_{\B(L^\infty_{s_2})} \|\langle s_2 + s_3 \rangle^{-\3/2}\|_{L^\infty_{s_2}} \les \langle s_3 \rangle^{-\3/2},
$$
and same after swapping $s_1$ and $s_3$. Finally,
$$
\min(\langle s_3 \rangle^{-\3/2}, \langle s_1 \rangle^{-\3/2}) \les \langle s_1 + s_3 \rangle^{-\3/2}.
$$
\end{proof}

\subsubsection{The Fourier transform} Next, we show that the Fourier transform of elements of $\W$ can also be handled within $\W$.

As an example, the Fourier transform of $T_1 = T_{12} \ast T_{21}$ belongs to $U_1 X_0 U_2$ and is given by $U_1 \widehat B_1(\lambda) U_2$, where
$$
\widehat B_1(\lambda) = \chi_{s_1, s_2 \geq 0}(s_1, s_2) e^{-s_1 A}(y_1, y) \delta_{y = \tilde y} e^{-s_2 A}(\tilde y, y_2) [(v_2(x_1, y_1) e^{i(s_1+s_2)(-\Delta_{x_1}-\lambda)} v_1(x_1, y_2)) \otimes I].
$$
Here $X_0$ is the Banach space defined by Definition \ref{defx}, after setting $t=0$. As stated above, all the $X_t$ norms are comparable, so the Fourier transform is also uniformly in $X_t$ for any given $t$.

\begin{lemma}\lb{fourier_transform_w} The Fourier transforms of $T_1$, $T_2$, and of elements $T \in U_1 W U_2$
$$
\widehat T(\lambda) = \int_0^\infty e^{-it\lambda} T(t) \dd t
$$
are of the form $\widehat T(\lambda) = U_1 \widehat B(\lambda) U_2$, where $\widehat B(\lambda)$ is in $X_0$:
$$
\sup_{\Im \lambda \leq 0} \|\widehat B_1(\lambda)\|_{X_0} = 
\sup_{\Im \lambda \leq 0} \|\widehat B_1(\lambda)\|_{\langle s_1 + s_2 \rangle^{-\3/2} L^\infty_{s_1, s_2} (L^\infty_{y_2} L^1_{\tilde y, y_1, y} \cap L^\infty_{y_1} L^1_{y, y_2, \tilde y}) \B(\frak S_1)} < \infty.
$$
Then $\displaystyle \sup_{\Im \lambda \leq 0} \|\widehat B(\lambda)\|_{X_0} < \infty$, for $B_1$ and $B_2$. For elements $T \in U_1 W U_2$
$$
\sup_{\Im \lambda \leq 0} \|\widehat B(\lambda)\|_{X_0} \les \|T\|_W.
$$
\end{lemma}

\begin{proof} The conclusion follows by Minkowski's inequality.
For $T_1$ and $T_2$, it is due to the fact that
$$
\int_0^\infty \langle t \rangle^{-\3/2} \delta_{t-s_2}(s_1) \dd t \les \langle s_1 + s_2 \rangle^{-\3/2}.
$$
For $T_3$, $T_4$, and more generally for elements of $U_1 W U_2$, it is due to the fact that, see (\ref{triangle}),
$$
\int_0^\infty \langle t \rangle^{-\3/2} \langle t-(s_1+s_2) \rangle^{-\3/2} \dd t \les \langle s_1 + s_2 \rangle^{-\3/2}.
$$
\end{proof}

Lemma \ref{fourier_transform_w} has a partial converse: every operator in $X_0$, times a sufficiently rapidly decaying function, is in $W$. This requires more than $\3/2$ powers of decay.
\begin{lemma} If $\chi:\R \to \C$ is such that $\chi(t) \les \langle t \rangle^{-\3}$ and $\tilde T \in X_0$, then $\chi \tilde T \in W$.
\end{lemma}
\begin{proof}
Note that
$$
\langle t \rangle^{-\3/2} \langle s_1 + s_2 \rangle^{-\3/2} \les \langle t - (s_1 + s_2) \rangle^{-\3/2}.
$$
Then, on top of that, another $\3/2$ powers of decay ensure that $\chi \tilde T \in W$.
\end{proof}

Together, these two statements characterize the Fourier transform on $\W$.

The same statements for ${\frak W_\beta}$ and ${\frak X_0}$, that the Fourier transform of an element of $\frak W_\beta$ is in ${\frak X_0}$ and that an element of ${\frak X_0}$, times a sufficiently rapidly decaying function of $t$, is in ${\frak W_\beta}$, obviously follow from Definition \ref{defxtilde}.

So the Fourier transform of $\widehat{S_1 \ast S_2}(\lambda)$ is of the form
$$
\widehat{S_1 \ast S_2}(\lambda) = U_1 B_\lambda U_2,
$$
where $B_\lambda \in X_0$, uniformly for $\Im \lambda \leq 0$.

We still need to prove appropriate bounds for
$$\begin{aligned}
(I+\widehat {S^2}(\lambda))^{-1} &= \begin{pmatrix}
(I + \widehat{S_1 \ast S_2}(\lambda))^{-1} & 0 \\
0 & (I + \widehat{S_2 \ast S_1}(\lambda))^{-1}
\end{pmatrix} \\
&= \begin{pmatrix}
(I + \widehat S_1(\lambda)\widehat S_2(\lambda))^{-1} & 0 \\
0 & (I + \widehat S_2(\lambda)\widehat S_1(\lambda))^{-1}
\end{pmatrix}.
\end{aligned}$$

This will be accomplished through a succession of lemmas below, leading to Lemma \ref{bdd_inverse}.

We shall invert $I+\widehat{S_1 \ast S_2}(\lambda)$ using Fredholm's alternative Lemma \ref{fredholm} and prove that its inverse is also contained in $U_1 X_0 U_2 \oplus \C I$.

For each $\Im \lambda \leq 0$, we have proved that $\widehat{S_1 \ast S_2}(\lambda) = U_1 B_\lambda U_2 \in U_1 X_0 U_2$. As a rule, operators in this space are not compact. Hence, proving the existence of the inverse can be complicated.

Moreover, we would like to prove the existence of an inverse of a specific form: one that has $U_1$ as its initial segment, has $U_2$ as its final segment, and can be approximated by finite-rank operators in the middle, in the specified norm.



The initial and final segments $U_1$ and $U_2$ cannot be approximated by finite-rank operators and are not compact, this being one of the reasons why they need a special treatment. Nor is $B_\lambda$ or $U_1 B_\lambda U_2$ compact, for that matter.

As a rule, $e^{-it\Delta}$ becomes compact and norm-continuous as a function of $t$ if it has localizing weights at both ends, but is not compact if one or both weights are missing.

The issue with $S_1 \ast S_2$ and with its Fourier transform (and analogously with $S_2 \ast S_1$) is that only the $x_1$ factor has weights at both ends, while the $x_2$ factor does not.

Going into more detail, the terms $\widehat B_3$ and $\widehat B_4$ are completely continuous in $X_0$, but $\widehat B_1$ and $\widehat B_2$ are not, because they contain the identity $I_{x_2}$.

Nevertheless, $\widehat{S_1 \ast S_2}$ is conjugate to an operator with the right properties and this turns out to be sufficient. We handle the lack of compactness as follows.

Consider Banach spaces $X$, $Y$, and $Z$ and an operator of the form
$$
L_1BL_2:X \to X,
$$
where $L_1: Z \to X$ and $L_2:X \to Y$ are fixed initial and final segments, while $B:Y \to Z$ can be any operator in some fixed operator space.

\begin{lemma}
The following formula is true: if $I+BL_2L_1:Z \to Z$ is invertible, then
\be\lb{inverse_formula}
(I+L_1BL_2)^{-1}=I+L_1CL_2,
\ee
where
$$
C=-(I+BL_2L_1)^{-1}B.
$$
\end{lemma}
\begin{proof} A direct computation shows that
$$\begin{aligned}
&(I-L_1(I+BL_2L_1)^{-1}BL_2)(I+L_1BL_2) \\
&= I - L_1(I+BL_2L_1)^{-1}BL_2 + L_1BL_2 - L_1(I+BL_2L_1)^{-1}BL_2 L_1BL_2 \\
&= I + L_1[-(I+BL_2L_1)^{-1} + I - (I+BL_2L_1)^{-1} BL_2L_1] BL_2 \\
&= I.
\end{aligned}$$
\end{proof}

This formula works together with Fredholm's alternative, Lemma \ref{fredholm}. Keeping the same notation, one has the following:

\begin{lemma} Suppose that $B L_2 L_1$ is completely continuous in $W \subset \B(Z)$ and that the equation
$$
g = - L_1 B L_2 g
$$
has no nonzero solution. Then $I + L_1 B L_2$ is invertible and $(I + L_1 B L_2)^{-1} \in \B(X)$.
\end{lemma}
\begin{proof}
Indeed, suppose that $B L_2 L_1$ is completely continuous, meaning that it can be approximated by finite-rank operators. Then the non-invertibility of $I+BL_2L_1$ would imply the existence of a nonzero solution to the equation
$$
f = - B L_2 L_1 f.
$$
But then, applying $L_1$ to this equation on the left and setting $L_1 f = g$, we would get a solution of the equation
$$
g = - L_1 B L_2 g.
$$
If this equation has no nonzero solution, then $g=0$, which implies that $f = -B L_2 g$ must be zero.

This contradiction proves that $I+BL_2L_1$ is invertible, so $(I+L_1BL_2)^{-1}$ exists and has the same initial and final segments by (\ref{inverse_formula}). The conclusion follows.
\end{proof}

So the problem reduces to inverting $I+B_\lambda U_2 U_1$. By previous computations, this perturbation of the identity, which we shall prove is completely continuous, belongs to the following space $\widehat X$:
$$
B_\lambda U_2 U_1 \in X_0 U_2 U_1 \subset \widehat X,
$$
where
\begin{definition}
\be\lb{Xhat}
\widehat X := \chi_{s_1, s_2 \geq 0} \langle s_1+s_2 \rangle^{-\3/2} L^\infty_{s_1, s_2} [L^\infty_{y_1} L^1_y L^1_{y_2} L^\infty_{\tilde y}] \B(\frak S_1).
\ee
\end{definition}

The operator $B_\lambda$ is given by the sum
$$
B_\lambda = \widehat B_1(\lambda) + \widehat B_2(\lambda) + \widehat B_3(\lambda) + \widehat B_4(\lambda).
$$

In particular, the simplest term, see (\ref{def_b1}), is
$$
\widehat B_1(\lambda) = \chi_{s_1, s_2 \geq 0} e^{-s_1 A}(y_1, y) \delta_{y = \tilde y} e^{-s_2 A}(\tilde y, y_2) [(v_2(x_1, y_1) e^{i(s_1+s_2)(-\Delta_{x_1}-\lambda)} v_1(x_1, y_2)) \otimes I].
$$

Note that $\widehat B_1(\lambda) = e^{-i(s_1+s_2)\lambda} \widehat B_1(0)$. Hence its $X_0$ norm cannot go to zero as $\lambda$ goes to infinity. However, the $\widehat X$ norm of $\widehat B_1(\lambda) U_2 U_1$ does:
$$
\lim_{\lambda \to \infty} \|\widehat B_1(\lambda) U_2 U_1\|_{\widehat X} = 0.
$$
This follows from the finite-rank approximation below, see Lemma \ref{completely_cont}.

A more complicated term such as $\widehat B_3(\lambda)$, see (\ref{def_b3}), has the form (using $s_1$ and $s_3$ as variables)
$$\begin{aligned}[]
\widehat B_3(\lambda) = &\int \chi_{s_1, s_2, s_3 \geq 0} e^{-s_1 A}(y_1, y) v_2(x_2, y) e^{-is_2(-\Delta_{x_2}+iA+1\otimes V)} v_1(x_2, \tilde y) \\
&e^{-s_3 A}(\tilde y, y_2) [(v_2(x_1, y_1) e^{i(s_1+s_2+s_3)(-\Delta_{x_1}-\lambda)} v_1(x_1, y_2)) \otimes I_{x_2}] \dd s_2.
\end{aligned}$$

Next, we show that $B_\lambda U_2 U_1$ is completely continuous in $\widehat X$, so we can apply Fredholm's alternative, Lemma \ref{fredholm}.

%

\begin{lemma}\lb{completely_cont} Assume that the infinitesimal generator $A$ satisfies properties C1-C5 and that the parameter space $Y$ is compact. Then the operator $B_\lambda U_2 U_1$ can be approximated by finite-rank operators in the $\widehat X$ norm: for any $\Im \lambda \leq 0$ and $\epsilon>0$, there exist $N$ and $f_n, g_n, h_n, \tilde g_n, \tilde h_n, k_n, \tilde k_n, \ell_n, \tilde \ell_n$, $1 \leq n \leq N$, such that
\be\lb{compl_cont}
\Big\| B_\lambda U_2 U_1 - \sum_{n=1}^N f_n(s_1, s_2) g_n(y_1) h_n(y) \tilde g_n(y_2) \tilde h_n(\tilde y) [k_n(x_1)\otimes\tilde k_n(x_1)] [\ell_n(x_2)\otimes\tilde \ell_n(x_2)] \Big\|_{\widehat X} < \epsilon.
\ee
Moreover,
\be\lb{goesto0}
\lim_{\lambda \to \infty} \|B_\lambda U_2 U_1\|_{\widehat X} = 0.
\ee
\end{lemma}

\begin{proof}
The main issue in obtaining this finite-rank approximation is that if $f \in \langle t \rangle^{-\3/2} L^\infty_t$, it does not necessarily follow that
$$
\lim_{R \to \infty} \|\chi_{t \geq R} f\|_{\langle t \rangle^{-\3/2} L^\infty_t} = 0.
$$
In fact, the decay rate at infinity would have to be $o(t^{-\3/2})$ instead of $O(t^{-\3/2})$ for this to happen. However, we can circumvent this issue by using a different approximation near infinity, as in the proof of Lemma \ref{lemma_compact}.

The overall kernel is a sum of terms, each of which is the composition of at least two factors (coming from $B_\lambda$ and from $U_2U_1$), either in the $x_1$ and $y$ variables or in the $x_2$ and $y$ variables.

We can use the weaker ${\frak L_1}$ norm, defined by (\ref{eser}), instead of the weighted $L^\infty$ norm for each factor, and yet their composition will still be in $\langle s_1 + s_2 \rangle^{-\3/2} L^\infty_{s_1, s_2}$ because $\frak L_1 \ast \frak L_1 \subset \langle t \rangle^{-\3/2} L^\infty_t$.

Next, we find finite-rank approximations for these factors, which are of the following form, involving only one out of the two space variables:
$$
K_1(t) = v_2(x, y_1) e^{-it\Delta} v_1(x, y_2)
$$
and
$$
K_2(t) = v_2(x, y_1) e^{-it\Delta} e^{-tA}(y_1, y_2) v_1(x, y_2).
$$
We try to approximate $K_2$ in the following norm:
\be\lb{weaker_norm}
({\frak L_1})_t (L^\infty_{y_2} L^1_{y_1} \cap L^\infty_{y_1} L^1_{y_2}) \B(L^2_x)
\ee
and $K_1$ in
\be\lb{even_weaker_norm}
({\frak L_1})_t (L^\infty_{y_1} L^\infty_{y_2}) \B(L^2_x).
\ee

This suffices to handle the $\widehat {T_1} = \widehat{T_{12} \ast T_{21}}$ term of $\widehat{S_1 \ast S_2}$. For the other terms, we also need to consider factors of the form
$$
v_2(x, y) e^{it(-\Delta+V+iA)} v_1(x, y),
$$
which can also be approximated in the norm (\ref{weaker_norm}) by finite-rank operators, see Lemma \ref{compl_cont_1var}.

$K_2$ being approximable by simple functions in the norm (\ref{weaker_norm}) is the content of Lemma \ref{lemma_compact}.

As for $K_1$, it is completely continuous in (\ref{even_weaker_norm}) because we are assuming that $Y$ is compact. Indeed, since $v_1(x, y)$ and $v_2(x, y)$ are continuous in $y$, when $Y$ is compact they must be uniformly continuous.

Thus, there exists a finite open cover of $Y$ on which we can approximate them by constant functions of $y$:
$$
v_1(x, y) \sim \sum_{n=1}^N \chi_n(y) v_1(x, y_n).
$$
For fixed values of $y$, one can approximate $v_1$ and $v_2$ with bounded functions of compact support in $x$. Henceforth, the proof proceeds as for Lemma \ref{lemma_compact}.

Finally, the composition of these simple factors is still simple and the convolution of two ${\frak L_1}$ functions is in $\langle t \rangle^{-\3/2} L^\infty_t$. This proves (\ref{compl_cont}).

As $\lambda$ goes to infinity, the norm of the overall kernel goes to zero, because each term is the composition of at least two factors, each of which can be approximated by simple functions.

With no loss of generality, after approximating again, each simple function in $s$ will be either the characteristic function of a bounded interval or exactly of the form $\langle s_1 + s_2 \rangle^{-\3/2}$.

Let $K_1(s_1, s_2) \in \chi_{s_1, s_2 \geq 0} ({\frak L_1})_{s_1+s_2}$ and $K_2(s_2, s_3) \in \chi_{s_2, s_3 \geq 0} ({\frak L_1})_{s_2+s_3}$ be these simple functions. Checking each possible case, it follows that
$$
\int_0^\infty e^{-\lambda(s_1+s_2)} K_1(s_1, s_2) K_2(s_2, s_3) \dd s_2 \les \langle s_1 + s_3 \rangle^{-\3/2},
$$
with a constant that goes to zero when $\lambda \to \infty$.
\end{proof}


We shall next prove that the Kato--Birman operator $I + \widehat {S_1 \ast S_2}(\lambda)$ is invertible.

By Lemma \ref{fourier_transform_w} it follows that
$$
\widehat {S_1 \ast S_2}(\lambda) = U_1 B_\lambda U_2 \in U_1 X_0 U_2.
$$
Consequently, by (\ref{bound1}) and (\ref{bound2}),
$$
U_1 B_\lambda U_2 \in \B(L^1_y[L^2_{x_1} \otimes (L^1_{x_2} \cap L^2_{x_2})], L^1_y[L^2_{x_1} \otimes (L^2_{x_2} + L^\infty_{x_2})]).
$$
In order to show it is a bounded operator on $L^1_y \frak S_1$, we rewrite $\widehat {S_1 \ast S_2}(\lambda)$ using the following identity derived from \cite{pillet}, p.\;269:
\begin{lemma}\lb{pillet_computation}
\begin{multline*}
S_1 \ast S_2 = (\one + iT_{11})^{-1} \ast T_{12} \ast (\one + i T_{22})^{-1} \ast T_{21} = \\
= (\one + iT_{11})^{-1} \ast (\one+iv_2(x_1, y)e^{it(-\Delta_{x_1}+\Delta_{x_2}+iA-1\otimes V)}v_1(x, y)) - \one\\
= (\one - iv_2(x_1, y)e^{it(-\Delta_{x_1}+\Delta_{x_2}+iA+V\otimes 1)}v_1(x, y)) \ast (\one + iv_2(x_1, y)e^{it(-\Delta_{x_1}+\Delta_{x_2}+iA-1\otimes V)}v_1(x, y)) - \one.
\end{multline*}
Hence for every $\lambda$ such that $\Im \lambda \leq 0$,
\begin{multline*}
\widehat {S_1 \ast S_2}(\lambda) = \widehat S_1(\lambda) \widehat S_2(\lambda) = (I + i\widehat T_{11}(\lambda-i0))^{-1} i\widehat T_{12}(\lambda-i0) (I + i\widehat T_{22}(\lambda-i0))^{-1} i\widehat T_{21}(\lambda-i0) = \\
= I - (I + \widehat T_{11}(\lambda-i0))^{-1}(I+v_2(x_1, y)R^{iA}_{-1\otimes V}(\lambda-i0)v_1(x, y)) \\
= I - (I - v_2(x_1, y) R^{iA}_{V \otimes 1} v_1(x_1, y))(I+v_2(x_1, y)R^{iA}_{-1\otimes V}(\lambda-i0)v_1(x, y)).
\end{multline*}
Consequently,
\be\lb{weak_est}
S_1 \ast S_2 \in \langle t \rangle^{-\3/2} [L^\infty_{y_2} L^1_{y_1} \cap L^\infty_{y_1} L^1_{y_2}] [\B(L^2) \otimes \B(L^2)]
\ee
and
$$
\widehat {S_1 \ast S_2}(\lambda) \in \B(L^p_y \frak S_1).
$$
\end{lemma}
In both cases, the second factor has weights involving only $x_1$, but attached to a perturbation of the free evolution/resolvent involving only $x_2$. This factor is of a different type and not directly related to those that have appeared before in the proof.

\begin{observation} Estimate (\ref{weak_est}) is actually much easier to prove than the one we want, (\ref{point}), but this approach does not lead to the full rate of decay, which is $\langle t \rangle^{-\3}$.
\end{observation}

\begin{proof}
By Duhamel's identity and Proposition (\ref{prop26}),
$$
(\one + iT_{22})^{-1} = \one - iv_2(x_2, y)e^{it(-\Delta_{x_1}+\Delta_{x_2}+iA-1\otimes V)}v_1(x_2, y).
$$
Hence we get that
$$\begin{aligned}
&iT_{12} \ast (\one + iT_{22})^{-1} \ast iT_{21} = \\
&= iv_2(x_1, y) e^{it(-\Delta_{x_1}+\Delta_{x_2}+iA)} v_1(x_2, y) \big(\one - iv_2(x_2, y)e^{it(-\Delta_{x_1}+\Delta_{x_2}+iA-1\otimes V)} v_1(x_2, y)\big) \\
&iv_2(x_2, y) e^{it(-\Delta_{x_1}+\Delta_{x_2}+iA)} v_1(x_1, y)\\
&= iv_2(x_1, y) (e^{it(-\Delta_{x_1}+\Delta_{x_2}+iA)} - e^{it(-\Delta_{x_1}+\Delta_{x_2}+iA-1\otimes V)}) v_1(x_1, y) \\
&= iT_{11} - iv_2(x_1, y) e^{it(-\Delta_{x_1}+\Delta_{x_2}+iA-1\otimes V)} v_1(x_1, y).
\end{aligned}$$
Again, this follows by Duhamel's formula, which guarantees that the third row equals the second row.

Applying $(\one + iT_{11})^{-1}$ to the left, we obtain
$$\begin{aligned}
iS_1 \ast iS_2 &= (\one + iT_{11})^{-1} \ast iT_{11} - (\one + iT_{11})^{-1} \ast iv_2(x_1, y) e^{it(-\Delta_{x_1}+\Delta_{x_2}+iA-1\otimes V)} v_1(x_1, y) \\
&= \one - (\one + iT_{11})^{-1} - (\one + iT_{11})^{-1} \ast iv_2(x_1, y) e^{it(-\Delta_{x_1}+\Delta_{x_2}+iA-1\otimes V)} v_1(x_1, y),
\end{aligned}$$
which gives the desired result.

The first factor is bounded due to Proposition \ref{prop26}, as is the second factor because
\be\lb{est1}
v_1(x_1, y) e^{-it\Delta_{x_1}} v_2(x_1, y) \in \langle t \rangle^{-\3/2} L^\infty_t L^\infty_{y_1, y_2} \B(L^2)
\ee
and
$$
e^{it(\Delta_{x_2}+iA-1 \otimes V)} \in L^\infty_t [L^\infty_{y_2} L^1_{y_1} \cap L^\infty_{y_1} L^1_{y_2}] \B(L^2).
$$
The second inequality admits a simple probabilistic proof. Indeed, $e^{it(\Delta_{x_2}+iA-1 \otimes V)}$ has a kernel given by
\be\lb{est2}\begin{aligned}
[e^{it(\Delta_{x_2}+iA-1 \otimes V)}(y_1, y_2)] \psi_0 = &\E(\psi_\omega(t) \mid \psi_\omega(0) = \psi_0,\ \omega(0)=y_2,\ \omega(t) = y_1) \cdot\\
\cdot &\P(\omega(0)=y_2 \mid \omega(t) = y_1).
\end{aligned}\ee
But the second factor is exactly
$$
\P(\omega(0)=y_2 \mid \omega(t) = y_1) = e^{-tA}(y_1, y_2) \in [L^\infty_{y_2} L^1_{y_1} \cap L^\infty_{y_1} L^1_{y_2}],
$$
while the first factor is uniformly $L^2$-bounded, because the evolution is always unitary.

The computation for the resolvent is simply the Fourier transformed version of the one above, with the resolvent identity and symmetric resolvent identity replacing Duhamel's formula.
\end{proof}

As an aside, this formula for $\widehat{S_1 \ast S_2}(\lambda)$ still evidences its lack of compactness. Indeed, while the product of both factors
$$
[v_2(x_1, y) R^{iA}_{V \otimes 1} v_1(x_1, y)][I+v_2(x_1, y)R^{iA}_{-1\otimes V}(\lambda-i0)v_1(x, y)]
$$
is compact, each of them by itself is not compact, and they both appear in the formula outside their product as well. However, now both factors are $L^p_y \frak S_1$-bounded for $1 \leq p \leq 2$, so equation (\ref{impossible}) below makes sense.

We now prove the spectral condition, analogous to Lemma \ref{invertibility}, that implies the absence of bound states in the lower half-plane for the Liouville-type averaged equation (\ref{average_liouville}).

\begin{lemma}\lb{invertibility_liouville} Suppose that $V \in C_y (L^1_x \cap L^\infty_x)$ is real-valued, $\Im \lambda \leq 0$, and $f \in L^p_y \frak S_1$, $1 \leq p \leq 2$, is a solution to the equation
\be\lb{impossible}\begin{aligned}
f &= i\widehat S_1(\lambda) i \widehat S_2(\lambda) f = (I + i\widehat T_{11}(\lambda-i0))^{-1} i\widehat T_{12}(\lambda-i0) (I + i\widehat T_{22}(\lambda-i0))^{-1} i\widehat T_{21}(\lambda-i0) f.
\end{aligned}\ee
Then $f=0$.
\end{lemma}
\begin{proof} 
By Lemma \ref{pillet_computation}, equation (\ref{impossible}) becomes
\be\lb{aaa}\begin{aligned}
0 = (I + \widehat T_{11}(\lambda-i0))^{-1} (I + v_2(x_1, y) R_{-1\otimes V}^{iA}(\lambda-i0) v_1(x_1, y)) f.
\end{aligned}\ee
Applying $I + i\widehat T_{11}(\lambda-i0)$ to both sides of (\ref{aaa}), we obtain
$$
f = -v_2(x_1, y) R_{-1\otimes V}^{iA}(\lambda-i0) v_1(x_1, y) f.
$$

If $f \in L^2_y \frak S_1 \subset L^2_{y, x_1, x_2}$, the pairing
$$
\langle \sgn V f, f \rangle = \langle v_1(x_1, y) f, R_{-1\otimes V}^{iA}(\lambda-i0) v_1(x_1, y) f \rangle
$$
is finite and real-valued. Then for some $c>0$
$$\begin{aligned}
0&=-\Im \langle R_{-1\otimes V}^{iA}(\lambda-i0) v_1(x_1, y) f, v_1(x_1, y) f \rangle \\
&= c \langle E_\lambda(-\Delta_{x_2}+V(x_2, y)) v_1(x_1, y) f, v_1(x_1, y) f \rangle +\\
&+ \langle A P_{>0} R_{-1 \otimes V}^{iA}(\lambda-i0) v_1(x_1, y) f, P_{>0} R_{-1 \otimes V}^{iA}(\lambda) v_1(x_1, y) f \rangle.
\end{aligned}$$
See (\ref{spectral_projection}) for the proof. Here $E_\lambda(-\Delta_{x_2}+V(x_2, y))$ is the spectral density of $-\Delta_{x_2}+V(x_2, y)$.

Then let
$$
g=R_{-1\otimes V}^{iA}(\lambda-i0) v_1(x_1, y) f.
$$

Since for any $\epsilon>0$ $P_{\geq \epsilon} A$ is a strictly positive operator, we obtain that $P_{\geq \epsilon} g = 0$, so $P_{>0}g=0$ and $g$ must be of the form $g=g(x_1, x_2)h(y)$, where $h$ is the ground state of $A$. Next, $g$ is a distributional solution to the equation
$$
(-\Delta_{x_1} + \Delta_{x_2} + V(x_1, y) - V(x_2, y) - \lambda) g(x_1, x_2) h(y) = 0.
$$
By the nontriviality Assumption \ref{nontrivial} it follows that there exists some open set $\mathcal O'$ such that $g(x_1, x_2)=0$ for all $(x_1, x_2) \in \R^\3 \times \mathcal O'$. By unique continuation (see Lemma 5.6 of \cite{pillet}) it follows that $g=0$. Therefore $f=0$ and the proof is concluded in this case.

If $f \in L^1_y \frak S_1$, then, as in the proof of Lemma \ref{invertibility}, we need to bootstrap from $L^1_y$ to $L^2_y$. It suffices to prove that there exists some $\epsilon>0$ such that, whenever $1/p - 1/{\tilde p} = \epsilon$, $f \in L^p_y \frak S_1$ implies $f \in L^{\tilde p}_y \frak S_1$.

We represent $R^{iA}_{-1\otimes V}(\lambda-i0)$ as the Fourier transform of
$$
i\chi_{t>0}(t) e^{it(-\Delta_{x_1}+\Delta_{x_2}+iA-1\otimes V)},
$$
which we decompose into two components,
$$
i\chi_{t>0}(t) \eta(t/\epsilon) e^{it(-\Delta_{x_1}+\Delta_{x_2}+iA-1\otimes V)} + i\chi_{t>0}(t) (1-\eta(t/\epsilon)) e^{it(-\Delta_{x_1}+\Delta_{x_2}+iA-1\otimes V)}.
$$
The second component is smoothing in the $y$ variable, taking $L^p$ to $L^{\tilde p}$, due to condition C3 on $e^{-tA}$, together with (\ref{est1})  and (\ref{est2}):
$$
\|v_2(x_1, y) i\chi_{t>0}(t) (1-\eta(t/\epsilon)) e^{it(-\Delta_{x_1}+\Delta_{x_2}+iA-1\otimes V)} v_1(x_1, y) f\|_{L^1_t L^{\tilde p}_y \frak S_1} \les C(\epsilon) \|f\|_{L^p_y \frak S_1},
$$
so the same is true for its Fourier transform:
$$
\sup_{\Im \lambda \leq 0}\|v_2(x_1, y) (i\chi_{t>0}(t) (1-\eta(t/\epsilon)) e^{it(-\Delta_{x_1}+\Delta_{x_2}+iA-1 \otimes V)})^\wedge(\lambda) v_1(x_1, y)\|_{\B(L^p_y \frak S_1, L^{\tilde p}_y \frak S_1)} \leq C(\epsilon).
$$
The first piece is not smoothing, but it and its Fourier transform become small in norm as $\epsilon \to 0$:
$$\begin{aligned}
\lim_{\epsilon \to 0} \sup_{\Im \lambda \leq 0} \|v_2(x_1, y) (i\chi_{t>0}(t) \eta(t/\epsilon) e^{it(-\Delta_{x_1}+\Delta_{x_2}+iA-1 \otimes V)})^\wedge(\lambda) v_1(x_1, y)\|_{\B(L^{\tilde p}_y \frak S_1)} = 0.
\end{aligned}$$
This is again a consequence of (\ref{est1}) and (\ref{est2}).

Consequently, for small $\epsilon>0$
$$\begin{aligned}
f = -(I - (v_2 i\chi_{t>0}(t) \eta(t/\epsilon) e^{it(-\Delta_{x_1}+\Delta_{x_2}+iA_y)})^\wedge(\lambda) v_1)^{-1} \\
(i\chi_{t>0}(t) (1-\eta(t/\epsilon)) v_2 e^{it(-\Delta_{x_1}+\Delta_{x_2}+iA_y)} v_1)^\wedge(\lambda) f,
\end{aligned}$$
where the inverse is given by a geometric series for small $\epsilon$.

After a finite number of such bootstrapping steps, $f \in L^2_y \frak S_1$ and we can continue as above.
\end{proof}

We can now prove the invertibility of $I+\widehat{S_1 \ast S_2}(\lambda)$ in the desired setting, $U_1 X_0 U_2$. Recall that $X_0 \subset \frak X_0$ by the discussion after Definition \ref{defxtilde}.
\begin{lemma}\lb{bdd_inverse} For $\Im \lambda \leq 0$, $I+\widehat{S_1 \ast S_2}(\lambda)$ is invertible and $(I+\widehat{S_1 \ast S_2}(\lambda))^{-1} - I \in U_1 X_0 U_2$ is norm-continuous and its norm goes to zero as $\lambda$ goes to infinity. Therefore
$$
\sup_{\Im \lambda \leq 0} \|(I+\widehat{S_1 \ast S_2}(\lambda))^{-1}-I\|_{U_1 X_0 U_2} < \infty.
$$
\end{lemma}

\begin{proof}
By virtue of its complete continuity in $\widehat X$, see (\ref{Xhat}), the operator $B_\lambda U_2 U_1$ is bounded and compact on $\langle s \rangle^{-\3/2} L^\infty_s L^\infty_{y_1} L^1_y \frak S_1$:
$$
\widehat X \subset \B(\langle s \rangle^{-\3/2} L^\infty_s L^\infty_{y_1} L^1_y \frak S_1).
$$

Next, we examine $(I + B_\lambda U_2 U_1)^{-1}$ in these settings. By Fredholm's alternative, Lemma \ref{fredholm}, either $I + B_\lambda U_2 U_1$ is invertible or else the equation
\be\lb{EQF}
F = - B_\lambda U_2 U_1 F
\ee
must have a nonzero solution $F \ne 0$ in $\langle s \rangle^{-\3/2} L^\infty_s L^1_{y_1} L^1_y \frak S_1$.
Then $U_1 F$ will be in $L^1_{y_1} \frak S_1$ by (\ref{u1est}) and solve the equation
$$
U_1 F = - (U_1 B_\lambda U_2) U_1 F.
$$
Since $U_1 B_\lambda U_2 = \widehat{S_1 \ast S_2}(\lambda)$, by Lemma \ref{invertibility_liouville} this implies that $U_1 F = 0$. But then by (\ref{EQF}) $F=0$, which is a contradiction, so $I + B_\lambda U_2 U_1$ must be invertible in $\B(\langle s \rangle^{-\3/2} L^\infty_s L^1_{y_1} L^1_y \frak S_1)$:
$$
(I + B_\lambda U_2 U_1)^{-1} \in \B(\langle s \rangle^{-\3/2} L^\infty_s L^\infty_{y_1} L^1_y \frak S_1).
$$

We next need to show that the inverse is in $\widehat X$. This will follow if the desired subalgebra (eventually $\widehat X$, but we have to start with others) has property (\ref{proper}). First, note that the following statement is false:
\be\lb{false}
[\chi_{s_1, s_2 \geq 0} \langle s_1 + s_2 \rangle^{-\3/2} L^\infty_{s_1, s_2}] \B(\langle s \rangle^{-\3/2} L^\infty_s) [\chi_{s_1, s_2 \geq 0}  \langle s_1 + s_2 \rangle^{-\3/2} L^\infty_{s_1, s_2}] \subset \chi_{s_1, s_2 \geq 0} \langle s_1 + s_2 \rangle^{-\3/2} L^\infty_{s_1, s_2},
\ee
because $\|\langle s_1 + s_2\rangle^{-\3/2}\|_{\langle s_1 \rangle^{-\3/2} L^\infty_{s_1}}$ has no decay as $s_2 \to \infty$; it is always $1$.

Instead, one can show that
$$
(\ref{false}) \subset \B(L^\infty_s),
$$
which implies that
$$
(I + B_\lambda U_2 U_1)^{-1} \in \B(L^\infty_s L^1_{y_1} L^1_y \frak S_1).
$$

However, it is easy to check that replacing $\B(\langle s \rangle^{-\3/2} L^\infty_s)$ by $\B(L^1_s) \cap \B(L^\infty_s)$ works, i.e.\;the following statement is true, as per Lemma \ref{pro}:
\be\lb{true}
[\chi_{s_1, s_2 \geq 0} \langle s_1 + s_2 \rangle^{-\3/2} L^\infty_{s_1, s_2}] [\B(L^1_s) \cap \B(L^\infty_s)] [\chi_{s_1, s_2 \geq 0}  \langle s_1 + s_2 \rangle^{-\3/2} L^\infty_{s_1, s_2}] \subset \chi_{s_1, s_2 \geq 0} \langle s_1 + s_2 \rangle^{-\3/2} L^\infty_{s_1, s_2}.
\ee

More generally, an integral kernel in $\langle s_1 + s_2 \rangle^{-\3/2} L^\infty_s$ is bounded on $L^p_s$, $1 \leq p \leq \infty$, as well as on $\langle s \rangle^\alpha L^{p, \infty}_s$, $1<p\leq\infty$, $|\alpha| \leq \3/2-1/p$.

Hence, in particular, $B_\lambda U_2 U_1$ also acts completely continuously on $L^1_s L^\infty_{y_1} L^1_y \frak S_1$. Then, using the Fredholm alternative as above leads to
$$
(I + B_\lambda U_2 U_1)^{-1} \in \B(L^1_s L^\infty_{y_1} L^1_y \frak S_1).
$$

Consequently, due to (\ref{true}),
$$
(I + B_\lambda U_2 U_1)^{-1} - I \in \chi_{s_1, s_2 \geq 0} \langle s_1+s_2 \rangle^{-\3/2} L^\infty_{s_1, s_2} \B(L^\infty_{y_1} L^1_y \frak S_1).
$$

Note that the operator subalgebra of kernels $L^\infty_{y_1} L^1_y L^1_{y_2} L^\infty_{\tilde y} \subset \B(L^\infty_{y_1} L^1_y)$ also has property (\ref{proper}):
$$
[L^\infty_{y_1} L^1_y L^1_{y_2} L^\infty_{\tilde y}] \B(L^\infty_{y_1} L^1_y) [L^\infty_{y_1} L^1_y L^1_{y_2} L^\infty_{\tilde y}] \subset L^\infty_{y_1} L^1_y L^1_{y_2} L^\infty_{\tilde y}.
$$

Together, these two facts imply that $\widehat X$ has property (\ref{proper}) as a subalgebra of $\B(L^1_s L^1_{y_1} L^1_y \frak S_1) \cap \B(L^\infty_s L^1_{y_1} L^1_y \frak S_1)$, so
$$
(I + B_\lambda U_2 U_1)^{-1} - I \in \chi_{s_1, s_2 \geq 0} \langle s_1+s_2 \rangle^{-\3/2} L^\infty_{s_1, s_2} [L^\infty_{y_1} L^1_y L^1_{y_2} L^\infty_{\tilde y}] \B(\frak S_1) = \widehat X.
$$


So the inverse is also in $\widehat X$, as per Fredholm's alternative, Lemma \ref{fredholm}, and is completely continuous in there.

By formula (\ref{inverse_formula}),
$$
(I + \widehat{S_1 \ast S_2}(\lambda))^{-1} - I = - U_1 (I + B_\lambda U_2 U_1)^{-1} B_\lambda U_2 \in U_1 \widehat X X_0 U_2.
$$
This implies the desired conclusion, because $\widehat X X_0 \subset X_0$.

Finally, the norm $\|(I + \widehat{S_1 \ast S_2}(\lambda))^{-1} - I\|_{X_0}$ goes to zero as $\lambda \to \infty$ because the same is true for $\|B_\lambda U_2 U_1\|_{\widehat X}$, as per (\ref{goesto0}).
\end{proof}




\subsubsection{Completing the proof} Continuity under translation (\ref{translation_high}) can be examined in the same context. The expression $S_1 \ast S_2$ is not continuous under translation, because the terms $T_1(t)$ and $T_2(t)$ are given by the singular measure $\delta_{t-s_1}(s_2)$ for each $t$, so are not comparable for different values of $t$.

However, it suffices to prove the continuity under translation of the higher powers, which are in $W \subset \W$, hence non-singular.

If $T \in W$, then for each $t \geq 0$
$$
T(t) \in \langle t - (s_1+s_2) \rangle^{-\3/2} L^\infty_{s_1, s_2},
$$
but these norms are comparable for different values of $t$, see (\ref{norm_comp}), and the underlying space is the same. The singular part $\W \setminus W$ is not comparable across different values of $t$, but it is absent from higher powers.

Also recall that $\W \subset {\frak W}$ by the discussion following Definition \ref{defxtilde}.

\begin{lemma}\lb{translat} For $N \geq 4$
$$
\lim_{\epsilon \to 0} \|(S_1 \ast S_2)^N(t+\epsilon) - (S_1 \ast S_2)^N(t)\|_{U_1WU_2} = 0.
$$
\end{lemma}
\begin{proof} The terms $T_3$ and $T_4$ are continuous under translation, after cutting away an $\epsilon$-length interval near zero, with no need to raise them to higher powers. Indeed, $T_\3$ is given by
$$\begin{aligned}
T_3(t) &= T_{12} \ast [(\one + i T_{22})^{-1} - \one] \ast T_{21} = U_1 B_3(t) U_2,
\end{aligned}$$
where
$$\begin{aligned}
B_3(t)(s_1, s_2) = &e^{-s_1 A}(y_1, y) e^{-s_2 A}(\tilde y, y_2) \\
&[(v_2(x_1, y_1) e^{-it\Delta_{x_1}} v_1(x_1, y_2)) \otimes (v_2(x_2, y) e^{i(t-s_1-s_2)(\Delta_{x_2}+iA)}(y, \tilde y) v_1(x_2, \tilde y))].
\end{aligned}$$

With no loss of generality, it suffices to prove continuity under translation in the case when the potentials $v_1$ and $v_2$ are bounded and compactly supported in $x$.

Then $v_2 e^{-it\Delta} v_1 \in \B(L^2_x)$ is a norm-continuous function of $t$, except near $t=0$. As $t \to \infty$,
$$
v_2 e^{-it\Delta} v_1 = v_2 (4\pi i t)^{-\3/2} v_1 + O(t^{-(\3+2)/2}).
$$
We can neglect the lower-order term for large $t$ and the leading-order term is clearly continuous under translation. For $t$ within a bounded interval, norm continuity must be uniform. This suffices to prove the continuity under translation of $\chi_{> \epsilon} v_2 e^{-it\Delta} v_1$.

Likewise, the continuity under translation of $\chi_{>\epsilon} (\one + i T_{22})^{-1} - \one$ follows from the continuity under translation of the original expression $\chi_{>\epsilon}(t) T_{22}(t)$.

We have thus proved that $\chi_{>\epsilon}(t) B_3(t)$ is continuous under translation. Therefore, higher powers of $B_3$ are continuous without the cutoff.

Same goes for $B_4$. The other term, $B_1+B_2$, behaves similarly after being squared or multiplied by $B_3$ or $B_4$, hence $N \geq 4$ is sufficient.
\end{proof}

We next prove the main trace-class dispersive estimates for the dissipative Liouville-type equation (\ref{average_liouville}), stated in Proposition \ref{prop27}.

\begin{proof}[Proof of Proposition \ref{prop27}]
Recall the main setup:
$$\begin{aligned}
\one + iT &= \bpm \one + iT_{11} & 0 \\ 0 & \one + iT_{22} \epm \bpm \one & i(\one + iT_{11})^{-1} \ast T_{12} \\ i(\one + iT_{22})^{-1} \ast T_{21} & \one \epm \\
& = (\one + iT_{diag}) \ast (\one + i(\one + iT_{diag})^{-1} \ast S_0) \\
&:= (\one + iT_{diag}) \ast (\one + iS),
\end{aligned}$$
where we let
$$\begin{aligned}
S := (\one + iT_{diag})^{-1} \ast S_0 = \bpm 0 & (\one + i T_{11})^{-1} \ast T_{12} \\ (\one + iT_{22})^{-1} \ast T_{21} & 0 \epm.
\end{aligned}$$
Then
\be\lb{Tinverse}\begin{aligned}
(\one + iT)^{-1} = (\one + iS)^{-1} \ast (\one + iT_{diag})^{-1} = (\one+S^2)^{-1} \ast (\one-iS) \ast (\one+iT_{diag})^{-1}.
\end{aligned}\ee
Below we explicitly estimate the last two factors and control the first one by Wiener's theorem.

Let $S_1$ and $S_2$ be the two nonzero (off-diagonal) components of the matrix $S$:
\be\lb{s1s2}
S_1 := (\one + iT_{11})^{-1} \ast T_{12},\ S_2 := (\one + iT_{22})^{-1} \ast T_{21}.
\ee
Therefore
\be\lb{S2}
S^2 = \bpm S_1 \ast S_2 & 0 \\ 0 & S_2 \ast S_1 \epm \text{ and } \widehat S^2(\lambda) = \bpm \widehat S_1(\lambda) \widehat S_2(\lambda) & 0 \\ 0 & \widehat S_2(\lambda) \widehat S_1(\lambda) \epm,
\ee
where $S^2$ is the convolution of $S$ with itself.

The terms $S_1$ and $S_2$ lack two desirable properties: $L^1_y \frak S_1$ norm decay in $t$ and compactness of their Fourier transform for $\lambda \in \R$. Thus we study their square instead, which has both properties.

We can use Wiener's Theorem \ref{thm:Wiener} together with Proposition \ref{subalg} to invert $\one+S^2$ in $U_1 \ov {\frak W}_\beta U_2$, where $\ov{\frak W}_\beta$ is defined in Definition \ref{defxtilde}. Indeed, due to Lemma \ref{translat}, $\one+S^2$ has property (\ref{translation_high}) in $U_1 \ov {\frak W}_\beta U_2$, and the Fourier transform of $\one+S^2$ is invertible in $\frak X_0$ by Lemma \ref{bdd_inverse}. We obtain that
$$
(\one + S_1 \ast S_2)^{-1} - \one \in U_1 \ov {\frak W}_\beta U_2,
$$
with a norm independent of $\beta \in [0, 1]$, because the spaces $\frak W_\beta$ have property (\ref{leib}) with respect to $\frak W_0$.

Duhamel's identity for equation (\ref{average_liouville}) is equivalent to
\be\lb{duhamel_l}\begin{aligned}
&\chi_{t>0} e^{it(-\Delta_{x_1}+\Delta_{x_2}+iA+ V\otimes 1-1\otimes V)} =\chi_{t>0} e^{it(-\Delta_{x_1}+\Delta_{x_2}+iA)} + \\ &+i\int_{t>s_1>s_2>0}T_{0V}(t-s_1) (\one+iT)^{-1}(s_1-s_2) T_{V0}(s_2) \dd s_1 \dd s_2,
\end{aligned}\ee
where
$$
T_{V0}(t) = \bpm T_{10}(t) \\ T_{20}(t) \epm = \chi_{t>0} V_2 e^{it(-\Delta_{x_1}+\Delta_{x_2}+iA)}
$$
and
$$
T_{0V}(t) = \bpm T_{01}(t) & T_{02}(t) \epm = \chi_{t>0} e^{it(-\Delta_{x_1}+\Delta_{x_2}+iA)} V_1.
$$

The only nontrivial factor is $(\one+iT)^{-1}$, given by (\ref{Tinverse}). First consider the inverse of the diagonal component:
$$
(\one+iT_{diag})^{-1} = \bpm (\one + i T_{11})^{-1} & 0 \\ 0 & (\one + iT_{22})^{-1} \epm.
$$
The inverse of the first component has the explicit form
$$\begin{aligned}
(\one+iT_{11})^{-1} &= [\one + i \chi_{t \geq 0} v_2(x_1, y) e^{it(-\Delta_{x_1} + \Delta_{x_2}+ iA)} v_1(x_1, y)]^{-1} \\
&= \one - i \chi_{t \geq 0} v_2(x_1, y) [e^{it(-\Delta_{x_1}+ iA+V \otimes 1)} \otimes e^{it\Delta_{x_2}}] v_1(x_1, y)
\end{aligned}$$
and likewise for $(\one+iT_{22})^{-1}$, by symmetry.

Using the notation (\ref{syx}), by estimate (\ref{kernel_estimate}) from Proposition \ref{prop26}, it follows that
$$
(\one+iT_{11})^{-1}-\one \in \langle t \rangle^{-\3/2} L^\infty_t (L^\infty_{y_1} L^1_{y_2} \cap L^\infty_{y_2} L^1_{y_1}) [\B(L^2) \otimes \C e^{it\Delta}].
$$
Consequently,
$$
(\one+iT_{11})^{-1} U_1 \ov {\frak W}_\beta \subset U_1 \ov {\frak W}_\beta,\ \ov {\frak W}_\beta U_2 (\one+iT_{11})^{-1} \subset \ov {\frak W}_\beta U_2,
$$
where the free Schr\"{o}dinger evolution in the $x_2$ variable just gets appended to $U_1$ or $U_2$, and likewise, see (\ref{U1tilde}),
$$
(\one+iT_{22})^{-1} \tilde U_1 \ov {\frak W}_\beta \subset \tilde U_1 \ov {\frak W}_\beta,\ \ov {\frak W}_\beta \tilde U_2 (\one+iT_{22})^{-1} \subset \ov {\frak W}_\beta \tilde U_2.
$$

The expression $S$ consists of two off-diagonal entries, $S_1$ and $S_2$. The first, $S_1$, is of the same type as $T_{12}$ and belongs, among other operator spaces, to
$$
S_1(t) \in \langle t \rangle^{-\3} (L^\infty_{y_2} L^1_{y_1} \cap L^\infty_{y_1} L^1_{y_2}) [\B(L^1_{x_1} \cap L^2_{x_1}, L^2_{x_1}) \otimes \B(L^2_{x_2}, L^2_{x_2}+L^\infty_{x_2})].
$$
More precisely, $S_1$ is of the form $S_1=U_1 B \tilde U_2$, where $B \in \ov \W$ and $\tilde U_i$ is $U_i$ with the variables $x_1$ and $x_2$ swapped:
\be\lb{U1tilde}
[\tilde U_1 F](y_1, x_1, x_2) = \int_0^\infty \int_Y e^{-is_1\Delta_{x_1}} v_1(x_1, y)  F(y_1, y, x_1, s_1, x_2) \dd y \dd s_1
\ee
and
$$
[\tilde U_2 f](s_2, y_2, \tilde y, x_1, x_2) = v_2(x_1, \tilde y) e^{-is_2\Delta_{x_1}} f(y_2, x_1, x_2).
$$
Thus $S_1$ has a final segment in the $x_1$ variable and an initial segment in the $x_2$ variable, swapping these two variables.

The second entry $S_2$ behaves like a mirror image of $S_1$, being of the same type as $T_{21}$, and has the form $S_2=\tilde U_1 B U_2$, $B \in \ov \W$.

Just like for $U_1$ and $U_2$, see (\ref{u2u1bound}), the composition $\tilde U_2 \tilde U_1$ has better decay properties than the individual factors.


In conclusion, a typical term coming from $(\one + iT)^{-1}-\one$ in the Duhamel formula (\ref{duhamel_l}) has the same structure as $S_1 \ast S_2$ or $S_2 \ast S_1$, if on the diagonal; as $S_1$, in the upper-right corner; and as $S_2$, in the lower-left corner. In brief,
$$
(\one + iT)^{-1}-\one \in \begin{pmatrix} U_1 \\ \tilde U_1 \end{pmatrix} \ov {\frak W}_\beta \begin{pmatrix} U_2 & \tilde U_2 \end{pmatrix}
$$

Considering this and the explicit form of the other two factors in (\ref{duhamel_l}), $T_{V0}$ and $T_{0V}$ (which have dissipative components), we obtain the following expression for the whole Duhamel term:
\be\begin{aligned}\lb{complicated}
&\Big[\int_{\substack{s_1, s_2 \geq 0 \\ s_1+s_2 \leq t}} T_{0V}(s_1) [(\one+iT)^{-1}-\one](t-s_1-s_2) T_{V0}(s_2) \dd s_1 \dd s_2 \Big] F(\tilde y_1, x_1, x_2) = \\
&=\int_{\substack{s_1, s_2 \geq 0 \\ s_1+s_2 \leq t}} \int_{\sigma_1, \sigma_2 \geq 0} \int_{Y^5} e^{-s_1A}(\tilde y_1, y_1) e^{-is_1\Delta_{x_1}+i(s_1+\sigma_1) \Delta_{x_2}} v_1(x_1, y_1) v_1(x_2, y)
\frak B_{t - s_1 - s_2}(\sigma_1, \sigma_2, y_1, y, \tilde y, y_2) \\
&v_2(x_2, \tilde y) v_2(x_1, y_2) e^{-is_2\Delta_{x_1}+i(s_2+\sigma_2) \Delta_{x_2}} e^{-s_2A}(y_2, \tilde y_2) F(\tilde y_2, x_1, x_2) \dd y_1 \dd y \dd \tilde y \dd y_2 \dd \tilde y_2 \dd s_1 \dd \sigma_1 \dd s_2 \dd \sigma_2,
\end{aligned}\ee
where $\frak B \in \ov {\frak W}_\beta$, plus three other analogous terms coming from the other matrix entries.

We focus on estimating the main term. For each choice of $s$'s and $y$'s, the expression is bounded in the correct norm in the $x$ variables:
$$\begin{aligned}
&\| (\ref{complicated}) F\|_{(L^2 + L^\infty) \otimes (L^2 + L^\infty)} \les \\
&\les \langle t-s_1-s_2 \rangle^{-\3} \langle s_1\rangle^{-\3/2} \langle s_1 + \sigma_1 \rangle^{-\3/2} \langle s_2 \rangle^{-\3/2} \langle s_2 + \sigma_2 \rangle^{-\3/2} \|V\|_{C_y (L^1_x \cap L^\infty_x)}^2 \|\frak B_{t-s_1-s_2}\| \|F\|_{(L^1 \cap L^2) \otimes (L^1 \cap L^2)},
\end{aligned}$$

For each choice of $s$'s, the $y$ integrals are appropriately bounded because
$$
e^{-s_1 A}(\tilde y_1, y_1) \in L^\infty_{y_1} L^1_{\tilde y_1},\ \frak B_{t-s_1-s_2} \in L^\infty_{y_2} L^1_{\tilde y} L^1_{y_1} L^1_y,\ e^{-s_2 A}(y_2, \tilde y_2) \in L^\infty_{\tilde y_2} L^1_{y_2},\ F \in L^1_{\tilde y_2},
$$
and the integral
$$
\int_{Y^5} e^{-s_1 A} \frak B_{t-s_1-s_2} e^{-s_2 A} F \dd y_1 \dd y \dd \tilde y \dd y_2 \dd \tilde y_2
$$
is bounded by the norms above.



Finally, we are left with computing the following integral in $s$:
\be\lb{final_expression}
\int \frak B_{t-s_1-s_2} \langle s_1\rangle^{-\3/2} \langle s_1 + \sigma_1 \rangle^{-\3/2} \langle s_2 \rangle^{-\3/2} \langle s_2 + \sigma_2 \rangle^{-\3/2} \|V\|_{C_y (L^1_x \cap L^\infty_x)}^2 \dd s_1 \dd s_2 \dd \sigma_1 \dd \sigma_2.
\ee

Since $\frak B_{t-s_1-s_2} \in {\frak X_{\3/2, \beta}}$, see Definition \ref{defxtilde}, for each value of $t-s_1-s_2$ we can write it as
$$
\frak B_{t-s_1-s_2} = \int_0^\infty B^{t-s_1-s_2}_\tau \dd\tau,\ B^{t-s_1-s_2}_\tau \in \mc X_\tau,
$$
where by the Definition \ref{defxtilde} of ${\frak X}_{\3/2, \beta}$
\be\lb{estb}
\int_0^\infty (1+\beta\tau)^{-\3/2} \|B^{t-s_1-s_2}_\tau\|_{\mc X_\tau} \dd \tau \les \|\frak B_{t-s_1-s_2}\|_{{\frak X}} \les \langle t-s_1-s_2 \rangle^{-\3/2} (1+\beta(t-s_1-s_2))^{-\3/2} \|\frak B\|_{\frak W_\beta}
\ee
and from Definition \ref{defx} of $\mc X_\tau$, since $B^{t-s_1-s_2}_\tau \in \mc X_\tau$, we get an extra $\langle \tau-\sigma_1 - \sigma_2\rangle^{-\3/2}$ factor.

Consequently, replacing $\beta$ by $1/\beta$, for all $\beta \geq 1$
$$
\int_0^\infty (\beta + \tau)^{-\3/2} \|B^{t-s_1-s_2}_\tau\|_{\mc X_\tau} \dd \tau \les \|\frak B_{t-s_1-s_2}\|_{{\frak X}_\beta} \les \langle t-s_1-s_2 \rangle^{-\3/2} (\beta + t-s_1-s_2)^{-\3/2} \|\frak B\|_{{\frak W}_{1/\beta}},
$$
where in fact $\|\frak B\|_{\frak W_{1/\beta}}$ is uniformly bounded independently of $\beta \geq 1$.

Setting $\beta = s_1 + s_2 + 1$, this leads to
\be\lb{final_aux}
\int_0^\infty \langle \tau + s_1 + s_2 \rangle^{-\3/2} \|B^{t-s_1-s_2}_\tau\|_{\mc X_\tau} \dd \tau \les \langle t-s_1-s_2 \rangle^{-\3/2} \langle t \rangle^{-\3/2} \|\frak B\|_{{\frak W}_{1/\beta}}.
\ee

Returning to the integral (\ref{final_expression}), after substituting $\frak B_{t-s_1-s_2} \in \frak X_{\3/2, \beta}$ by a combination of $B^{t-s_1-s_2}_\tau \in \mc X_\tau$ in accordance with Definition \ref{defxtilde}, it is of size
\be\lb{sint}
\int \|B_\tau^{t-s_1-s_2}\|_{\mc X_\tau} \langle \tau-\sigma_1-\sigma_2 \rangle^{-\3/2} \langle s_1 \rangle^{-\3/2} \langle s_2 \rangle^{-\3/2} \langle s_1+\sigma_1 \rangle^{-\3/2} \langle s_2+\sigma_2 \rangle^{-\3/2} \dd \sigma_1 \dd \sigma_2 \dd s_1 \dd s_2 \dd \tau.
\ee
Here
$$
\int \langle \tau-\sigma_1-\sigma_2 \rangle^{-\3/2} \langle s_1+\sigma_1 \rangle^{-\3/2} \langle s_2+\sigma_2 \rangle^{-\3/2} \dd \sigma_1 \dd \sigma_2 \les \langle \tau + s_1 + s_2 \rangle^{-\3/2}.
$$
Integrating in $\tau$ and using inequality (\ref{final_aux}), this is no more than
$$
(\ref{sint}) \les \int \langle t-s_1-s_2 \rangle^{-\3/2} \langle t \rangle^{-\3/2} \langle s_1 \rangle^{-\3/2} \langle s_2 \rangle^{-\3/2} \dd s_1 \dd s_2 \les \langle t \rangle^{-\3}.
$$

The more singular terms supported at $\delta_{t-s_1-s_2-\sigma_1}(\sigma_2)$, which are the contribution of $\mc X_\tau \setminus X_\tau$, see Definition (\ref{defx}), obey the same bound.


The same computation works with initial and/or final segments of uneven length. In particular, if composed with $e^{-it_1\Delta_{x_1}} e^{it_2\Delta_{x_2}}$ on the right, then (\ref{complicated}) will be bounded in norm by
$$
\int \|B_\tau^{t-s_1-s_2}\|_{\mc X_\tau} \langle \tau-\sigma_1-\sigma_2 \rangle^{-\3/2} \langle s_1 \rangle^{-\3/2} \langle s_2 + t_1 \rangle^{-\3/2} \langle s_1+\sigma_1 \rangle^{-\3/2} \langle s_2+\sigma_2 + t_2 \rangle^{-\3/2} \dd \sigma_1 \dd \sigma_2 \dd s_1 \dd s_2 \dd \tau,
$$
which in turn is now bounded by
$$
\les \langle t+t_1 \rangle^{-\3/2} \langle t+t_2 \rangle^{-\3/2}.
$$
This implies (\ref{decay_est'}) and (\ref{decay_est''}) is proved similarly.
\end{proof}

\begin{proof}[Proof of Corollary \ref{cor_strichartz}] The proof uses Duhamel's formula (\ref{duhamel_v}) in combination with the pointwise decay estimate (\ref{decay_est''}).

Recall formula (\ref{complicated}):
$$\begin{aligned}
&\int_{\substack{s_1, s_2 \geq 0 \\ s_1+s_2 \leq t}} \int_{\sigma_1, \sigma_2 \geq 0} \int_{Y^5} e^{-s_1A}(\tilde y_1, y_1) e^{-is_1\Delta_{x_1}+i(s_1+\sigma_1) \Delta_{x_2}} v_1(x_1, y_1) v_1(x_2, y)
\frak B_{t - s_1 - s_2}(\sigma_1, \sigma_2, y_1, y, \tilde y, y_2) \\
&v_2(x_2, \tilde y) v_2(x_1, y_2) e^{-is_2\Delta_{x_1}+i(s_2+\sigma_2) \Delta_{x_2}} e^{-s_2A}(y_2, \tilde y_2) F(\tilde y_2, x_1, x_2) \dd y_1 \dd y \dd \tilde y \dd y_2 \dd \tilde y_2 \dd s_1 \dd \sigma_1 \dd s_2 \dd \sigma_2.
\end{aligned}$$

The initial and final segments can be bounded using endpoint Strichartz estimates: for each $y$ configuration,
$$
\|v_2(x_2, \tilde y) v_2(x_1, y_2) e^{-is_2\Delta_{x_1}+i(s_2+\sigma_2) \Delta_{x_2}} F(x_1, x_2)\|_{L^2_{s_2, x_1} \otimes L^2_{\sigma_2, x_2}} \les \|F\|_{\frak S_1}.
$$
In order to use it below, we weaken the $L^2_{s_2, x_1} \otimes L^2_{\sigma_2, x_2}$ norm as follows: for each $s_2$ and $\sigma_2$, $L^2_{s_2, x_1} \times L^2_{\sigma_2, x_2} \subset \frak S_1$ in just the space variables, and
each individual tensor has a bound of the form
\be\lb{tensorp}
\|f_1(s_2, x_1) f_2(\sigma_2, x_2)\|_{\frak S_1} \leq f(s_2) g(\sigma_2),\ f, g \in L^2.
\ee

Strichartz estimates do not provide uniform decay in $s_1$ and $s_2$, so we need to put the $t$ norms on the inside and estimate them before the $y$ norms. This is possible because, by Minkowski's inequality,
$$
X_t = \chi_{s_1, s_2 \geq 0} \langle t-(s_1 + s_2) \rangle^{-\3/2} L^\infty_{s_1, s_2} [L^\infty_{y_2} L^1_{\tilde y, y, y_2}] \B(\frak S_1) \subset [L^\infty_{y_2} L^1_{\tilde y, y, y_2}] \chi_{s_1, s_2 \geq 0} \langle t-(s_1 + s_2) \rangle^{-\3/2} L^\infty_{s,_1, s_2} \B(\frak S_1).
$$
Consequently,
$$
\frak B_{t-s_1-s_2}(\sigma_1, \sigma_2, y_1, y, \tilde y, y_2) = \int_0^\infty B^{t-s_1-s_2}_\tau \dd \tau
$$
and
$$
B^{t-s_1-s_2}_\tau \in [L^\infty_{y_2} L^1_{\tilde y, y, y_2}] \chi_{s_1, s_2 \geq 0} \langle \tau-(s_1 + s_2) \rangle^{-\3/2} L^\infty_{s,_1, s_2} \B(\frak S_1).
$$
Estimating $\frak B$ in $\frak W_0$, see Definition \ref{defxtilde}, the overall norm of $\frak B_{t-s_1-s_2}$ is of size $\langle t-s_1-s_2 \rangle^{-\3/2}$.

For each $t$ and $\tau$, integration against $\langle t - (s_1+s_2) \rangle^{-\3/2} \langle \tau - (\sigma_1 + \sigma_2) \rangle^{-\3/2}$ is bounded from $L^2_{s_2} \otimes L^2_{\sigma_2}$ to $L^2_{s_1} \otimes L^2_{\sigma_1}$.

Since the kernel $\frak B \in \frak W_0$ is integrable over $\tau$, it acts in the same manner as each individual expression. Thus, for each fixed $y$ configuration, the image of each tensor (\ref{tensorp}) under $\frak B$ is still a combination of $h(s_2, \sigma_2, x_1, x_2)$, where
$$
\|h(s_2, \sigma_2, x_1, x_2)\|_{\frak S_1} \leq \tilde f(s_2) \tilde g(\sigma_2),\ \tilde f, \tilde g \in L^2.
$$

Finally, to estimate the last factor, for each $y$ configuration we need to analyze the following expression:
\be\lb{bilin}
\int_{t > s_1 > s_2} w_1(x_1, \tilde y_1) w_2(x_2, \tilde y_1) e^{-i(t-s_2)\Delta_{x_1}} e^{i(t-\sigma_2)\Delta_{x_2}} v_1(x_1, y_1) v_1(x_2, y) h(s_2, \sigma_2, x_1, x_2) \dd s_1 \dd s_2.
\ee

By Lemma \ref{uniform_lemma}, the $L^2$ operator norms of $w_1(x_1) e^{-i(t-s_2)\Delta_{x_1}} v_1(x_1, y_1)$ and $w_2(x_2) e^{i(t-\sigma_2)\Delta_{x_2}} v_1(x_2, y)$ are $L^1$ functions of $t-s_2$, respectively $t-\sigma_2$; denote them $k_1(t-s_2)$ and $k_2(t-\sigma_2)$.

Then, for each $s_2$, $\sigma_2$, and $t$, the expression (\ref{bilin}) is bounded in the trace class by
$$
k_1(t-s_2) k_2(t-\sigma_2) \tilde f(s_2) \tilde g(\sigma_2).
$$
Since $k_1$, $k_2 \in L^1$ and $\tilde f, \tilde g \in L^2$, this expression is integrable in $s_2$, $\sigma_2$, and $t$. Summing over all the tensors in the tensor product, and then integrating in $y$, we obtain the desired conclusion (\ref{strichartz}).

The other estimate (\ref{strichartz'}) follows in the same manner.

\end{proof}

\subsection{Proofs of the main results}
Finally, we are ready to give the proof of the main linear result, concerning dispersive estimates for a Schr\"{o}dinger equation with time-dependent potential.
\begin{proposition}\lb{main_linear} Let $\psi_\omega$ be a solution of (\ref{random_lin}). Supposing that $V \in C_y (L^1_x \cap L^\infty_x)$ satisfies Assumption \ref{nontrivial},
\be\lb{random_decay}
\|\psi_\omega\|_{\langle t \rangle^{-\3/2} L^\infty_t L^2_\omega (L^2_x + L^\infty_x)} \les \|\psi_0\|_{L^2_\omega (L^1_x \cap L^2_x)} + \|\Psi_\omega\|_{\langle t \rangle^{-\3/2} L^\infty_t L^2_\omega (L^1_x \cap L^2_x)}
\ee
and
\be\lb{random_strichartz}
\|\psi_\omega\|_{L^2_t L^2_\omega L^{6, 2}_x} \les \|\psi_0\|_{L^2_\omega L^2_x} + \|\Psi_\omega\|_{L^2_t L^2_\omega L^{6/5, 2}_x}.
\ee
\end{proposition}
\begin{proof} [Proof of Theorem \ref{main_result}]
Recall that $\psi_\omega$ is the solution of the equation
$$
i\partial_t \psi_{\omega} - \Delta \psi_{\omega} + V \psi_{\omega} = \Psi_{\omega},\ \psi_\omega(0, x) = \psi_0(x, \omega),\ \omega(0) = X_0.
$$
Here $X_0: \Omega \to Y$ is a random variable that gives the initial distribution of the Markov process driving the random potential. Its distribution $\mu_0(A) = \set P(\omega_0 \in A)$ is a probability measure on $Y$.

Define
$$
f(x_1, x_2, y, t)=\E(\psi_\omega(x_1, t) \ov\psi_\omega(x_2, t): \omega(t) = X_t = y),
$$
which fulfills in turn the averaged Liouville-type equation (\ref{average_liouville})
$$
i \partial_t f - \Delta_{x_1} f + \Delta_{x_2} f + i A f + (V \otimes 1 - 1 \otimes V) f = F,
$$
where
$$
f(0, y, x_1, x_2) = f_0(y, x_1, x_2) = \E(\psi_0(x_1, \omega) \ov{\psi_0(x_2, \omega)} : \omega(0) = \omega_0 = y)
$$
and
$$
F(t, y, x_1, x_2)=\E(\Psi_\omega(x_1) \ov \psi_\omega(x_2) - \psi_\omega(x_1) \otimes \ov \Psi_\omega(x_2): \omega(t)=X_t = y).
$$

Because the equation is linear, we can treat the contributions of the initial data $\psi_{\omega 0}$ and of the source term $\Psi_\omega$ separately.

First, consider the case of uniform decay estimates, for $\psi_{0} \in L^1 \cap L^\infty$.

The random initial data of the random equation have the following contribution in (\ref{average_liouville}) after averaging. If $\psi_0(\omega) \in L^\infty_\omega (L^1_x \cap L^2_x)$ and $\omega_0 \in \mathcal M_\omega$, then
$$\begin{aligned}
\|f_0\|_{L^1_y (L^1_{x_1} \cap L^2_{x_1}) \otimes (L^1_{x_2} \cap L^2_{x_2})} &= \E \|\psi_0(x_1, \omega) \ov{\psi_0(x_2, \omega)}\|_{(L^1_{x_1} \cap L^2_{x_1}) \otimes (L^1_{x_2} \cap L^2_{x_2})}\\
&= \|\psi_0(x_1, \omega) \ov{\psi_0(x_2, \omega)}\|_{L^1\omega (L^1_{x_1} \cap L^2_{x_1}) \otimes (L^1_{x_2} \cap L^2_{x_2})} \\
&= \|\psi_0(x, \omega)\|_{L^2_\omega(\mu_0) (L^1_x \cap L^2_x)}^2.
\end{aligned}$$


By Pillet's Feynman--Kac formula (\ref{feynman}), for each $t \geq 0$
\be\lb{feynman_applied}\begin{aligned}
\||V_\omega(x, t)|^{1/2} \psi_\omega(x, t)\|_{L^2_\omega L^2_x}^2 &= \E \int |V_\omega(x, t)| |\psi_\omega(x, t)|^2 \dd x \\
&= \int_{Y \times \R^\3} |V(x, y)| f(x, x, y, t) \dd x \dd y \\
&\leq \|V\|_{L^\infty_y (L^1_x \cap L^\infty_x)} \|f(t)\|_{L^1_y  (L^2_{x_1} + L^\infty_{x_1}) \otimes (L^2_{x_2} + L^\infty_{x_2})}.
\end{aligned}\ee
By the dispersive estimate (\ref{point}), in the absence of a source term,
$$
\|f(t)\|_{L^1_y (L^2_{x_1} + L^\infty_{x_1}) \otimes (L^2_{x_2} + L^\infty_{x_2})} \les \langle t \rangle^{-\3} \|f_0\|_{L^1_y (L^1_{x_1} \cap L^2_{x_1}) \otimes (L^1_{x_2} \cap L^2_{x_2})}.
$$

But for each path $\omega \in \Omega$ the left-hand side of (\ref{feynman_applied}) controls the pointwise decay:
$$\begin{aligned}
\|\psi_\omega(t)\|_{L^2_x + L^\infty_x} &\les \langle t \rangle^{-\3/2} \|\psi_0(x, \omega)\|_{L^1_x \cap L^2_x} + \int_0^t \langle t-s \rangle^{-\3/2} \|V_\omega(x, s) \psi_\omega(s)\|_{L^1_x \cap L^2_x} \dd s \\
& \les \langle t \rangle^{-\3/2} \|\psi_0(x, \omega)\|_{L^1_x \cap L^2_x} + \int_0^t \langle t-s \rangle^{-\3/2} \|V\|_{L^\infty_y (L^1_x \cap L^\infty_x)}^{1/2} \||V_\omega(x, s)|^{1/2} \psi_\omega(s)\|_{L^2_x} \dd s.
\end{aligned}$$
Indeed, by Minkowski's inequality
$$\begin{aligned}
\bigg\|\int_0^t \langle t-s \rangle^{-\3/2} \||V_\omega(x, s)|^{1/2} \psi_\omega(s)\|_{L^2_x} \dd s\bigg\|_{L^2_\omega} &\leq \int_0^t \langle t-s \rangle^{-\3/2} \||V_\omega(x, s)|^{1/2} \psi_\omega(s)\|_{L^2_{\omega, x}} \dd s \\
&\les \bigg(\int_0^t \langle t-s \rangle^{-\3/2} \langle s \rangle^{-\3/2} \dd s\bigg) \|V\|_{L^\infty_y (L^1_x \cap L^\infty_x)}^{1/2} \|f_0\|_{L^1_y (L^1_{x_1} \cap L^2_{x_1}) \otimes (L^1_{x_2} \cap L^2_{x_2})}^{1/2} \\
&\les \langle t \rangle^{-\3/2} \|V\|_{L^\infty_y (L^1_x \cap L^\infty_x)}^{1/2} \|\psi_0(x, \omega)\|_{L^2_\omega (L^1_x \cap L^2_x)}.
\end{aligned}$$
We conclude that
$$
\|\psi_\omega(t)\|_{L^2_\omega (L^2_x + L^\infty_x)} \les \langle t \rangle^{-\3/2} \|\psi_0(x, \omega)\|_{L^2_\omega (L^1_x \cap L^2_x)}.
$$


Next, consider the inhomogenous term $F$ in (\ref{average_liouville}), which is the conditional expectation of the following tensor product:
\be\lb{inhomogenous_term}\begin{aligned}
F(t) = \E\big\{\Psi_{\omega}(t) \otimes \ov \psi_{\omega}(t) - \psi_{\omega} \otimes \ov \Psi_{\omega} : \omega(t) = y\big\}.
\end{aligned}\ee
This is the sum of two terms, $F(t) = F_1(t) + F_2(t)$, where
$$
F_1(t) = \E\big\{\Psi_{\omega}(t) \otimes \ov \psi_{\omega}(t) : \omega(t) = y\big\},\ F_2(t) = -\E\big\{\psi_{\omega}(t) \otimes \ov \Psi_{\omega}(t) : \omega(t) = y\big\}.
$$

Here $\psi$ only belongs to Strichartz spaces, never to dual Strichartz spaces, so we cannot use estimates such as (\ref{point}). The inhomogenous term has to be handled in a special manner, by means of (\ref{decay_est'}).

We write $\psi_{\omega}$ in (\ref{inhomogenous_term}) using Duhamel's formula:
$$
\psi_{\omega} = i \int_0^t e^{-i(t-s)\Delta} (V_\omega(s) \psi_{\omega} + \Psi_{\omega}) \dd s.
$$
Note that
$$\begin{aligned}
\|V_\omega(s) \psi_{\omega}(s) + \Psi_{\omega}(s)\|_{L^2_{\omega} (L^1_x \cap L^2_x)} \les \|\psi_{\omega}(s)\|_{L^2_\omega (L^2_x + L^\infty_x)} + \|\Psi_{\omega}(s)\|_{L^2_\omega (L^1_x \cap L^2_x)}.
\end{aligned}$$
By performing this expansion within the averaged Liouville-type equation (\ref{average_liouville}) for $f$, we get that the inhomogenous term $F=\E((\ref{inhomogenous_term}):X_t=y)$ consists of
$$
F_2(t, y_1) = \int_0^t e^{i(t-s)\Delta_{x_2}} G_2(t, s, y_1, x_1, x_2) \dd y \dd s,
$$
where
$$
G_2(t, s) = i\E([V_\omega(s) \psi_{\omega}(s) + \Psi_\omega(s)] \otimes \ov {\Psi_{\omega}}(t) \mid \omega(t) = y_1).
$$


Due to properties of Markov chains, the expression will be in $L^1_{y_1}$:
$$\begin{aligned}
\|G_2(t, s)\|_{L^1_{y_1} (L^1_{x_1} \cap L^2_{x_1}) \otimes (L^1_{x_2} \cap L^2_{x_2})} &\les \E \|[V_\omega(s) \psi_{\omega}(s) + \Psi_\omega(s)] \otimes \ov {\Psi_{\omega}}(t)\|_{(L^1_{x_1} \cap L^2_{x_1}) \otimes (L^1_{x_2} \cap L^2_{x_2})} \\
&\les \|V_\omega(s) \psi_{\omega}(s) + \Psi_\omega(s)\|_{L^2_\omega (L^1_x \cap L^2_x)} \|\Psi_\omega(t)\|_{L^2_\omega (L^1_x \cap L^2_x)} \\
&\les (\|\psi_{\omega}(s)\|_{L^2_\omega (L^2_x + L^\infty_x)} + \|\Psi_{\omega}(s)\|_{L^2_\omega (L^1_x \cap L^2_x)}) \|\Psi_\omega(t)\|_{L^2_\omega (L^1_x \cap L^2_x)}.
\end{aligned}
$$
$F_1$ can be treated in the same way, swapping $x_1$ and $x_2$.

By (\ref{decay_est'}), then,
$$\begin{aligned}
\|\psi_\omega(t)\|_{L^2_\omega (L^2_x+L^\infty_x)}^2 \les \iint_{0 \leq s_1, s_2 \leq t} &\langle t-s_1 \rangle^{-\3/2} \langle t-s_2 \rangle^{-\3/2} \|\Psi_\omega(s_1)\|_{L^2_\omega (L^1_x \cap L^2_x)} \\
& (\|\Psi_\omega(s_2)\|_{L^2_\omega (L^1_x \cap L^2_x)} + \|\psi_\omega(s_2)\|_{L^2_\omega (L^1_x \cap L^2_x)}) \dd s_1 \dd s_2.
\end{aligned}$$
Then
$$
\|\psi_\omega\|_{\langle t \rangle^{-\3/2} L^\infty_t L^2_\omega (L^2_x+L^\infty_x)}^2 \les \|\Psi_\omega\|_{\langle s \rangle^{-\3/2} L^\infty_s L^2_\omega (L^1_x \cap L^2_x)}^2 + \|\Psi_\omega\|_{\langle s \rangle^{-\3/2} L^\infty_s L^2_\omega (L^1_x \cap L^2_x)} \|\psi_\omega\|_{\langle s \rangle^{-\3/2} L^\infty_s L^2_\omega (L^2_x+L^\infty_x)}
$$
and the conclusion (\ref{random_decay}) follows.

For the Strichartz estimate (\ref{random_strichartz}), the computation is similar, but with powers of decay replaced by $L^p$ bounds, employing estimate (\ref{strichartz'}). By the Feynman--Kac formula (\ref{feynman}),
$$\begin{aligned}
\||V_\omega(x, t)|^{1/2} \psi_\omega(x, t)\|_{L^2_t L^2_\omega L^2_x}^2 &= \E \iint |V_\omega(x, t)| |\psi_\omega(x, t)|^2 \dd x \dd t\\
&= \iint_{\R \times Y \times \R^\3} |V(x, y)| f(x, x, y, t) \dd x \dd y \dd t\\
&\leq \||V|^{1/2} f(t) |V|^{1/2}\|_{L^1_y L^1_t \frak S_1}.
\end{aligned}$$

By virtue of Corollary \ref{cor_strichartz}, the computation now reads
$$\begin{aligned}
\|\psi_\omega(t)\|_{L^2_t L^2_\omega L^{6, 2}_x}^2 &\les \|G_2\|_{L^1_y [L^2_t L^{6/5, 2}_{x_1} \otimes L^2_s L^{6/5, 2}_{x_2}]} \\
&\les \|\Psi_\omega\|_{L^2_s L^2_\omega L^{6/5, 2}_x} (\|\Psi_\omega\|_{L^2_s L^2_\omega L^{6/5, 2}_x} + \|\psi_\omega\|_{L^2_s L^2_\omega L^{6, 2}_x}),
\end{aligned}$$
implying the conclusion (\ref{random_strichartz}).

Concerning the boundedness of the average energy, note that by the definition (\ref{def_avg}) and by the same Feynman--Kac formula (\ref{feynman}) it follows that for each $t \geq 0$ and $h:Y \to \C$
$$
\E\big\{h(\omega(t))\|\dl\psi_\omega(t)\|_{L^2_x}^2 + \langle \psi_\omega(t), V_\omega(t) \psi_\omega(t) \rangle\big\} = \int_Y h(y) \tr [(-\Delta_{x_1}+V(x_1, y)) f(x_1, x_2, y, t)] \dd y.
$$
Taking a $t$ derivative and using the equation (\ref{average_liouville}) for $f$, we see that
\be\lb{deriv}\begin{aligned}
&\frac \partial {\partial t} (-\Delta_{x_1}+V(x_1, y)) f(x_1, x_2, y, t) = \\
&= (-\Delta_{x_1}+V(x_1, y))(i(-\Delta_{x_1}+V(x_1, y)+\Delta_{x_2}-V(x_2, y))-A)f(x_1, x_2, y, t).
\end{aligned}\ee

Consider a Hilbert space $L^2$ of square integrable functions over some measure space. For a trace-class operator $K \in \frak S_1(L^2)$ of kernel $K(x_1, x_2)$, its singular-value decomposition implies that
$$
\tr K = \int K(x, x) \dd x.
$$
Furthermore, given $f, g \in L^2$ and a possibly unbounded operator $T$,
$$
\tr [Tf \otimes \ov g] = \int (T f) \ov g = \int f (T^* \ov g) = \tr [f \otimes T^* \ov g].
$$
This equality holds as long as both sides are finite.

By using the singular-value decomposition, this also generalizes to any trace-class operator $K$, relating to a well-known identity:
$$
\tr [T K] = \tr [K T].
$$

Consequently, for each $t \geq 0$ and $y \in Y$, when $V$ is real-valued and $f$ is sufficiently smooth
$$\begin{aligned}
\tr [(-\Delta_{x_1}+V(x_1, y))^2 f(x_1, x_2, y, t)] &= \tr [(-\Delta_{x_1}+V(x_1, y))(-\Delta_{x_2}+V(x_2, y))f(x_1, x_2, y, t)] \\
&= \tr [(-\Delta_{x_2}+V(x_2, y))^2 f(x_1, x_2, y, t)].
\end{aligned}$$

So the two main terms in (\ref{deriv}) cancel and we are left with
\be\lb{derivative_identity}\begin{aligned}
\frac \partial {\partial t} \tr [(-\Delta_{x_1}+V(x_1, y)) f(x_1, x_2, y, t)] &= -\tr [(-\Delta_{x_1}+V(x_1, y)) A f(x_1, x_2, y, t)] \\
&= A \tr [\nabla_{x_1} \nabla_{x_2} f(x_1, x_2, y, t)] + \tr [V(x_1, y) A f(x_1, x_2, y, t)].
\end{aligned}\ee

This identity holds pointwise for each $y \in Y$.

Next, integrating in $y$ and using the self-adjointness of $A$, assuming that $A1=0$ ($A$ applied to the constant function is $0$), we obtain that
$$
-\int_Y (-\Delta_{x_1}+V(x_1, y)) A f(x_1, x_2, y, t) \dd y = \int_Y [A V(x_1, y)] f(x_1, x_2, y, t) \dd y
$$
and same after taking the trace, so
$$\begin{aligned}
\frac \partial {\partial t} \int_Y \tr [(-\Delta_{x_1}+V(x_1, y)) f(x_1, x_2, y, t)] &= \int_Y [A V(x_1, y)] \tr [f(x_1, x_2, y, t)] \dd y.
\end{aligned}$$

More generally, when integrating the identity (\ref{derivative_identity}) against a weight $h(y)$ instead of $1$
$$
-\int_Y h(y) (-\Delta_{x_1}+V(x_1, y)) A f(x_1, x_2, y, t) \dd y = \int_Y (-Ah) \Delta_{x_1} f + A [h(y) V(x_1, y)] f(x_1, x_2, y, t) \dd y.
$$
Taking the trace this leads to
\be\lb{energy}
\frac \partial {\partial t} \int_Y h(y) \tr [(-\Delta_{x_1}+V(x_1, y)) f(x_1, x_2, y, t)] \dd y = \int_Y (Ah) \tr(-\Delta_{x_1} f) + \tr (A [h V] f) \dd y.
\ee
The first term is
$$
\int_Y (Ah) \tr(-\Delta_{x_1} f) \dd y = \int_Y (Ah)(y) \E\big\{\|\nabla \psi_\omega(t)\|_{L^2_x}^2 \mid \omega(t) = y\big\} \dd y.
$$
If $h$ is a ground state of $A$, i.e.\;$A h = 0$, or more generally if $A h \leq 0$ for all $y$, then this term is non-positive, so for a ground state $h$
$$
\frac \partial {\partial t} \int_Y h(y) \tr [(-\Delta_{x_1}+V(x_1, y)) f(x_1, x_2, y, t)] \dd y = \int_Y \tr (A [h V] f) \dd y.
$$
Note that
$$
\int_Y \tr[A (h(y)V(x_1, y)) f(x_1, x_2, y, t)] dy \leq \|A (hV)\|_{L^\infty_y L^{\3/2, \infty}_x} \|f(t)\|_{L^1_y (L^{6, 2}_{x_1} \otimes L^{6, 2}_{x_2})}.
$$

So for $0 \leq t_1 \leq t_2$ and $Ah=0$
$$
\big|\E \big\{ h(\omega(t_2)) E[\psi_\omega](t_2) \big\} - \E \big\{ h(\omega(t_1)) E[\psi_\omega](t_1) \big\}\big| \les \|f\|_{L^1_t([t_1, t_2]) L^1_y (L^{6, 2}_{x_1} \otimes L^{6, 2}_{x_2})}.
$$

The boundedness of this weighted average of energy now follows from Corollary \ref{cor_strichartz}. Furthermore, this inequality implies that the average energy at time $t$ has a limit as $t$ goes to infinity.

More generally, if $Ah \leq 0$, one can show that the limit superior is the same as the limit inferior, so the same conclusion holds.

When $A(hV) \in L^\infty_y (L^1_x \cap L^\infty_x)$ and $\psi_0 \in L^2_\omega (L^1_x \cap L^2_x)$, Proposition \ref{prop27} implies that the average energy at time $t$ approaches its limit at a rate of $O(t^{-1/2})$.
\end{proof}

Finally, we restate and prove the main nonlinear result in terms of small initial data.
\begin{theorem}[Well-posedness for small initial data] Consider $\chi \in L^{\3/2, \infty}$ and a random potential of the form (\ref{def_v}). Then the equation
\be\lb{random_nonlin'}
i \partial_t \psi_\omega - \Delta \psi_\omega + V_\omega \psi_\omega + (\chi \ast |\psi_\omega|^2) \psi_\omega = 0,\ \psi_\omega(x, 0) = \psi_0(x, \omega)
\ee
is well-posed for small initial data. Namely, there exists some $\epsilon_0$ such that, if
$$
\|\psi_0\|_{L^\infty_\omega L^2_x} < \epsilon_0,
$$
then the solution $\psi_\omega$ exists with probability one and
$$
\|\psi_\omega\|_{L^\infty_\omega L^\infty_t L^2_x} < \epsilon_0,\ \E \|\psi_\omega\|_{L^2_t L^{6, 2}_ x}^2 \les \|\psi_0\|_{L^2_\omega L^2_x}^2 < \epsilon_0^2.
$$
Moreover, average energy is uniformly bounded if it is finite at $t=0$, same as in the linear case.
\end{theorem}
Strichartz bounds imply scattering and the almost sure existence of the wave operator $W_\omega = \lim_{t \to \infty} e^{it\Delta} \psi_\omega(t)$ for small initial data.

The solution map is locally Lipschitz continuous from $L^2_x$ to $L^2_\omega L^2_t L^{6, 2}_x$. We get local-in-time well-posedness for large initial data as well.
\begin{proof} The proof is conducted by means of a standard contraction scheme. Let $\psi_{\omega 0} \equiv 0$ and for $n \geq 1$ define iteratively $\psi_{\omega n}$ to be the solution to the equation
$$
i \partial_t \psi_{\omega n} - \Delta \psi_{\omega n} + V_\omega \psi_{\omega n} + (\chi \ast |\psi_{\omega n-1}|^2) \psi_{\omega n} = 0,\ \psi_{\omega n}(x, 0) = \psi_0(x, \omega).
$$
Assuming $\|\psi_{\omega n-1}\|_{L^\infty_{\omega, t} L^2_x} < \infty$, the solution $\psi_{\omega n}$ exists on the interval $[0, \infty)$ for almost every $\omega$. Due to the $L^2$ norm conservation, $\psi_{\omega n}$ a priori belongs to $L^\infty_t L^2_x$ with probability one and
$$
\|\psi_{\omega n}\|_{L^\infty_\omega L^\infty_t L^2_x} = \|\psi_{\omega n}(0)\|_{L^\infty_\omega L^2_x} \leq \epsilon.
$$

By our previous linear estimates from Proposition \ref{main_linear},
$$
\|\psi_{\omega n}\|_{L^2_\omega L^2_t L^{6, 2}_x} \les \|\psi_0\|_{L^2_\omega L^2_x} + \|\Psi_{\omega n}\|_{L^2_\omega L^2_t L^{6/5, 2}_x},
$$
where $\Psi_{\omega n} = -(\chi \ast |\psi_{\omega n-1}|^2) \psi_{\omega n}$ is adapted.

Then the norm of the source term $\Psi_{\omega n}$ is bounded by
$$
\|\Psi_{\omega n}\|_{L^2_\omega L^2_t L^{6/5, 2}_x} \les \|\psi_{\omega n-1}\|_{L^\infty_\omega L^\infty_t L^2_x}^2 \|\psi_{\omega n}\|_{L^2_\omega L^2_t L^{6, 2}_x} \les \epsilon_0^2 \|\psi_{\omega n}\|_{L^2_\omega L^2_t L^{6, 2}_x},
$$
where the computations work because
$$
[\chi \ast (L^2 L^2)] L^{6, 2} \subset L^{6/5, 2}.
$$
This is guaranteed by $\chi \in L^{\3/2, \infty}$, so we need inverse-square decay.

When $\epsilon_0$ is sufficiently small this implies by induction that
$$
\|\psi_{\omega n}\|_{L^2_\omega L^2_t L^{6, 2}_x} \les \|\psi_0\|_{L^2_\omega L^2_x}
$$
for all $n \geq 1$.


Likewise, consider the equation of the difference of $\psi_{\omega n}$ and $\psi_{\omega n-1}$:
$$
i\partial_t (\psi_{\omega n}-\psi_{\omega n-1}) - \Delta (\psi_{\omega n}-\psi_{\omega n-1}) + V (\psi_{\omega n}-\psi_{\omega n-1}) = \Psi_{\omega n}-\Psi_{\omega n-1},
$$
where
$$
\Psi_{\omega n}-\Psi_{\omega n-1} = -(\chi \ast |\psi_{\omega n-1}|^2) \psi_{\omega n} + (\chi \ast |\psi_{\omega n-2}|^2) \psi_{\omega n-1}.
$$
Pointwise,
$$\begin{aligned}
-(\chi \ast |\psi_{\omega n-1}|^2) \psi_{\omega n} + (\chi \ast |\psi_{\omega n-2}|^2) \psi_{\omega n-1} &\les (\chi \ast |\psi_{\omega n-1}|^2) |\psi_{\omega n} - \psi_{\omega n-1}| + \\
&+ (\chi \ast [(|\psi_{\omega n-1}| + |\psi_{\omega n-2}|) |\psi_{\omega n-1}-\psi_{\omega n-2}|]) \psi_{\omega n-1}.
\end{aligned}$$
Then we estimate each term and get that
$$
\|(\chi \ast |\psi_{\omega n-1}|^2) |\psi_{\omega n} - \psi_{\omega n-1}|\|_{L^2_\omega L^2_t L^{6/5, 2}_x} \les \epsilon_0^2 \|\psi_{\omega n} - \psi_{\omega n-1}\|_{L^2_\omega L^2_t L^{6, 2}_x}
$$
and
$$
\|(\chi \ast [(|\psi_{\omega n-1}| + |\psi_{\omega n-2}|) |\psi_{\omega n-1}-\psi_{\omega n-2}|]) \psi_{\omega n-1}|\|_{L^2_\omega L^2_t L^{6/5, 2}_x} \les \epsilon_0^2 \|\psi_{\omega n-1} - \psi_{\omega n-2}\|_{L^2_\omega L^2_t L^{6, 2}_x}
$$

The important fact is that
$$
[\chi \ast (L^2 L^2)] L^{6, 2} \subset L^{6/5, 2},\ [\chi \ast (L^2 L^{6, 2})] L^2 \subset L^{6/5, 2}.
$$
This is guaranteed by $\chi \in L^{\3/2, \infty}$, so we need inverse-square decay.

Consequently, when $\epsilon_0$ is sufficiently small, the mapping is a contraction. Differences $\psi_{\omega n} - \psi_{\omega n-1}$ are dominated by an exponentially decreasing sequence, so the sequence $\psi_{\omega n}$ converges in $L^2_\omega L^2_t L^{6, 2}_x$.

The limit is an ensemble of solutions $(\psi_\omega)_{\omega \in \Omega}$ with the desired properties.


Concerning the boundedness of the average energy, define the averaged density matrix
$$
f(x_1, x_2, y, t) = \E(\psi_\omega(x_1, t) \ov \psi_\omega(x_2, t) \mid \omega(t) = y).
$$
By a computation similar to (\ref{energy}) we obtain that
$$
\frac \partial {\partial t} \E E[\psi_\omega](t) = \int_Y \tr(A [h(y) V(x_1, y)] f(x_1, x_2, y, t)) \dd y,
$$
from which the boundedness of the average energy follows by Proposition \ref{prop28}.
\end{proof}

\section*{Acknowledgments}
M.B.\;is partially supported by the NSF grants DMS-1128155 and DMS-1700293, an AMS--Simons Foundation Travel Grant, and a Simons Foundation Collaboration Grant. In addition, M.B.\;gratefully acknowledges the support of UC Berkeley, the University of Chicago, and the University at Albany SUNY, as well as useful conversations with Wilhelm Schlag, Viorel Barbu, and Rupert Frank.


A.S.\;is partially supported by the NSF grants DMS-1201394 and DMS-1600749.

The initial stages of this work were accomplished and its goal and its scope were decided in collaboration with J\"{u}rg Fr\"{o}hlich. Both authors are grateful for his substantial contribution and wish to thank him for his continued support, comments, and contributions during the writing of this paper.


\end{document}